\documentclass[11pt]{article} %TBC

\usepackage[hyphens]{url} 

\usepackage{amsmath}
\usepackage{amssymb}
\usepackage{amsthm}
\usepackage{mathtools}

\usepackage{nicefrac}

\usepackage{booktabs}
\usepackage{tabularx,makecell,rotating,lscape}
\usepackage{longtable}
\usepackage{xltabular}
\usepackage{dcolumn}
\usepackage{multirow}
\usepackage{adjustbox}

\usepackage[backend=bibtex,maxcitenames=2,style=numeric,sortcites=false]{biblatex}
\usepackage{csquotes}

\usepackage[british]{babel}

\usepackage{pdflscape}

\usepackage[defaultlines=4,all]{nowidow}

\usepackage[ruled,vlined]{algorithm2e}

\usepackage{enumitem}
\usepackage{siunitx}
\usepackage{float}
\usepackage{color}
\usepackage{changebar}
\usepackage{acronym} 
\usepackage{geometry}
\usepackage{graphicx}
\usepackage[dvipsnames]{xcolor}
\usepackage{multicol}
\usepackage{savesym}
\usepackage{tikz-cd}
\usetikzlibrary{babel}
\usepackage{moresize}
\usepackage{changepage}
\usepackage{wrapfig}
\usepackage{adjustbox}
\usepackage{csquotes}
\usepackage{mathtools}
\usepackage{comment}
\savesymbol{bigtimes}
\savesymbol{widebar}
\usepackage{amsmath,amssymb,amsthm, mathabx,mathrsfs}
\usepackage{hyperref}
\restoresymbol{tx}{bigtimes}
\restoresymbol{tx}{widebar}

\newcommand{\MFN}{\widetilde{\textrm{FN}}}

\newcommand{\FN}{\textrm{FN}}

\newcommand{\FNP}{\textrm{FN}^{\le}}
\newcommand{\Ff}{   \textrm{sd } \mathcal{F}}

\newcommand{\Conf}{\textrm{Conf}}
\newcommand{\Emb}{\textrm{Emb}}
\DeclareMathOperator{\im}{Im}

\newcommand{\Kons}{\textrm{Konts}}
\newcommand{\norm}[1]{\lVert #1\rVert}

\newcommand{\Ass}{\textbf{Ass}}
\newcommand{\colim}{\operatornamewithlimits{colim}}
\newcommand{\elm}[1]{\mathrm{elt}(\sse{#1})}
\newcommand{\elt}{\mathrm{elt}}
\newcommand{\sse}[1]{\mathscr{#1}}
\newcommand{\Hom}{\textrm{Hom}}
\newcommand{\WHT}{\mathrm{WHT}}
\newcommand{\WT}{\mathrm{WT}} 
\DeclareMathOperator{\cl}{cl}

\newcommand{\Ttop}{\textbf{Top}}
\newcommand{\vect}[1]{\underline{#1}}

\newcommand{\Nerve}{\mathcal{N}}
\newcommand{\bt}[1]{ {#1}_{\bullet} }
\newcommand{\bts}[1]{ {\vect{#1}}^{\bullet} }

\newcommand{\Id}{\textrm{Id}}

\newcommand{\Chains}{\textbf{Ch}^{-}(\mathbb{Z}) }

\newcommand{\FNC}{\textrm{FN}^{\textrm{sd}}}

\newcommand{\Rext}{\bar{\mathbb{R}}_{\ge 0 } }

\newcommand{\BFNA}{\mathcal{B}(\FN_m,A)}

\newcommand{\BZ}{\textrm{sd\,}\NBZ}
\newcommand{\NBZ}{\mathrm{BZ}}
\newcommand{\sgn}{\textrm{sgn}}

\newcommand{\std}{\textrm{std}}
\newcommand{\Cube}{\mathcal{Q}}
  
\newcommand{\bit}{\begin{itemize}} 
\newcommand{\eit}{\end{itemize}}

\allowdisplaybreaks

\newtheoremstyle{mystyle}
  {3pt} % Space above
  {3pt} % Space below
  {} % Body font
  {} % Indent amount
  {\bfseries} % Theorem head font
  {.} % Punctuation after theorem head
  {.5em} % Space after theorem head
  {} % Theorem head spec (can be left empty, meaning `normal')
  
\theoremstyle{mystyle} 
\newtheorem{theorem}{Theorem}[section]
\newtheorem{corollary}[theorem]{Corollary}
\newtheorem{definition}[theorem]{Definition}
\newtheorem{lemma}[theorem]{Lemma}
\newtheorem{remark}[theorem]{Remark}
\newtheorem*{recall*}{Recall}

\newtheorem{example}[theorem]{Example}
\theoremstyle{plain}
\newtheorem*{ss-theorem}{Theorem \ref{ssthm}}
\newtheorem*{zz-theorem}{Theorem \ref{zzthm}}

\graphicspath{{./}}

\author{Andrea Marino}

\title{A Fox-Neuwirth Basis for the Sinha Spectral Sequence}

\date{}

\MakeOuterQuote{"}

\begin{document}
\maketitle

\begin{abstract}
Recently, Sinha defined a spectral sequence approximating the (co)homology of the space of long knots in $\mathbb{R}^m$ modulo immersions, stemming from a cosimplicial structure on the compactified configuration spaces à la Kontsevich. We provide an equivalent cosimplicial structure on (the barycentric subdivision of) a regular CW complex with cells indexed by Fox-Neuwirth trees. As a corollary, we give a combinatorial presentation of the Sinha Spectral Sequence in terms of Fox-Neuwirth trees for all dimensions $m\ge 2$ and all coefficients $A$.
\end{abstract}

\tableofcontents

\section{Introduction}Knot invariants, from a modern perspective, can be regarded as elements of the zeroth cohomology group of the space of (long) knots in $\mathbb{R}^3$. This interpretation gained prominence following the seminal contributions of Vassiliev \cite{Vassiliev}, who developed a spectral sequence approximating the cohomology of the space of knots in $\mathbb{R}^m$ by resolution of singularities. In further work \cite{vas-computation,vas-combinatorial}, he explicitly computed certain cohomology classes using higher differentials of the spectral sequence—some of which have since received elegant geometric interpretations \cite{teiblum-turchin, budney}.
\\

In the classical case $m=3$, this spectral sequence in degree zero only captures a subset of all invariants—those known as finite-type. A central problem in modern knot theory is whether every pair of distinct knots can be distinguished by such finite-type invariants. While a definitive answer remains far off, recent advances have emerged from a different framework: embedding calculus, as introduced by Goodwillie and Weiss \cite{weiss,weiss-2}.
\\

This theory provide, among others, algebraic means to investigate the space of long knots in terms of embeddings from open intervals into $\mathbb{R}^m$. The result of the construction is a \textit{Taylor Tower}—a sequence of spaces that approximates $\Emb_m$ increasingly well. The $n$-th stage of the tower for long knots \textit{modulo immersions} $\overline{\Emb}_m$ is particularly well-behaved: it turns out to be homotopically determined by a configuration of at most $n$ points in $\mathbb{R}^m$, that cna be thought of as $n$ points uniformly sampled from the knot. Sinha, in his foundational article\cite{sinha2004operads}, proved that the entire Taylor tower for $\overline{\Emb}_m$ can be modeled by a cosimplicial structure on \textit{compactified} configuration spaces, known as Kontsevich spaces $\Kons_m(\bullet)$ \cite{sinha2009topology}. The compactification is needed to implement the coherence maps between the stages of the tower. The spectral sequence in (co)homology derived from $\Kons_m(\bullet)$ turns out to coincide on the second page (up to a regrading) with the first page of Vassiliev’s spectral sequence \cite{tourtchine2005bialgebra}. For $m \geq 4$ and rational coefficients, they both converge to the same limit, due to known collapses at early pages\cite{volic}.
\\

Further connections between these two perspectives have been drawn by recent work \cite{boavida,conant,kosanovic,kosanovic2}, which links stages of the Taylor tower to finite-type invariants that are additive under connected sum. These developments, along with the relative computational tractability of Sinha’s model, have contributed to its increasing adoption in the field. It was shown by Lambrechts, Turchin, and Volić \cite{Lambrechts_2010} that for $m\ge 4$, and later generalized by Tsopmene \cite{tsopmene} to $m=3$, that the Sinha Spectral Sequence with rational coefficients collapses at the second page. Their proof relies on the rational formality of the operadic pair $(\mathbb{E}_1, \mathbb{E}_m)$. Explicit computations by Turchin \cite{tourtchine2005bialgebra} on the second page, together with the collapse, provides precious information about the rational cohomology of the space of long knots (modulo immersions).
\\

The case of non-zero characteristic, however, remains largely unexplored. For $m=2$, the spectral sequence is known not to collapse over any of $\mathbb{Z}$, $\mathbb{Q}$, or $\mathbb{F}_p$ \cite{willwacher,goodwillie}. This is consistent with the failure of formality for $\mathbb{E}_m$ in positive characteristic, as shown in \cite{salvatore2018planar}. For general dimension $m$, Boavida De Brito and Horel \cite{boavida} showed that the only possibly non-trivial differentials over $\mathbb{F}_p$ are $d_{1+n(m-1)(p-1)}$ for some $n\ge 0$.  Combinatorial models such as the surjection operad \cite{smith} offer a potential method to compute the differentials of the spectral sequence, but they become unwieldy in higher dimensions. As we explain in Section \ref{configurations}, only efficient complexes in terms of number of cells are practically viable for such computations.
\\

Among known models, the Fox–Neuwirth decomposition provides the most efficient cellular structure for the compactified configuration spaces $\Conf_n(\mathbb{R}^m)$, with origins in the work of Nakamura \cite{Nakamura} and later recovered by Giusti and Sinha \cite{Giusti}. Cells correspond to Fox–Neuwirth trees $\FNP_m(n)$—combinatorial objects given by sequences of inequalities $\sigma(1) <_{i_1} \dots <_{i_{n-1}} \sigma(n)$ indexed by permutations and height indices. The sum of indices gives the cell’s codimension. Building on the work by De Concini-Salvetti\cite{DeConcini}, Blagojević–Ziegler constructed a complex\cite{blagojevic2013convex}, denoted $\NBZ_m(n)$, dual to the Fox-Neuwirth stratification on the space of configurations. In practice, $\NBZ_m(n)$ is a regular CW complex that deformation retracts onto $\Conf_n(\mathbb{R}^m)$, with cells indexed by the same Fox–Neuwirth trees (but dual dimension grading).
\\

We define in Section \ref{fnposet} a cosimplicial structure on the posets $\FNP_m(\bullet)$, analogous to that on the Kontsevich spaces. A coface map $d_i$ for $0<i<n+1$ may be interpreted geometrically as replacing the point $x_i$ of the configuration by a pair $x_i, x_{i+1}$ aligned vertically; at the level of trees, this corresponds to inserting $i <_{m-1} i+1$ and relabeling as needed. The coface $d_0$ (resp. $d_{n+1}$), on ther other hand, adds a new point $x_1$ (resp. $x_{n+1}$) at the bottom (resp. on top) of the configuration. The nerve of each poset, followed by normalized chains functor, yields a semicosimplicial chain complex $\FNC_m(\bullet)$, whose associated total bicomplex we call the Barycentric Fox–Neuwirth bicomplex (Def. \ref{bar-fn}). Its generators are ascending chains of trees, forming the barycentric subdivision of $\NBZ_m(\bullet)$. The horizontal differential $D_1$ is induced by the alternated sum of cofaces, and the vertical differential $D_0$ is the boundary map of the barycentric cell structure.
\\

The main goal of this article is to show the following
\begin{ss-theorem} For all $m \ge 2$ and abelian groups $A$, the Barycentric Fox-Neuwirth and Sinha Spectral Sequence in (co)homology are isomorphic from the first page on.
In particular, $E_{pq}^r(\FNC_m,A)$ (resp. $E^{pq}_r(\FNC_m,A)$) converges to the homology (resp. cohomology) of $E_m$ with coefficients in $A$ for $m \ge 4$. 
\end{ss-theorem}
This result is exploited in \cite{am} to show that the Sinha Spectral Sequence in dimension $3$ and $\mathbb{F}_2$ coefficients does not collapse at the second page, as opposed to the rational case.
\\

The strategy is to connect the (semi)cosimplicial spaces $\Kons_m$ and $\BZ_m$ via a zig-zag of homotopy equivalences; the result will then follow by abstract non-sense. We will bridge the two constructions via an object $\WHT_m(n)$ that serves as a common ground between configuration spaces and their compactification. Formally, we will prove that:
\begin{zz-theorem}For all $m\ge 2$, there is a zig-zag of semicosimplicial homotopy equivalences:
\[\begin{tikzcd}
	& {\textrm{WHT}_m} \\
	{\BZ_m} && {\textrm{Kons}_m}
	\arrow["\simeq"', from=1-2, to=2-1]
	\arrow["\simeq", from=1-2, to=2-3]
\end{tikzcd}\]
\end{zz-theorem}

Let us provide an intuition behind the proof. The coface action $d_i$ in $\Kons_m(n)$ replace $x_i$ with a pair of points $x_i, x_{i+1}$ infinitesimally close, with direction of collapse being vertical. Similarly, $d_0$ (resp. $d_{n+1}$) adds a point infinitely distant at the bottom (resp. on top) of the configuration. On the other hand, $\BZ_m(n)$ possesses a piecewise linear cosimplicial structure: cofaces maps barycenters $v(\Gamma) \in \Gamma$ to $v(d_i\Gamma) \in d_i \Gamma$, and then extend linearly to the barycentric cell structure. At a closer inspection, one can see that the configuration $v(d_i\Gamma)$, for $i\neq 0, n+1$ takes the point $x_i$ in $v(\Gamma)$ and replaces it with $x_i, x_{i+1}$, where the new point is placed in vertical direction at distance $1$. For $i=0$ (resp. $i=n+1$), the new point is placed at the bottom (resp. on top) of the configuration at distance $1$.
\\

In order to fill the gap between distance $1$ and $0,\infty$, we craft a space $\WHT_m(n)$ that would admit a family of coface actions, with $d_i$ doubling a point at distance $\varepsilon$ and $d_0, d_{n+1}$ placing a new point at distance $\frac{1}{\varepsilon}$ at $\pm \infty$. Geometrically, the space can be thought of as a space of configurations where points are allowed to collapse in vertical direction only. This is enough to define the infinitesimal coface action, and allows for a simple "collapse resolution" map to actual configurations—transforming the infinitesimal vertical collapses in small distances (precisely, one). In practice, we endow the spaces $\WHT_m$ with the infinitesimal coface action, and show that it reduces to the two structures $\BZ_m(n), \Kons_m(n)$ when applying respectively the maps $\WHT_m(n) \to \BZ_m(n), \WHT_m(n)\to \Kons_m(n)$.
\\

The formal implementation of this idea is based on how the configurations of points corresponding to barycenters $v(\Gamma)$ are constructed from $\Gamma$. By allowing weights $\in \mathbb{R}_{>0}$ on the branches of a Fox-Neuwirth tree $\Gamma$, we can govern the distance between points in the corresponding configuration $v(\Gamma)$ (Def. \ref{ass-conf}). This provides the first step of the construction, \textit{the space of weighted trees} $\WT_m(n)$. Such space is obtained by "blowing-up" each barycenter to an open cube $\mathbb{R}_{>0}^{n-1}$ (representing weights). The cells, corresponding to the convex combination of barycenters, become the join of such open cubes (Remark \ref{polytope-exploration-wt}). These "blown-up" cells are then glued together following the same scheme of the barycentric cell structure. We formalize this procedure via a \textit{twisted geometric realization}, in which the simplicial set $\Nerve(\FNP_m(n))$ is geometrically realized by using more complicated bricks than the classical simplices (Section \ref{twisted-geo}).
\\

However, in order to allow these weights to degenerate to $0, \infty$ in the vertical direction, a more sophisticated construction is needed. In short, the technical difficulties arise from the need to make these "virtual collapses" geometrically meaningful when seen in $\Kons_m$. A general point in $\WHT_m(n)$ is a convex combination of configurations with groups of vertically collapsed points. In order for the convex combination to be geometrically well-defined (e.g. corresponding to a point of $\Kons_m(n)$), the partitions of the configurations in vertically collapsed groups have to be the same for all summed configurations. This subtlety requires us to introduce a stratification that keeps track of the groups of collapsed points, resulting in a rather technical construction (Section \ref{wht-section}). Such blocks of collapsed points are encoded into groups of "hairs" on the weighted trees, hence the name \textit{space of weighted hairy trees} for the object $\WHT_m(n)$.
\\

Once the space is constructed and endowed with the infinitesimal vertical (semi)cosimplicial structure, the two maps $\WHT_m(n) \to \BZ_m(n)$ and $\Kons_m(n)$ are formalized and showed to respect the two key properties (Section \ref{kons-to-fox-l}): respecting the cofaces and being homotopy equivalences. The first property comes from the geometric intuition of infinitesimal doubling that we provided; the second one is a consequence of the fact that all the involved spaces—$\BZ_m(n), \Kons_m(n), \WT_m(n), \WHT_m(n)$—are homotopical variants of the space of configurations.
\\

We acknowledge the University of Tor Vergata and the University of Victoria for providing the space and resources that made the research possible. Heartfelt thanks to my Ph.D. supervisor, Paolo Salvatore, for having inspired my mathematical style and ideas behind the paper in many ways. Thanks to Dev Sinha and Victor Turchin for the thorough review of the presented material. Lastly, thanks to Tommaso Rossi, Andrea Bianchi and Lorenzo Guerra for the helpful conversations during the writing of this paper.

\nopagebreak
\section{Background and Notation}
In this section, the fundamental definitions and theorems behind the article are introduced.

\subsection{The Fox-Neuwirth Stratification on Configuration Spaces} \label{configurations}

Configuration spaces are a fundamental object of a study in geometry, as they arise naturally in numerous topological problems\footnote{See, for example, their connection to little disk operads, braid groups and embedding spaces \cite{knudsen}}. Given $n,m\ge 1$, they are defined as
$$ \Conf_n(\mathbb{R}^m) : \{ (x_1, \ldots, x_n) \in (\mathbb{R}^m)^n: x_i \neq x_j \ \ \textrm{ for } i \neq j \}\ . $$
In order to study (co)homological properties of configuration spaces, a particuarly effective combiantorial presentation is given by the Fox-Neuwirth stratification \cite{Giusti}. Let us denote by $<$ the lexicographical order on $\mathbb{R}^m$. Every configuration $\vec{x} \in \Conf_n(\mathbb{R}^m)$ defines a permutation $\sigma \in \Sigma_n$ such that 
$$x_{\sigma(1)} < \ldots < x_{\sigma(n)}$$
Although the subspace of configurations with a given permutation is contractible, it is not regular enough to provide a convenient stratification. We need to retain more information about the lexicographic structure — not only the order, but also the number of shared components. This concept is formalized in the following way:

\begin{definition}A depth-ordering of height $m$ on a set $S$ is the data of
\begin{itemize}
\item A linear order on $S$;
\item For each strict inequality $x < y$ in $S$, a depth index $d(x,y) \in \{0, \ldots, m-1\}$ such that
$$ d(x,z) = \min\{d(x,y), d(y,z) \} \ .$$
\end{itemize}
We write $x <_a y$ as a notation for $d(x,y) = a$.
\end{definition}
The lexicographic order induces a depth-ordering on $\mathbb{R}^m$, with $x <_d y$ whenever $x$ and $y$ share the first $d$ components. Notice that in case of a finite set $S=\{1, \ldots, n\}$, a depth-ordering $\Gamma$ amounts to the folowing data:
\begin{itemize}
\item A permutation $\sigma \in \Sigma_n$, that determines the linear order;
\item Depth indices $a_1, \ldots, a_{n-1} \in \{0, \ldots, m-1\}$ such that $\sigma(1) <_{a_1} \sigma(2) <_{a_2} \ldots <_{a_{n-1}} \sigma(n) $. The other depth-indices are determined by the min-rule.
\end{itemize}

The correct strata on $\Conf_n(\mathbb{R}^m)$ arising from the lexicographic order are then descibed in the following way:
\begin{definition} For a depth-ordering $\Gamma = (\sigma, a_{\bullet}) $ of height $m$ on $n$ points, define
$$ \Conf(\Gamma) := \{ \vect{x} \in \Conf_n(\mathbb{R}^m) : x_{\sigma(1)} <_{a_1} \ldots <_{a_{n-1}} x_{\sigma(n)} \} \ .$$
\end{definition}
This is known as the \textit{Fox-Neuwirth} stratification, as they first introduced it \cite{FN} for $m=2$. They were later generalized to arbitrary dimensions by Nakamura \cite{Nakamura}.

From now on, we denote by $\FNP_m(n)$ the set of depth-orderings of height $m$ on $\{1, \ldots, n\}$. The set $\FNP_m(n)$ has a useful combinatorial interpretation in terms of trees. Given $\Gamma= (\sigma, a_{\bullet}) \in \FNP_m(n)$ we can draw a tree of height $m$ and $n$ leaves with the following characteristics:
\begin{itemize}
\item Leaves are labeled according to $\sigma$;
\item Consecutive branches labeled $\sigma(i)$ and $\sigma(i+1)$ merge at height $a_i$.
\end{itemize}
\vspace{0.5cm}
\begin{figure}[h]
	\centering
	\includegraphics{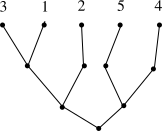}
	\caption{The tree associated to the Fox-Neuwirth cell $3 <_2 1 <_1 2 <_0 5 <_1 4 \in \FNP_3(5)$}
	\label{fox cell}
\end{figure}

 It is easy to see \cite{Giusti} that the following holds:
\begin{theorem} \label{FN-theorem}For any $\Gamma \in \FNP_m(n)$, the subspace $\Conf(\Gamma)$ is homeomorphic to a Euclidean ball of dimension $mn - \sum a_i$. The images of the $\Conf(\Gamma)$ are the interiors of cells in an equivariant
CW structure on the one-point compactification $\Conf_n(\mathbb{R}^m)^+$.
\end{theorem}
This theorem allows one to use the combinatorics of $\FNP_m(n)$ to describe homological properties of the configuration spaces. Indeed, the boundary map can be reformulated explicitly in terms of the trees associated to Fox-Neuwirth strata. The first step in this direction is to understand which strata are contained in the closure of others — the analogue of the face poset structure in the context of stratification, called exit poset. In \cite{blagojevic2013convex}, Lemma 3.3, it is proven that $\Conf(\Gamma') \subset \overline{\Conf(\Gamma)}$ if for all $\alpha, \beta \in \{1, \ldots,n\}$
$$ \alpha <_r \beta \text{ in } \Gamma \ \  \Rightarrow \ \ \alpha <_r \beta \text{ or } \alpha <_{s} \beta \textrm{ or } \alpha >_s \beta \text{ in } \Gamma', \text{ where } s < r\ . $$

We consider $\FNP_m(n)$ to be a poset with the \textbf{reverse} exit poset structure.\footnote{That is, minimal elements are open cells of $\Conf_n(\mathbb{R}^m)$, corresponding to trees with no merged branches.} The connection between the combinatorics of the poset and the homological properties is provided by the following construction. Building on the work of De Concini and Salvetti \cite{DeConcini}, Blagojevic and Ziegler showed in \cite{blagojevic2013convex} the existence of a subspace $\NBZ_m(n) \subset \Conf_n(\mathbb{R}^m)$ with the following features:
\begin{itemize}
    \item The inclusion $\NBZ_m(n) \to \Conf_n(\mathbb{R}^m)$ is a homotopy equivalence;
    \item It has a regular CW-complex structure, with cells  $c(\Gamma)$ indexed by $\FNP_m(n)$;
    \item It has the opposite order with respect to strata: $c(\Gamma) \subset \overline{c(\Gamma')}$ iff $\Conf(\Gamma') \subset \overline{\Conf(\Gamma)}$;
    \item If $\Gamma = (\sigma, a_{\bullet})$, the dimension of $c(\Gamma)$ is $\sum a_i$.
\end{itemize}
As a consequence, the (co)homology of $\Conf_n(\mathbb{R}^m)$ can be obtained from the (dual of the) cellular chain complex of $\NBZ_m(n)$. We denote such complex by $\MFN_m(n)$. The work \cite{Giusti} gives a complete description of signs in the differential. The presentation via Fox-Neuwirth trees is more efficient in terms of size than the McClure-Smith \cite{smith} surjection operad $\mathcal{S}_m(n)$, the other well-established model for $C_*(\Conf_n(\mathbb{R}^m))$.

Functorial constructions in terms of the poset $\FNP_m(n)$ will be related to the barycentric subdivision of $\NBZ_m(n)$, because of how the nerve of a poset is defined. Interestingly enough, $\BZ_m(n)$ is an explicit realization of $\Nerve(\FNP_m(n))$ inside the configuration space\footnote{Denoted in their article as $\Ff (m,n)$.}: Theorem 3.13 of \cite{blagojevic2013convex} shows that\footnote{Denoted in their article as $W_n^{\oplus d}$} $\BZ_m(n)$, when included in $\Conf_n(\mathbb{R}^m)$, is the boundary of a PL star-shaped closed subset. Explicitly, given $(\sigma, a_{\bullet})$ a Fox-Neuwirth cell in $\FNP_m(n)$, the vertices $v(\sigma, a_{\bullet})$  of $\BZ_m(n)$ are built in the following way: set $v_{\sigma(1)}(\sigma,a_{\bullet})=0$, and then define inductively
$$ v_{\sigma(p+1)} (\sigma, a_{\bullet} ) = v_{\sigma(p)} (\sigma, a_{\bullet}) + e_{a_p +1} \ .$$
Given a chain $(\sigma_0, a_{\bullet}^0) < \ldots < (\sigma_d, a_{\bullet}^d)$ in $\Nerve(\FNP_m(n))_d$, the corresponding face in $\Conf_n(\mathbb{R}^m)$ is the convex hull of the vertices $v(\sigma_0, a_0), \ldots , v(\sigma_d, a_d)$.  

The figure \ref{fig:zb-triangulation} illustrates the case $m=3, n=2$ with the following convention: the number of bars between two labels is $2-$depth, and $w(\Gamma) = v_2(\Gamma)$ represents the non-zero component. Since $\BZ_3(2)$ is a deformation retract of $\Conf_2(\mathbb{R}^3)$, we obtain a PL 2-sphere. The faces of the polyhedron corresponds to chains in the poset $\FNP_2(3)$.

\begin{figure}[h]
	\centering
	\includegraphics{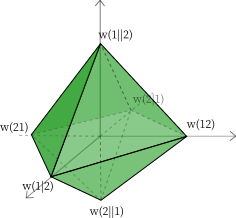}
        \caption{An illustration of $\BZ_3(2)$ as a PL-sphere}
	\label{fig:zb-triangulation}
\end{figure}

\subsection{The Cosimplicial Structure on Kontsevich spaces} \label{kons-cosimp}
A notable application of configuration spaces is to model the spaces constituting \textit{little disk operads} $\mathbb{E}_m$: a point in $\mathbb{E}_m(n)$ is a collection of $n$ disjoint disks of dimension $m$ inside a disk of the same dimension. The map $\mathbb{E}_m(n) \to \Conf_n(\mathbb{R}^m)$ retaining the centers of the disks is an homotopy equivalence. However, the operad structure on $\mathbb{E}_m$ — given by replacing a disk with a collection of disks — it is not clearly generalizable to $\Conf_n(\mathbb{R}^m)$: how should we canonically replace a point with a collection of points? A solution is given by allowing 'infinitesimal configurations': the configuration to insert is scaled to fit into a disk of radius $\epsilon \to 0$. A quantity that remains unchanged when scaling a configuration $\vect{x}$ in a disk of radius $\epsilon \to 0$ is
$$ (*) \ \ \ \phi_{ij}(\vect{x}) = \frac{x_i - x_j}{||x_i - x_j||} \in S^{m-1} \ ,$$
since it is homogeneous. This motivates the following
\begin{definition} The space $\Kons_m(n)$ is the closure in $(S^{m-1})^{\binom{n}{2}}$ of the image of
$$ \phi = (\phi_{ij})_{i \neq j} : \Conf_n(\mathbb{R}^m) \to (S^{m-1})^{\binom{n}{2}} \ ,$$
where $\phi_{ij}$ is defined as in equation $(*)$.
\end{definition}

In \cite{sinha2004manifold,sinha2004operads}, Sinha showed that $\Conf_n(\mathbb{R}^m) \to \Kons_m(n)$ is an homotopy equivalence, and that Kontsevich spaces possess an operadic structure akin to the little disk operads (Theorem 4.5). Points in Kontsevich spaces can be thought as collections of nested clouds, each cloud containing a configuration (Fig. \ref{Kontsevich}).

\begin{figure}[H]
    \centering
    \includegraphics[scale=0.25]{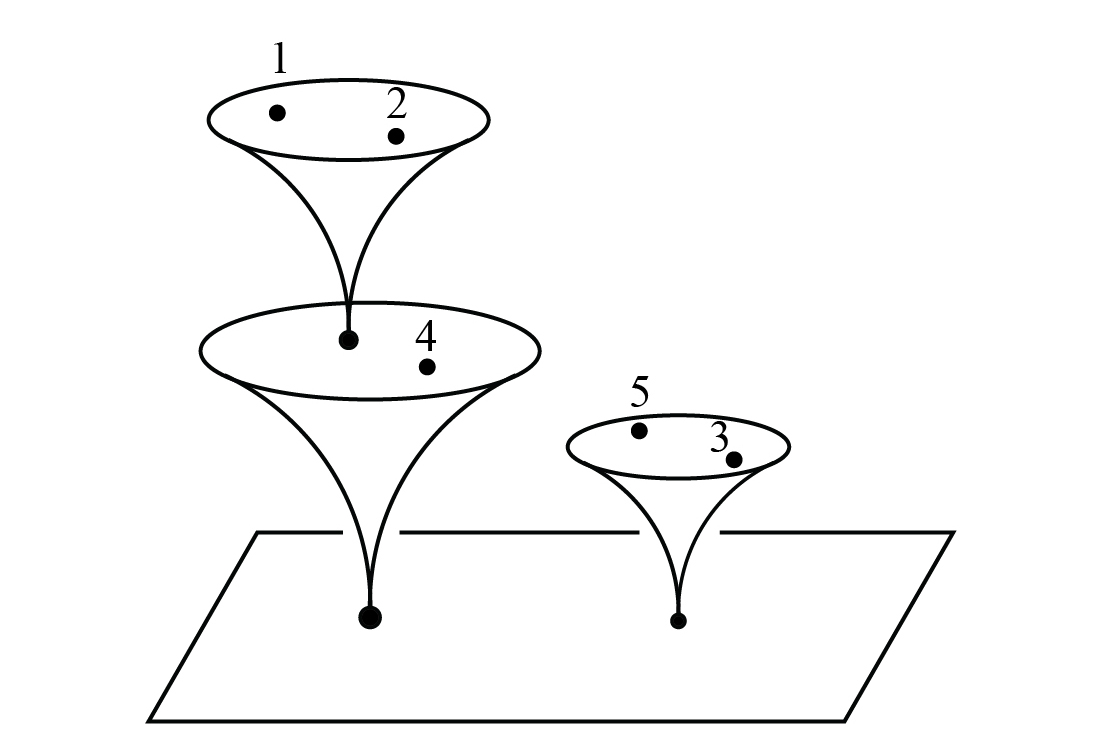}
    \caption{A point in $\Kons_2(5)$}
    \label{Kontsevich}
\end{figure}

The operad structure on $\Kons_m$ can be used to define a cosimplicial structure. McClure and Smith \cite{mcclure2001solution} proved that \textit{multiplicative operad} $\mathcal{O}$ - that is, equipped with a morphism $\Ass \to \mathcal{O}$  - defines a cosimplicial structure in the following way. Note that multiplicativity corresponds to a distinguished binary operation $\mu \in \mathcal{O}^2$  such that $\mu( \mu(-,-),-) = \mu(-, \mu(-,-))$ (a \textit{multiplication}) and a unit $e \in \mathcal{O}^0$ such that $\mu(e,-) = \mu(-,e) = \Id$.  The maps $d_i : \mathcal{O}^n \to \mathcal{O}^{n+1}$ for $i=0, \ldots, n+1$ and $s_j: \mathcal{O}^{n+1} \to \mathcal{O}^n$ for $j=0, \ldots, n$ are then defined as:
\begin{align*}
    \boxed{1 \le i \le n} & \ \ \ \ \  d_i :& \hspace{-3cm} \mathcal{O}^n \overset{(\Id,\mu)}{\to} \mathcal{O}^n \times \mathcal{O}^2 \overset{\circ_i}{\to} \mathcal{O}^{n+1} \ , \\
    \boxed{i=0} & \ \ \ \ \  d_0 :& \hspace{-3cm} \mathcal{O}^n \overset{(\mu,\Id)}{\to} \mathcal{O}^2 \times \mathcal{O}^n \overset{\circ_1}{\to} \mathcal{O}^{n+1} \ ,\\
    \boxed{i=n+1} & \ \ \ \ \  d_{n+1} :&  \hspace{-3cm}\mathcal{O}^n \overset{(\mu,\Id)}{\to} \mathcal{O}^2 \times \mathcal{O}^n \overset{\circ_2}{\to} \mathcal{O}^{n+1} \ ,\\
    \boxed{0 \le j \le n} & \ \ \ \ \ s_j :& \hspace{-3cm}\mathcal{O}^{n+1} \overset{( \Id,e)}{\to} \mathcal{O}^{n+1} \times \mathcal{O}^0 \overset{\circ_{j+1}}{\to} \mathcal{O}^n \ ,
\end{align*}

summarized in the picture \ref{smith} using corollas' diagrams for operads' operations. The axioms for multiplication ensure that cosimplicial identities are satisfied. A multiplication for the Kontsevich operad corresponds to a direction $\mu \in \Kons_m(2) \cong S^{m-1}$. One can see that any such direction satisfies the axioms for a multiplication and, since the sphere is homogeneous, all such cosimplicial structures are isomorphic. There is no choice for the unit, as\footnote{since $\Conf_0(\mathbb{R}^m) = \Hom( \emptyset, \mathbb{R}^m) = \{\emptyset\}$, and $\Kons_0(\mathbb{R}^m)$ is the closure of the map from $\{\emptyset\}$ to $(S^{m-1})^{\binom{0}{2}} = \{*\}$.} $\Kons_m(0) = \{*\}$. In geometrical terms, the resulting cosimplicial moves are:
\begin{itemize}
    \item $s_j$ forgets the $(j+1)$-th point;
    \item $d_i$ for $i=1, \ldots,n$ doubles the $i$-th point, placing $(i+1)$ infinitesimally close to $i$ in direction $\mu$;
    \item $d_0$ inserts a new point with label $1$ at $\infty$ in direction $-\mu$;
    \item $d_{n+1}$ inserts a new point with label $n+1$ at $\infty$ in direction $\mu$
\end{itemize} 

\begin{figure}[H]
    \centering
    \includegraphics[scale=0.15]{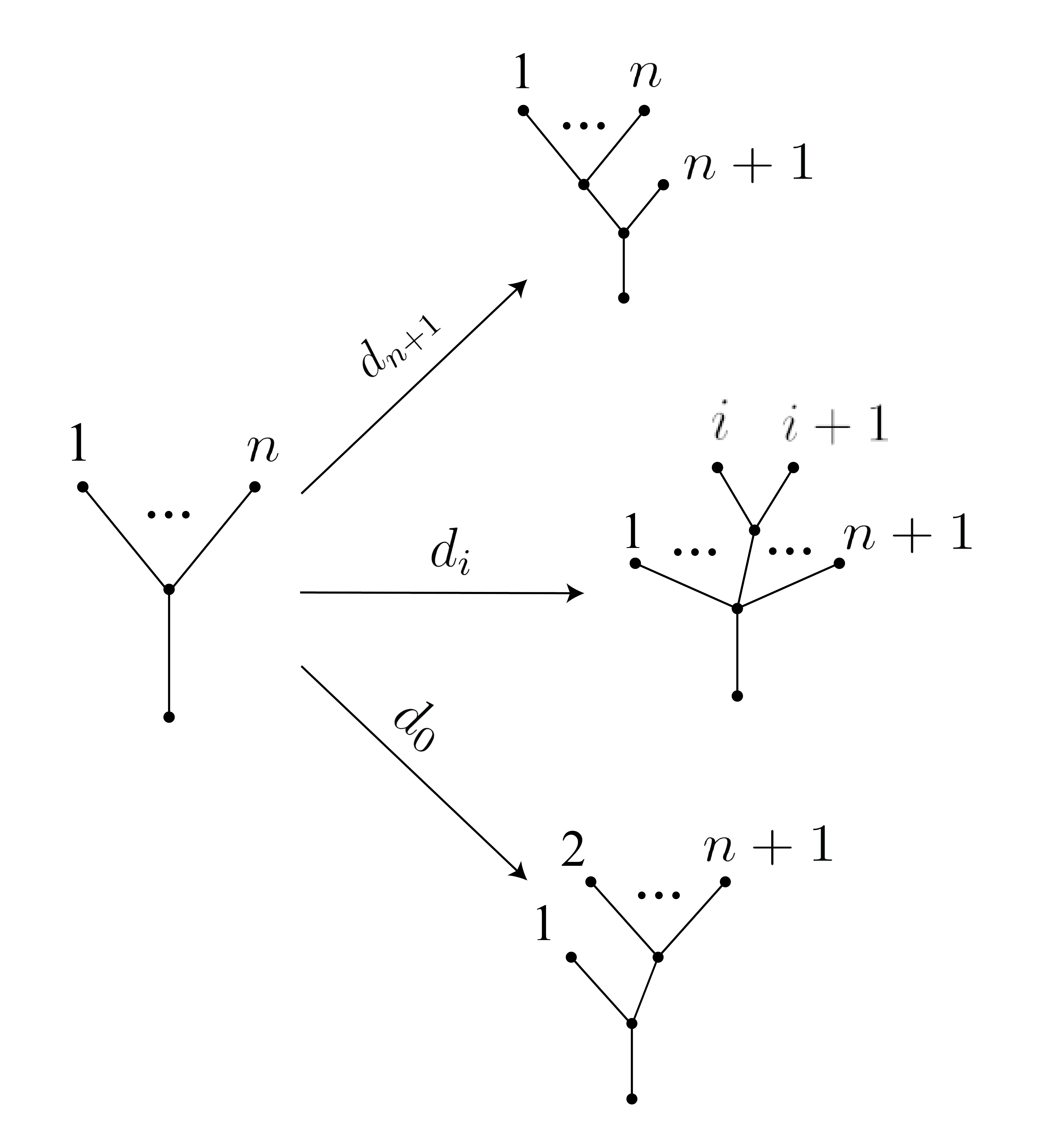}
    \caption{The McClure-Smith cosimplicial structure on a multiplicative operad}
    \label{smith}
\end{figure}

The convention used by Sinha for the multiplication is the southpole of the sphere. In the present article, we will use $\mu = e_m$: when looking for algebraic analogs of the cosimplicial structure on Fox-Neuwirth Trees, the choice of a direction matters. Fox-Neuwirth trees are based on lexicographic order, and the direction $\mu =e_m$ ensures that we perturb a configuration in the 'lexicographically smallest' component.
 
\subsection{Barycentric Fox-Neuwirth Bicomplex}
\label{fnposet}
In \cite{am}, we presented an analog of the cosimplicial structure on Kontsevich Spaces for Fox-Neuwirth cells, that ultimately furnishes a degreewise finite-dimensional bicomplex constructed from Fox-Neuwirth Trees. The aim of the article is to prove that the spectral sequence of this bicomplex is isomorphic to the Sinha Spectral Sequence from the first page on. Let us recall the relevant definitions.

Firstly, we define a cosimplicial structure on the Fox-Neuwirth posets that turns $\FNP_m(n)$ into a functor
$$\FNP_m( \bullet) : \Delta \to \textrm{Poset} \ .$$ 
Let us define the image of $\Delta$ generators\footnote{We denote by $d_i, s_j$ both the generators of $\Delta$ and the cosimplicial structure maps; their meaning can be easily deduced from the context.}:

\begin{definition} \label{fn-cosimplicial}
Given a depth-ordering $\Gamma =(\sigma, a) \in \FNP_m(n)$, we construct the depth-orderings $ d_i \Gamma, s_j \Gamma$  in the following way. Let $\alpha := \sigma^{-1}(i)$ and $\beta := \sigma^{-1}(j+1)$.
\begin{itemize}
\item \textit{Internal cofaces, $0 < i < n+1$}. 
$$ d_i(\sigma)(k) = \left\{ \begin{array}{ll}
        i,  & \textrm{if }  k =\alpha\\
       d_i (\sigma( s_{\alpha}(k))) , & \textrm{ otherwise} 
    \end{array} \right. \ \ \ \ \ \ d_i(a)_k= \left\{ \begin{array}{ll}
        m-1,  & \textrm{if }  k =\alpha\\
       a_{s_{\alpha}(k)} , & \textrm{ otherwise} 
    \end{array} \right. 
    $$
\item \textit{Left extremal coface, $i=0$}
$$ d_0(\sigma)(k) = \left\{ \begin{array}{ll}
        1,  & \textrm{if }  k =1\\
       \sigma(k-1)+1 , & \textrm{ otherwise} 
    \end{array} \right. \ \ \ \ \ \ d_0(a)_k= \left\{ \begin{array}{ll}
        m-1,  & \textrm{if }  k =1\\
       a_{k-1} , & \textrm{ otherwise} 
    \end{array} \right.
    $$
\item \textit{Right extremal coface, $i=n+1$} 
$$ d_{n+1}(\sigma)(k) = \left\{ \begin{array}{ll}
        n+1,  & \textrm{if }  k =n+1\\
       \sigma(k) , & \textrm{ otherwise} 
    \end{array} \right.  \ \ \ \ \ \ d_{n+1}(a)_k= \left\{ \begin{array}{ll}
        m-1,  & \textrm{if }  k = n\\
       a_k , & \textrm{ otherwise} 
    \end{array} \right.
    $$
\item \textit{Codegeneracy, $0 \le j \le n-1$}
$$ s_j(\sigma)(k) = s_{j+1}( \sigma ( d_{\beta}(k) ) )   \ \ \ \ \ \  s_j(a)_k= \left\{ \begin{array}{ll}
        \min \{ a_{\beta-1}, a_{\beta} \},  & \textrm{if }  k = \beta-1 \\
       a_{ d_{\beta}(k) }, & \textrm{ otherwise} 
    \end{array} \right.
    $$
\end{itemize}
\end{definition}
\begin{remark} Informally, the cosimplicial structure acts in the following way:
\begin{itemize}
\item The $i$-th internal coface substitutes the branch labeled $i$ with a small fork labeled $i, i+1$;
\item The left (resp. right) extremal coface inserts $1$ (resp. $n+1$) with a small fork at the left (resp. right) of the tree;
\item The $j$-th codegeneracy eliminates the branch labeled $(j+1)$. 
\end{itemize}
\end{remark}

By taking the nerve of involved posets, followed by normalized simplicial chains, we get a cosimplicial chain complex $ \FNC_m( \bullet) : \Delta_s \to \Chains$. Taking the alternated sum of faces, we obtain the following

\begin{definition} \label{bar-fn}For any $m \ge 2$, the \textit{Barycentric Fox-Neuwirth Bicomplex} with coefficients in an abelian group $A$ is defined as
$$ \BFNA_{dn} = \FNC_m(n)_d \otimes A \ ,$$
with horizontal differential induced by
$$ \partial_h( \Gamma_0 < \ldots < \Gamma_d) = \sum_i (-1)^i \Gamma_0 < \ldots < \hat{\Gamma}_i < \ldots < \Gamma_d \ ,$$
and vertical differential by
$$ \partial_v(\Gamma_0 < \ldots < \Gamma_d) = \sum_i (-1)^i d_i \Gamma_0 < \ldots < d_i \Gamma_d \ .$$
Here $\Gamma_0 < \ldots < \Gamma_d$ is a chain of Fox-Neuwirth trees with $n$ leaves and height $m$. The spectral sequence that starts by taking vertical homology is denoted by $E_{pq}^r(\FNC_m,A)$. The vertical spectral sequence associated to the \textit{dual} bicomplex $\BFNA^{\vee}$ is denoted by $ E^{pq}_r(\FNC_m, A)$.
\end{definition}

\begin{remark} The main claim of the present article is that the spectral sequences $E_{pq}^r(\FNC_m,A), E^{pq}_r(\FNC_m, A)$ are isomorphic, from the first page on, to the Sinha Spectral Sequences respectively in homology and cohomology. See Theorem \ref{ssthm} for a more detailed account.

\end{remark}
\section{Combinatorics of cosimplicial Fox-Neuwirth trees}
This is a rather technical sections, in which we explore the combinatorics of the cosimplicial action on Fox-Neuwirth trees. We suggest using tree pictures as Fig. \ref{shape-tree} to follow proofs. Many results are intuitive once graphically reformulated. On a first read, proofs in this section can be safely skipped.

\subsection{The shape-tree lemma}
This paragraph is devoted to the proof of the lemma appearing in the title:
\begin{lemma}[Shape-Tree]\label{twist-initial} Let $\Lambda \in \FNP_m(n)$ and $\psi:[n] \to [\ell]$. Denote by $\Lambda = (\sigma^{\Lambda}, a^{\Lambda})$ and $\psi \Lambda = (\sigma^{\psi \Lambda}, a^{\psi \Lambda} ) $ the permutations and depth indices associated to the Fox-Neuwirth trees $\Lambda, \psi \Lambda$. Then the following are true:
\begin{enumerate}
\item For all $1 \le t \le \psi(0)$
$$ \sigma^{\psi \Lambda} t = t,\ \ \ \ a^{\psi \Lambda}_t = m-1\ . $$
\item If $\psi(0) < \alpha < \beta < \psi(n)$ and $\{\alpha, \ldots, \beta-1\} \subset (\im \psi)^c$, then
\begin{align*}
(1) \ \ \ \ & (\sigma^{\psi \Lambda})^{-1}(\beta) - (\sigma^{\psi \Lambda})^{-1}(\alpha) = \beta - \alpha \\
(2) \ \ \ \ & \textrm{for all } (\sigma^{\psi \Lambda})^{-1}(\alpha) \le t \le (\sigma^{\psi \Lambda})^{-1}(\beta) -1: \ \ \ a^{\psi \Lambda}_t = m-1 \ .
\end{align*}
\item For all $\psi(n) +1 \le s \le \ell$
$$ \sigma^{\psi \Lambda} (s) = s, \ \ \ a^{\psi \Lambda}_s = m-1\ . $$
\end{enumerate}

\end{lemma}
\begin{remark} What the shape tree lemma says is summarized by picture \ref{shape-tree}. The upshot is that when applying $\psi$ we replace each $k$ with a bunch of little "hairs"; labels on the hairs are consecutive integers. Furthermore, the last $\ell - \psi(n)$ integers of $\{1, \ldots, \ell\}$ appear on hairs at the end of the tree, and the first $\psi(0)$ at the beginning. We will deepen this observation in Remark \ref{hair-blocks-position}.
\end{remark}

\begin{figure}[H]
	\centering
	\includegraphics[scale=0.17]{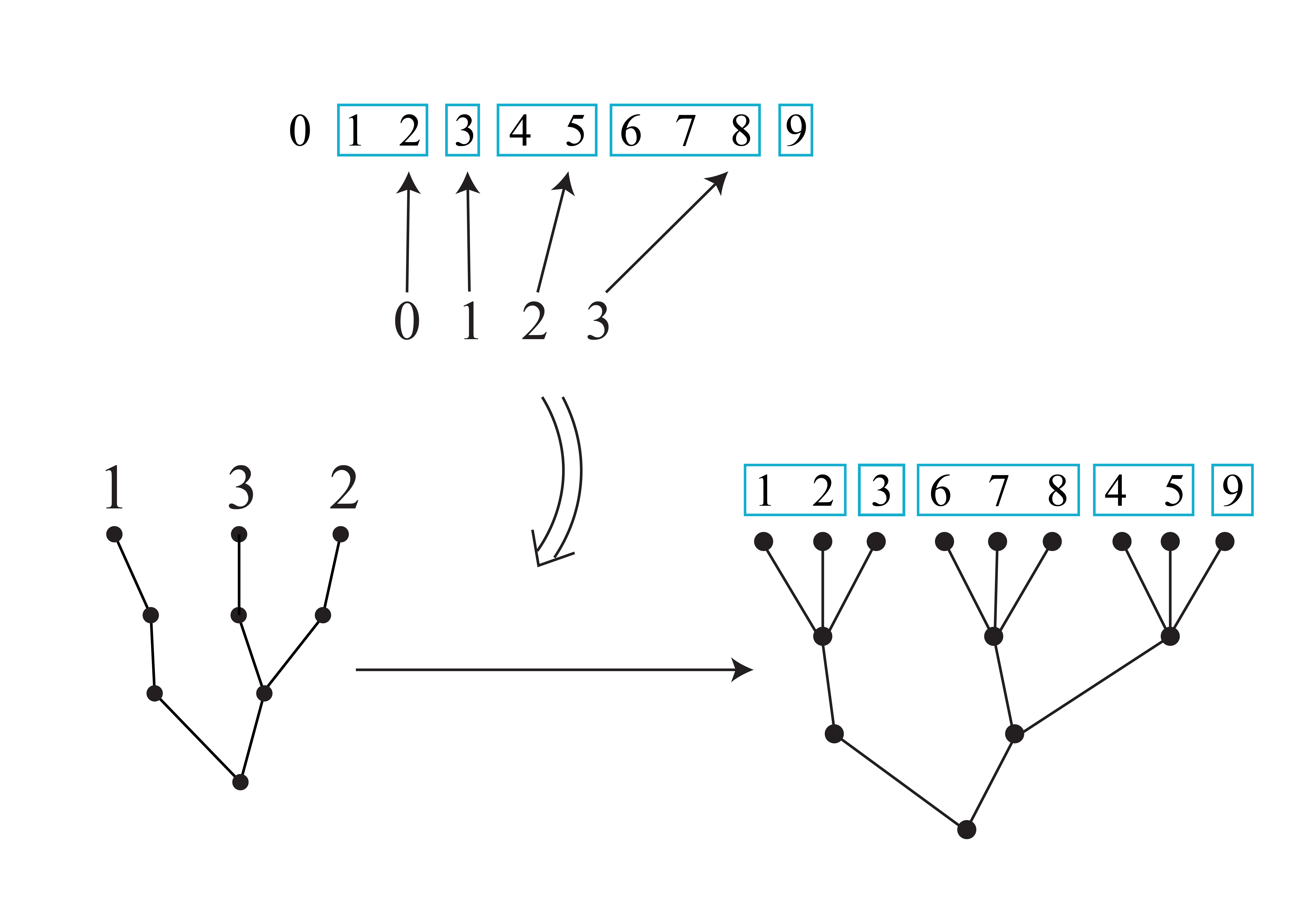}
	\caption{The shape of a Fox-Neuwirth tree after several cosimplicial moves}
	\label{shape-tree}
\end{figure}

\begin{proof}We separately show the three points, although the strategies are quite similar.
\begin{enumerate}
\item Write $\psi=\psi' d_0^{\psi(0)}$, where
$$\psi'(t) = \left \{ \begin{array}{ll} 
t, & \textrm{if } t \le \psi(0) \\
\psi(t-\psi(0)), & \textrm{if } t \ge \psi(0)
\end{array} \right. $$
It is well defined in $t= \psi(0)$ and it gives an increasing function. Also, note that
$$ \psi'd_0^{\psi(0)}(t) = \psi'(t+\psi(0) ) = \psi(t) \ . $$
We proceed by induction on the length of $\psi'$.

If $\psi'$ is the identity, let us show $\sigma^{d_0^p \Lambda}(t) = t$ for $t \le p$ by (a nested) induction on $p \ge 0$ . For $p=0$ there is nothing to show. Let us show the inductive step. For $t \ge 1$, we have
$$ \sigma^{d_0 d_0^p \Lambda}(t+1) = 1+ \sigma^{d_0^p \Lambda}(t) = 1+t\ , $$
$$ a^{d_0 d_0^p \Lambda}_{t+1} = a^{d_0^p \Lambda}_t = m-1\ ,$$
by the inductive hypothesis, while $\sigma^{d_0 d_0^p \Lambda}(1) = 1$ and $a^{d_0 d_0^p \Lambda}_1 = m-1$ by definition.

Now let us examine the inductive step of the induction on $\psi'$. Let us write $\psi' = d_k\phi$ for some $\phi$ of shorter length. Note that since $\psi'(t) = t$ for $t \le \psi(0)$, we must have $k > \psi(0)$. There are two cases. If $k = \ell+1$, then
$$ \sigma^{d_{\ell+1} \phi d_0^{\psi(0)} \Lambda}(t ) = \sigma^{\phi d_0^{\psi(0)} \Lambda}( t) = t \ ,$$
$$ a^{ d_{\ell+1} \phi d_0^{\psi(0)} \Lambda}_t = a^{\phi d_0 d_0^{\psi(0)} \Lambda}_t = m-1  \ ,$$
Since $t \le \psi(0)$. If $k \le \ell$:
$$ \sigma^{d_k \phi d_0^{\psi(0)} \Lambda}(t ) = d_k \sigma^{\phi d_0^{\psi(0)} \Lambda} s_{\gamma} t \ , $$
$$ a^{d_k \phi d_0^{\psi(0)} \Lambda}_t = a^{\phi d_0^{\psi(0)} \Lambda}_{ s_{\gamma} t} \ ,$$
where $\gamma$ is such that $\sigma^{\phi d_0^{\psi(0)} \Lambda}(\gamma) = k$. Since $k > \psi(0)$ and $\sigma^{\phi d_0^{\psi(0)} \Lambda}\{1, \ldots, \psi(0)\} = \{1, \ldots, \psi(0)\}$ by the inductive hypothesis, we must have $\gamma > \psi(0) \ge t$ too. We conclude that
$$ d_k \sigma^{\phi d_0^{\psi(0)} \Lambda} s_{\gamma} t = d_k \sigma^{\phi d_0^{\psi(0)} \Lambda} t = d_k  t = t\ ,$$
$$ a^{\phi d_0^{\psi(0)} \Lambda}_{ s_{\gamma} t }= a^{\phi d_0^{\psi(0)} \Lambda}_t = t \ . $$
In the second passage, we used again the inductive hypothesis on $\phi$.
\item The core of the proof is the following statement, that we will show by induction on the length of $\psi$: if $\alpha \not \in \im \psi$ is such that $\psi(0) < \alpha < \psi(n)$ and $\sigma^{\psi \Lambda}(j) = \alpha$, then $\sigma^{\psi \Lambda}(j+1) = \alpha+1$ and $a^{\psi \Lambda}_j = m-1$.

If $\psi = \Id$ has length zero there is nothing to show, since $\psi$ is surjective. Otherwise, write $\psi = d_k \phi$ for $\phi$ of shorter length. We distinguish three cases for $k$, as usual. If $k=0$, then
$$ \alpha \not \in \im d_0 \phi, \ d_0 \phi(0) < \alpha < d_0 \phi(n) \ \ \Leftrightarrow \ \ \alpha -1 \not \in \im \phi, \ \ \  \phi(0) <  \alpha-1 < \phi(n)\ . $$
Note that $\alpha -1 > \phi(0) \ge 0$ implies $\alpha > 1$. 
Then
\begin{align*}
    &\alpha = \sigma^{\phi \Lambda}(j-1)+1 \ \ \Rightarrow \ \ \alpha -1 = \sigma^{\phi \Lambda}(j-1) \ \ \\
 \xRightarrow{\text{i.hp}} \ \ &\sigma^{\phi \Lambda}(j) = \alpha \ \ \Rightarrow \ \ \sigma^{d_0 \phi \Lambda}(j+1) = \sigma^{\phi \Lambda}(j)+1 = \alpha +1 \ ,
\end{align*}
and also
$$ a^{d_0 \phi \Lambda}_{\alpha} = a^{ \phi \Lambda}_{\alpha -1} = m-1 \ .$$
If $0 < k < \ell+1$, then 
$$\alpha \not \in \im d_k\phi , \ \ \ d_k \phi(0) < \alpha < d_k \phi(n) \ \ \Rightarrow \  \ s_k(\alpha) \not \in \im \phi \textrm{ or } \alpha = k, \phi(0) < s_k(\alpha) < \phi(n) \ .$$
Let us denote by $\gamma$ the index such that $\sigma^{\phi \Lambda}(\gamma) = k$. We have to distinguish three subcases: $\alpha=k-1, \alpha = k, \alpha \neq k,k-1$. If $\alpha = k-1$, by inductive hypothesis $\sigma^{\phi \Lambda}(j+1) = \alpha+1 = k$, which implies $\gamma = j+1$. This means
$$ \sigma^{d_k \phi \Lambda}(j+1) = \sigma^{d_k \phi \Lambda}(\gamma) = k = \alpha+1 \ ,$$
because of the exception in the transformed permutation definition. Also,
$$ a^{ d_k \phi \Lambda}_{\gamma-1} = a^{\phi \Lambda}_{s_{\gamma} (\gamma-1)} =a^{\phi \Lambda}_{ \gamma -1} = m-1 \ ,$$
because of the inductive hypothesis and the observation $\sigma^{\phi \Lambda}(\gamma -1 ) = \sigma^{\phi \Lambda}(j) = \alpha$.

If $\alpha = k$, then $j=\gamma$ and we have
$$ \sigma^{d_k \phi \Lambda}(\gamma+1) = d_k \sigma^{\phi \Lambda}s_{\gamma}(\gamma+1) = d_k \sigma^{\phi \Lambda}(\gamma) = d_k(k) = k+1 = \alpha+1 \ ,$$
$$ a^{d_k \phi \Lambda}_{\gamma} = m-1 \ ,$$
because of the exception in the definition of the transformed depth-index.

Finally, if $\alpha \neq k, k-1$ we have
$$ \alpha = \sigma^{d_k \phi \Lambda}(j) = d_k \sigma^{ \phi \Lambda}( s_{\gamma}j) \ \ \Rightarrow \ \ s_k \alpha = \sigma^{\phi \Lambda}( s_{\gamma} j)  \ \ \xRightarrow{\text{i.hp}} \ \ \sigma^{\phi \Lambda}(s_{\gamma}j +1) = s_k \alpha +1\ .$$
In particular, $\sigma^{\phi \Lambda}(s_{\gamma} j) = s_k \alpha$ and $s_k \alpha$ respect the hypothesis, thus
$$ a^{d_k \phi \Lambda}_{j} = a^{ \phi \Lambda}_{s_{\gamma} j} = m-1 \ .$$
Let us go on with the examination of the permutations. Since $\alpha \neq k$, it is easy to see that $s_k \alpha +1 = s_k(\alpha +1)$. Note that $j \neq \gamma$, otherwise $\alpha = k$. Thus we also have $s_{\gamma} j +1 = s_{\gamma}(j+1)$. We deduce that
$$ \sigma^{\phi \Lambda}(s_{\gamma}(j+1)) = s_k(\alpha+1) \ . $$
Note that $\alpha+1 \neq k$ because of our hypothesis. Applying $d_k$ and using the identity $d_k s_k(x) = x$ for $x \neq k$, we have that
$$ \sigma^{d_k \phi \Lambda}(j+1) = d_k \sigma^{\phi \Lambda}(s_{\gamma}(j+1) ) = d_k s_k(\alpha+1) = \alpha+1\ ,$$
as desired. 

At last, consider the case $k= \ell+1$. We have that
$$ \alpha \not \in \im d_{\ell+1} \phi, \ \ d_{\ell+1} \phi(0) < \alpha < d_{\ell+1} \phi(n) \ \ \Leftrightarrow \ \ \alpha \not \in \im \phi, \ \ \phi(0) < \alpha < \phi(n) \ .$$
The last condition implies $\alpha \le \phi(n)-1 \le \ell-1$.  Thus we have
$$ \sigma^{d_{\ell+1} \phi \lambda}(j) = \alpha \ \ \xRightarrow{\alpha \neq \ell+1} \ \ \sigma^{\phi \Lambda}(j) = \alpha \ .$$
Regarding the depth-index, if $j=\ell$ the thesis is automatically satisfied. Otherwise
$$ a^{d_{\ell+1} \phi \Lambda}_j = a^{ \phi \Lambda}_j = m-1\ . $$
Regarding the permutation, by the inductive hypothesis the above equation implies $\sigma^{\phi \lambda}(j+1) = \alpha+1$. Again, since $\alpha+1 \neq \ell+1$, we deduce $\sigma^{d_{\ell+1} \phi \Lambda}(j+1) = \alpha+1$, as desired.

Now that we have shown the core statement, we are ready to prove point $(2)$, that is for all $\psi(0) < \alpha < \beta \le \psi(n)$ such that $\{\alpha, \ldots, \beta-1\} \subset (\im \psi )^c$, it holds
$$ (\sigma^{\psi \Lambda} )^{-1}(\beta) - (\sigma^{\psi \Lambda} )^{-1}(\alpha) = \beta - \alpha \ ,$$
$$  \textrm{for all } (\sigma^{\psi \Lambda})^{-1}(\alpha) \le t \le (\sigma^{\psi \Lambda})^{-1}(\beta) -1: \ \ \ a^{\psi \Lambda}_t = m-1 \ ,$$
We use induction on $\beta - \alpha$. If $\beta - \alpha = 1$, the hypothesis reformulate as $\psi(0) < \alpha < \psi(n)$ and $\alpha \not \in \im \psi$. Let us denote by $j= (\sigma^{\psi \Lambda} )^{-1}(\alpha) $. By the core statement, we have that $\sigma^{\psi \Lambda}(j+1)= \alpha +1 = \beta$, that yields
$$ (\sigma^{\psi \Lambda} )^{-1}(\beta) - (\sigma^{\psi \Lambda} )^{-1}(\alpha) = (j+1) - j = \beta - \alpha\ .$$
For the general case, write $\beta = \beta'+1$. The hypothesis are still verified for the pairs $\alpha < \beta'$ and $\beta' < \beta$. By the inductive hypothesis we have
$$ (\sigma^{\psi \Lambda} )^{-1}(\beta') - (\sigma^{\psi \Lambda} )^{-1}(\alpha) = \beta' - \alpha \ , $$
while for the base case (since $\beta'-\beta = 1$) we have
$$ (\sigma^{\psi \Lambda} )^{-1}(\beta) - (\sigma^{\psi \Lambda} )^{-1}(\beta') = \beta - \beta' \ .$$
Summing the two equations we obtain the statement.

The deduction on depth-indices is even simpler. Because of what we have shown about the permutation, we have that $t= (\sigma^{\psi \Lambda} )^{-1}(\gamma)$ for some $\gamma \in \{\alpha, \ldots, \beta-1\} \subset (\im \psi)^c$. Because of the core statement, we conclude that $a^{\psi \Lambda}_t = m-1$.

\item We decompose an arbitrary $\psi$ as $\psi' d_{n+p} \ldots d_{n+1}$, where $p= \ell- \psi(n)$. Here $\psi'$ is defined as
$$ \psi'(t) =  \left \{ \begin{array}{ll} 
\psi(t) , & \textrm{if } t \le n \\
\psi(n)-n+t, & \textrm{if} t \ge n
\end{array} \right. $$
This is well defined in the common point of definition and it gives an increasing function. Also for $t \in [n]$
$$ \psi' d_{n+p} \ldots d_{n+1}(t) =  \psi'(t) = \psi(t)\ .$$
We now show the thesis by induction on the length of $\psi'$. In the base case $\psi'=\Id$, we use a nested induction on $p\ge 0$. If $p=0$, there is nothing to show since there is no $s \in [n]$ with $s \ge \psi(n)+1= n+1$. Otherwise, we have that
$$ \sigma^{d_{n+p} d_{n+p-1} \ldots d_{n+1} \Lambda}(s) =  \left \{ \begin{array}{ll} 
\textrm{if } s=n+p: & n+p \\
\textrm{otherwise if } s \le n+p-1: & \sigma^{ d_{n+p-1} \ldots d_{n+1} \Lambda}(s) \stackrel{i.hp}{=} s
\end{array} \right. $$
$$ a^{d_{n+p} d_{n+p-1} \ldots d_{n+1} \Lambda}_s =  \left \{ \begin{array}{ll} 
\textrm{if } s=n+p-1: & m-1 \\
\textrm{otherwise: } &  a^{ d_{n+p-1} \ldots d_{n+1} \Lambda}_s \stackrel{i.hp}{=} m-1
\end{array} \right. $$

\begin{comment} SONO ARRIVATO QUI PER I DOPPI DOLLARI 
$$ ciao $$
RICOMINCIA DA QUI
\end{comment}

Let us examine the inductive step on the length of $\psi'$. Suppose $\psi' = d_k \phi$ for some $\phi$ of shorter length. Note that we must have $k < \psi(n)$, otherwise
$$ \psi(n) = \psi'(n) = d_k \phi(n) \le d_k (\psi(n)-1) = \psi(n) -1 \ .$$
The inequality $\phi(n) \le \psi(n)-1$ comes from the fact that $\phi : [n+p] \to [\ell-1]$ is strictly increasing, thus
$$ \phi(n) \le \phi(n+1) -1 \le \ldots \phi(n+p)-p \le \ell-1-p = \psi(n)-1 \ .$$
We now have two cases. Denote by $\tilde{\phi} = \phi d_{n+p} \ldots d_{n+1}$. If $k=0$, then for $s \ge \psi(n)+1$
$$ \sigma^{d_0 \tilde{\phi} \Lambda}(s) = \sigma^{\tilde{\phi}\Lambda}(s-1) +1 \ , $$
$$ a^{d_0 \tilde{\phi} \Lambda}_s = a^{\tilde{\phi}\Lambda}_{s-1}\ . $$
Recall that $ \tilde{\phi}(n) = \phi(n) \le \psi(n)-1$, thus $s-1 \ge \psi(n) \ge \tilde{\phi}(n)+1$. Applying the inductive hypothesis we get 
$$ \sigma^{\tilde{\phi} \Lambda}(s-1) +1 = (s-1)+1 = s \ ,$$
$$ a^{\tilde{\phi} \Lambda}_{s-1}= m-1 \ . $$
In case $k > 0$, consider $\gamma = (\sigma^{\tilde{\phi} \Lambda})^{-1}(k)$. Then for $s \ge \psi(n)+1$
$$ \sigma^{d_k \tilde{\phi}\Lambda}(s) = d_k \sigma^{\tilde{\phi} \Lambda}(s_{\gamma}(s) ) \ , $$
$$ a^{d_k \tilde{\phi}\Lambda}_s = a^{\tilde{\phi} \Lambda}_{s_{\gamma}(s)} \ .$$
Since $k < \psi(n)$ and $\sigma^{\tilde{\phi} \Lambda} \{ \psi(n), \ldots, \ell-1\} = \{ \psi(n), \ldots, \ell-1\}$ by inductive hypothesis, we must have that $\gamma < \psi(n)$ too. Thus $s_{\gamma}(s) = s-1$ and we get, applying the inductive hypothesis
$$d_k \sigma^{\tilde{\phi} \Lambda}(s-1) = d_k(s-1) = s \ ,$$
$$ a^{\tilde{\phi} \Lambda}_{s-1} = m-1\ , $$
since $s-1 \ge \psi(n) > k$.
\end{enumerate}
\end{proof}

\subsection{Relationship between positions and labels}
In order to work with the cosimplicial action, we have to somehow explicit its effect on positions; this is the purpose of the twisted morphisms we introduce in the following definition, illustrated in figure \ref{hair-blocks-position-img}.  The slogan is the following: the argument of $\phi$ is a point (label), the argument of $\phi^{\Lambda}$ is a position.

\begin{definition} Consider $\psi: [p] \to [q]$, $\Lambda \in \FNP_m(p), \Gamma \in \FNP_m(q)$ such that $\psi \Lambda = \Gamma$. Let us denote by $\sigma^{\Lambda}, \sigma^{\Gamma}$ the label permutations associated to the two trees. Define the morphism $\psi$ twisted by $\Lambda$, denoted $\psi^{\Lambda}$, in the following way: 
$$ \psi^{\Lambda}(s) = \left\{ \begin{array}{ll} (\sigma^{\Gamma})^{-1}\psi(\sigma^{\Lambda}(s)) & \textrm{if } s=1, \ldots,p \\
\psi(0) & \textrm{for } s=0 \end{array} \right. $$
\end{definition}
Since twisting a morphism involves several definitions, we take the time to go through an explicit example while giving some down-to-earth tips to simplify the calculations.
\begin{example}\label{hair-blocks-example} Consider $\Lambda = (312, (0,1)) \in \FNP_3(3)$ and $\psi: [3] \to [9]$ as in the following picture:
\begin{figure}[H]
	\centering
	\includegraphics[scale=0.15]{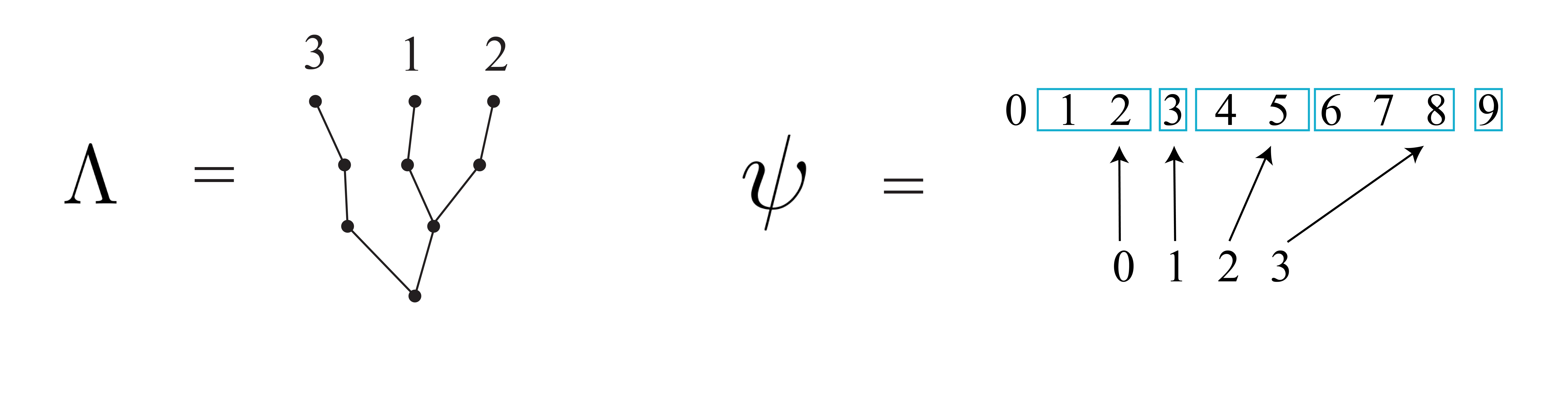}
\end{figure}
We will write an increasing function by juxtaposing its values, as we did for the permutation; in our case, $\psi = 2358$. The first step to compute $\psi^{\Lambda}$ is to find $\sigma^{\Gamma} = \psi(312)$.  Since the cosimplicial action on Fox-Neuwirth trees is defined in terms of generators, the second step is to decompose $\psi$ in terms of cofaces. Given the picture of an increasing morphism, it is easy to find its decomposition. Let us call a value in the codomain "jumped" if it is not in the image of $\psi$. Then $\psi$ can be obtained with the following recipe:
\begin{itemize}
\item Start with as many $d_0$ as the number of jumped values at the beginning;
\item Going from left to right, add a term $d_k$ for every jumped value $k \ge \psi(0)$.
\end{itemize}
Since composed functions are written from right to left, we will get $\psi = d_9 d_7 d_6 d_4 d_0^2$. Let us compute the transformed permutation. Recall that, up to relabeling, $d_0$ adds a $1$ at the beginning, $d_i$ for "internal" $i$ doubles $i$, and $``d_{\textrm{last}}"$ adds the maximum label plus one at the end. Thus:
\begin{align*}
\sigma^{\psi \Lambda} = & d_9 d_7 d_6 d_4 d_0^2 (312) = d_9 d_7 d_6 d_4 (12534) = d_9 d_7 d_6 (126345) = \\
= & d_9 d_7 (1267345) = d_9 (12678345) = 126783459 \ .
\end{align*}
The effect of $\psi$ on $\Lambda$ is depicted in figure \ref{shape-tree}. Lastly, we can apply the main definition $\psi^{\Lambda} = (\sigma^{\psi \Lambda} )^{-1} \circ \psi \circ \sigma^{\Lambda}$ to get the desired result. There is a workaround to include the case $\psi^{\Lambda}(0)$ in the latter formula, that is by appending $0$ at the beginning of both permutations. With this trick, we can compute $\psi^{\Lambda}$ in one line:
\begin{align*}
\psi^{\Lambda} = & (0126783459)^{-1} \circ (2358) \circ (0312) = (0126783459)^{-1} \circ (2835) = 2578
\end{align*}
\end{example}

As you can see, the twisted morphism bears some striking similarities with the original map. First of all, it is increasing. Secondly, the first and last value remains unchanged. Some well-versed-in-sequences readers could have also noticed the following fact: if we take the difference between consecutive values, the original sequence yields $1,2,3$, while the new one yields $3,2,1$. 

Unfortunately, directly applying the definition makes these similarities seem like a coincidence. Using the Shape-Tree Lemma, together with Lemma \ref{twisted-example}, we will shed light on the apparent simple behaviour of $\psi^{\Lambda}$. Indeed, we will later prove the following facts:
\begin{corollary} \label{increasing} $\psi^{\Lambda}: [n] \to [\ell]$ is increasing for all $\psi: [n] \to [\ell]$ and $\Lambda \in \FNP_m(n)$.
\end{corollary}

\begin{corollary} \label{extrema} For all $\psi:[n] \to [\ell]$ and $\Lambda \in \FNP_m(n)$ we have 
$$ \psi^{\Lambda}(0) = \psi(0), \ \ \ \ \ \psi^{\Lambda}(n) = \psi(n) \ . $$
\end{corollary}

\begin{corollary} \label{consecutive} Consider $\psi : [n] \to [\ell], \Lambda \in \FNP_m(n)$ and $\alpha, \beta \in [\ell]$, with $a,b$ such that $ \sigma^{\psi \Lambda}(a) = \alpha, \sigma^{\psi \Lambda}(b) = \beta$. Then
$$\psi(0) < \alpha < \beta \le \psi(n), \ \ \ \{\alpha, \ldots, \beta -1 \} \subset (\im \psi)^c $$
if and only if
$$\psi^{\Lambda}(0) < a < b \le \psi^{\Lambda}(n), \ \ \ \{a, \ldots, b-1 \} \subset (\im \psi^{\Lambda} )^c\ . $$
\end{corollary}
The last Corollary is a set-theoretical manifestation of our previous numerical observation: differences between consecutive values are the cardinality of consecutive jumped values (plus one).

\begin{remark} \label{hair-blocks-position}
Let us further illustrate the significance of Corollaries in the Example \ref{hair-blocks-example}. As you can see in picture \ref{hair-blocks-position-img}, $\psi$ prescribes the `hair blocks substitution' for \textit{labels}. Indeed, up to relabeling we have: 
\begin{enumerate}
\item Insert the jumped values at the beginning as `hairs' at the left of the tree;
\item For each label, substitute $p$ with $\psi(p)$, and add hairs to its left corresponding to jumped values before $\psi(p)$;
\item Insert the jumped values at the end as `hairs' at the right of the tree.
\end{enumerate}
In the same spirit, $\psi^{\Lambda}$ prescribes the `hair blocks substitution' for \textit{labels} for \textit{positions}. However, since the information of the new permutation is not contained in the twisted morphism, this recipe is less informative about labels (but equally effective in terms of hair structure). Let us make it explicit:
\begin{enumerate}
\item Insert the jumped values at the beginning as `hairs' at the left of the tree;
\item For each position, substitute the label $p$ in the place $k$ with the label in position $\psi^{\Lambda}(k)$ of the transformed tree. Then, add hairs to its left corresponding to jumped values before $\psi^{\Lambda}(k)$.
\item Insert the jumped values at the end as `hairs' at the right of the tree.
\end{enumerate}
Corollary \ref{extrema} ensures that we get the same hairs in extremal parts of the tree. Corollary \ref{increasing} guarantees that we can talk about jumps in a way that is similar to the original case. At last, the correspondence under permutation of jumps is ultimately what Corollary \ref{consecutive} describes. 
\begin{figure}[H]
	\centering
	\includegraphics[scale=0.15]{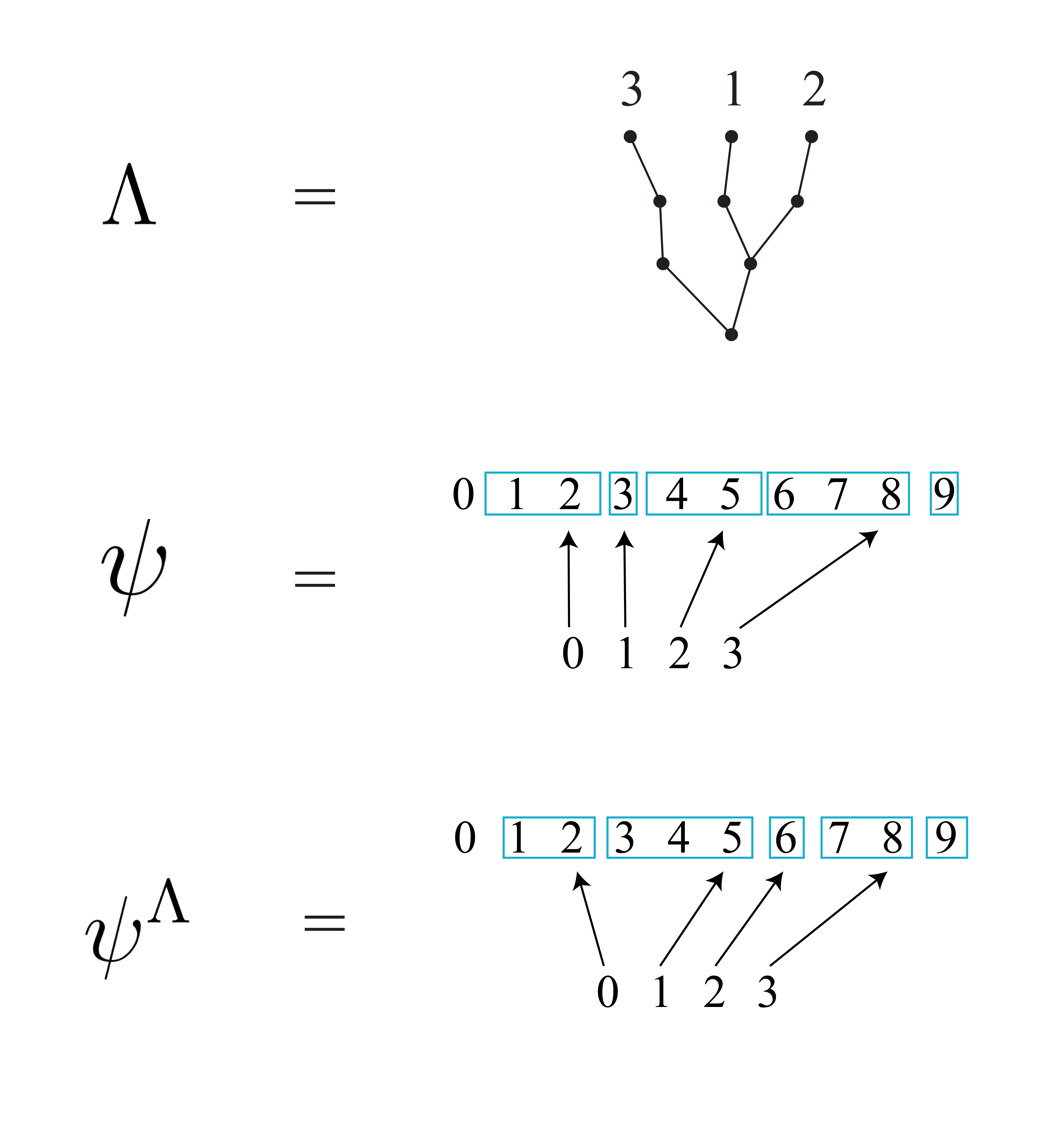}
	\caption{The effect on \enquote{hair blocks} of twisting a morphism}
	\label{hair-blocks-position-img}
\end{figure}

\end{remark}

We now proceed to the proof of such Corollaries, starting from two preliminary lemmas. The first regards the inverses of transformed permutations.
\begin{lemma} Given $\Lambda \in \FNP_m(n)$, we have
$$\hspace{-2cm} \boxed{0 < i < n+1}: \ \ \  d_i(\sigma^{\Lambda})^{-1}(t) = \left\{ \begin{array}{ll}
        (\sigma^{\Lambda})^{-1}(t),  & \textrm{if }  t=i\\
       d_{\alpha} (\sigma^{\Lambda})^{-1}( s_i(t))) , & \textrm{ otherwise} 
    \end{array} \right.
    $$
$$ d_0(\sigma^{\Lambda})(k) = \left\{ \begin{array}{ll}
        1,  & \textrm{if }  k =1\\
       \sigma^{\Lambda}(k-1)+1 , & \textrm{ otherwise} 
    \end{array} \right.
    $$
$$ d_{n+1}(\sigma^{\Lambda})^{-1}(t) = \left\{ \begin{array}{ll}
        n+1,  & \textrm{if }  t =n+1\\
       (\sigma^{\Lambda})^{-1}(t) , & \textrm{ otherwise} 
    \end{array} \right.
    $$
\end{lemma}
\begin{proof} Recall the definitions of $\sigma^{d_i \Lambda} = d_i(\sigma^{\Lambda})$ from Definition \ref{fn-cosimplicial}. For $0 < i < n+1$:
$$ d_i(\sigma^{\Lambda})(k) = \left\{ \begin{array}{ll}
        i,  & \textrm{if }  k =\alpha\\
       d_i (\sigma^{\Lambda}( s_{\alpha}(k))) , & \textrm{ otherwise} 
    \end{array} \right.
    $$
where $\alpha = (\sigma^{\Lambda})^{-1}(i)$. It follows that
$$ d_i(\sigma^{\Lambda})^{-1}(t) = \left\{ \begin{array}{ll}
        (\sigma^{\Lambda})^{-1}(t),  & \textrm{if }  t=i\\
       d_{\alpha} (\sigma^{\Lambda})^{-1}( s_i(t))) , & \textrm{ otherwise} 
    \end{array} \right.
    $$
Indeed, beside the case $t=i$ that follows from the case $k=\alpha$, we have
    $$ d_{\alpha} (\sigma^{\Lambda})^{-1} ( s_i(  d_i (\sigma^{\Lambda}( s_{\alpha}(k))))) = d_{\alpha} (\sigma^{\Lambda})^{-1}  (\sigma^{\Lambda}( s_{\alpha}(k)))) = d_{\alpha} s_{\alpha}(k) = k \ ,$$
where in the last passage we used $k \neq \alpha$. The other composition is analogous. 

For $i=0$ we have
$$ d_0(\sigma^{\Lambda})(k) = \left\{ \begin{array}{ll}
        1,  & \textrm{if }  k =1\\
       \sigma^{\Lambda}(k-1)+1 , & \textrm{ otherwise} 
    \end{array} \right.
    $$
Thus, since in the general case it is a composition of a bijection and two translations, we have:
$$ d_0(\sigma^{\Lambda})^{-1}(t) = \left\{ \begin{array}{ll}
        1,  & \textrm{if }  t =1\\
       (\sigma^{\Lambda})^{-1}(t-1)+1 , & \textrm{ otherwise} 
    \end{array} \right.
    $$
Finally for $i=n+1$
    $$ d_{n+1}(\sigma^{\Lambda})(k) = \left\{ \begin{array}{ll}
        n+1,  & \textrm{if }  k =n+1\\
       \sigma^{\Lambda}(k) , & \textrm{ otherwise} 
    \end{array} \right.
    $$
Which yields
$$ d_{n+1}(\sigma^{\Lambda})^{-1}(t) = \left\{ \begin{array}{ll}
        n+1,  & \textrm{if }  t =n+1\\
       (\sigma^{\Lambda})^{-1}(t) , & \textrm{ otherwise} 
    \end{array} \right.
    $$
\end{proof} \vspace{-1cm}

The second provides an inductive way to deal with twisted morphisms.
\begin{lemma} \label{twisted-example} Given  $\Lambda \in \FNP_m(n)$, $\phi: [n] \to [n+1]$ with permutation $\sigma$ on labels, we have  
$$ d_i^{\Lambda} = \left\{ \begin{array}{ll}
       d_{\sigma^{-1}(i)},  & \textrm{if }  0 < i<  n+1\\
      	d_i, & \textrm{ otherwise} 
    \end{array} \right. $$
Furthermore, if $f:[n] \to [p]$ and $g: [p] \to [q]$, the twisting of the composition is computed as
$$ (fg)^{\Lambda} = f^{g \Lambda} g^{\Lambda} \ .$$
\end{lemma}
\vspace{-0.5cm}
\begin{proof} 
Let us start with $0< i < n+1$. For $x \neq 0$ we have
$$ \begin{array}{ll}
d_i^{\Lambda}(x) & = (\sigma^{d_i \Lambda})^{-1} d_i \sigma^{\Lambda}(x) =\\
& = d_{\alpha} (\sigma^{\Lambda})^{-1} s_i d_i \sigma^{\Lambda}(x) = \\
& = d_{\alpha} (\sigma^{\Lambda})^{-1} \sigma^{\Lambda}(x) = \\
& = d_{\alpha}(x) \ ,
\end{array} $$
while for $x=0$ we easily have $d_i^{\Lambda}(0) = 0 = d_{\alpha}(0)$. Let us go on to the case $i=0$. If $x\neq 0$
$$ \begin{array}{ll}
d_0^{\Lambda}(x) & = (\sigma^{d_0 \Lambda})^{-1} d_0 \sigma^{\Lambda}(x) \\
&  = 1+ (\sigma^{\Lambda})^{-1}( d_0 \sigma^{\Lambda}(x)-1) =  \\
& = 1+ (\sigma^{\Lambda})^{-1}(\sigma^{\Lambda}(x)) =  \\
& = x+1 = d_0(x) \ ,
\end{array} $$
and $d_0^{\Lambda}(0) = d_0(0)$ by definition. We conclude analysing $i=n+1$:
$$ \begin{array}{ll}
d_{n+1}^{\Lambda}(x) & = (\sigma^{d_{n+1} \Lambda})^{-1} d_{n+1} \sigma^{\Lambda}(x) \\
&  = (\sigma^{\Lambda})^{-1}\sigma^{\Lambda}(x) =  \\
& = x = d_{n+1}(x) \ ,
\end{array} $$
with $d_{n+1}^{\Lambda}(0)=d_{n+1}(0)$ by definition. We now pass to the composition formula. For $x \neq 0$:
$$ \begin{array}{ll}
f^{g\Lambda} g^{\Lambda}(x) & = [(\sigma^{fg\Lambda})^{-1} f \sigma^{g \Lambda} ][ (\sigma^{g \Lambda})^{-1} g \sigma^{\Lambda}](x) = \\
&  = (\sigma^{fg\Lambda})^{-1} f g \sigma^{\Lambda}(x) =  \\
& = (fg)^{\Lambda}(x) \ ,
\end{array} $$
while for $x=0$, we have to distinguish the cases $g(0) = 0$ and $g(0) > 0$. If $g(0) > 0$, we invoke the shape-tree Lemma ( \ref{twist-initial}):
$$ \begin{array}{ll} 
f^{g\Lambda} g^{\Lambda}(0) & = f^{g\Lambda} g(0) = (\sigma^{fg \Lambda} )^{-1} f \sigma^{g \Lambda} g(0) = \\
& = (\sigma^{fg \Lambda})^{-1} f g(0) = fg(0) = \\
& = (fg)^{\Lambda}(0)  \ .
\end{array}$$
Otherwise, we simply have
$$ f^{g\Lambda} g^{\Lambda}(0) = f^{g\Lambda} g(0) = f^{g\Lambda}(0) = f(0) = fg(0) = (fg)^{\Lambda}(0)\ . $$
\end{proof}

We are finally ready to prove the Corollaries stated at the beginning of the section.
\begin{proof} \textit{(Corollary \ref{increasing})}  By induction on the length of $\psi$. If $\psi= \Id$, then $\psi^{\Lambda} = \Id$. Otherwise, write $\psi= d_{k} \psi'$. By the above lemma, we have that $\psi^{\Lambda} = d_k^{\psi' \Lambda} (\psi')^{\Lambda}$. The first function is increasing by the inductive hypothesis, while $d_k^{\psi' \Lambda} = d_{k'}$ for some $k'$ is increasing by definition.
\end{proof}

\begin{proof} \textit{(Corollary \ref{extrema})} The first one is the very definition of $\psi^{\Lambda}(0)$. Regarding the second one, we use induction on the length of $\psi$. If $\psi = \Id$, then $\psi^{\Lambda} = \Id$ and we are done. If $\psi= d_k \phi$ for some $\phi$ of shorter length, we distinguish two cases. In case $k=0$ or $k = \ell+1$, we have 
$$ \psi^{\Lambda}(n) = (d_k \phi)^{\Lambda}(n) = d_k^{\phi \Lambda} \phi^{\Lambda}(n) = d_k \phi(n) = \psi(n) \ .$$
Otherwise, let $\alpha= (\sigma^{\phi \Lambda})^{-1}(k)$. Since $\sigma^{\phi \Lambda} \{ \phi(n)+1, \ldots, \ell-1\} = \{\phi(n)+1, \ldots, \ell-1\}$, we have that $\alpha > \phi(n)$ if and only if $k > \phi(n)$. It follows that 
$$\psi^{\Lambda}(n) =(d_k \phi)^{\Lambda}(n) = d_{\alpha} \phi^{\Lambda}(n) = d_{\alpha} \phi(n) = d_k \phi(n) = \psi(n)\ .$$
\end{proof}

\begin{proof} \textit{(Corollary \ref{consecutive})} We firstly concentrate on the top-to-bottom deduction. For all $p \le \beta - \alpha$, we have that $\{\alpha, \ldots, \alpha+p-1\} \subset (\im \psi)^c$, thus by the above lemma
$$ (\sigma^{\psi \Lambda})^{-1}(\alpha+p) = (\sigma^{\psi \Lambda})^{-1}(\alpha) + p \ .$$
Suppose by contradiction that $\psi^{\Lambda}(k) = (\sigma^{\psi \Lambda})^{-1}(\alpha)+p$ for some $p \le (\sigma^{\Lambda})^{-1}(\beta)-(\sigma^{\Lambda})^{-1}(\alpha)-1$. Firstly, note that by the above lemma
$$p \le (\sigma^{\Lambda})^{-1}(\beta)-(\sigma^{\Lambda})^{-1}(\alpha)-1 = \beta - \alpha -1 \ ,$$
so that we can rewrite the equation as
$$\psi^{\Lambda}(k) = (\sigma^{\psi \Lambda})^{-1}(\alpha+p) \ .$$
Using $\psi^{\Lambda} =  (\sigma^{\psi \Lambda})^{-1} \psi \sigma^{\Lambda} $ we get
$$ \psi( \sigma^{\Lambda}(k) ) = \alpha +p \le \beta - 1 \ ,$$
which is a contradiction by the hypothesis $\{\alpha, \ldots, \beta-1 \} \subset (\im \psi)^c$.

Now let us show that  $\psi^{\Lambda}(0) < a < b \le \psi^{\Lambda}(n)$. The fact that $a < b$ follows from the discussion above. Regarding  the inequalities $\psi^{\Lambda}(0) < a$ and $b \le \psi^{\Lambda}(n)$, recall that
$$ \psi^{\Lambda}(0) = \psi(0), \ \ \ \ \psi^{\Lambda}(n) = \psi(n) $$
from Corollary \ref{extrema}. The result follows since
$$ \sigma^{\psi \Lambda}\{1, \ldots, \psi(0)\} = \{1, \ldots, \psi(0)\}, \ \ \ \ \sigma^{\psi \Lambda}\{\psi(n)+1, \ldots, \ell \} = \{\psi(n)+1, \ldots, \ell\}\ . $$

We pass to examine the bottom-to-top implication. Note that $\alpha > \psi(0)$ and $\beta \le \psi(n)$ easily follows from Corollary \ref{extrema} and Lemma \ref{twist-initial}. Consider
$$ p = \max \{ p' > \alpha: \{\alpha, \ldots, p'-1\} \subset (\im \psi)^c \} \ .$$
Note that the set is non-empty, since $\alpha+1$ is a valid element: if by contradiction $\alpha = \psi(\alpha')$, then 
$$ a = (\sigma^{\psi \Lambda} )^{-1}(\alpha) = \psi^{\Lambda}(a'), \ \ \ a' = (\sigma^{\Lambda})^{-1}(\alpha') \ , $$
against the hypothesis. Now suppose by contradiction that $p-\alpha < b-a$. Applying the usual lemma to $\alpha < p$ we get
$$ (\sigma^{\psi \Lambda})^{-1}(p) - (\sigma^{\psi \Lambda})^{-1}(\alpha) = p-\alpha \ \ \ \Rightarrow \ \ \ (\sigma^{\psi\Lambda})^{-1}(p) = a+(p-\alpha) \in \{a, \ldots, b-1\} \ . $$
As above, this implies that $p \not \in \im \psi$. But then $(p+1) \in \{p' \ge \alpha: \{\alpha, \ldots, p'-1\} \subset (\im \psi)^c\}$, contradicting the maximality of $p$. We finally want to show that $p \ge \beta$, concluding that 
$$\{\alpha, \ldots, \beta-1\} \subset \{\alpha, \ldots, p-1\} \subset  (\im \psi)^c \ .$$
Consider $p_* = \alpha + b-a$. By the above observation, $p_* \le p$. We can then apply the lemma and conclude
$$ (\sigma^{\psi \Lambda})^{-1}(p_*) =(\sigma^{\psi \Lambda})^{-1}(\alpha) + p_* - \alpha = a+ p_*-\alpha = b $$
that is $p_* = \beta$, and we are done.
\end{proof}

\section{Spaces of Weighted Trees}
The cosimplicial structure we defined on $\BZ_m(n)$ seems to mimic well the cosimplicial action on Kontsevich spaces. In the latter space, the action of $d_i$ doubles a point infinitesimally close in the direction of $e_1$. On Fox-Neuwirth trees, $d_i$ doubles the label $i$ inserting $i+1$ as close as possible to the right of it. 

However, when we get to the concrete $\BZ_m(n)$ inside configurations (see Section \ref{configurations}), we see that in $v(d_i \Gamma)$ the $(i+1)$-th point has distance $1$ from the $i$-th point. This is not even close to infinitesimal. 

The idea is to allow \textit{weights} on the branches of a Fox-Neuwirth tree, that prescribe the distance between consecutive points. As an example, we would like the weighted tree
\begin{figure}[H]
	\centering
	\includegraphics[scale=0.25]{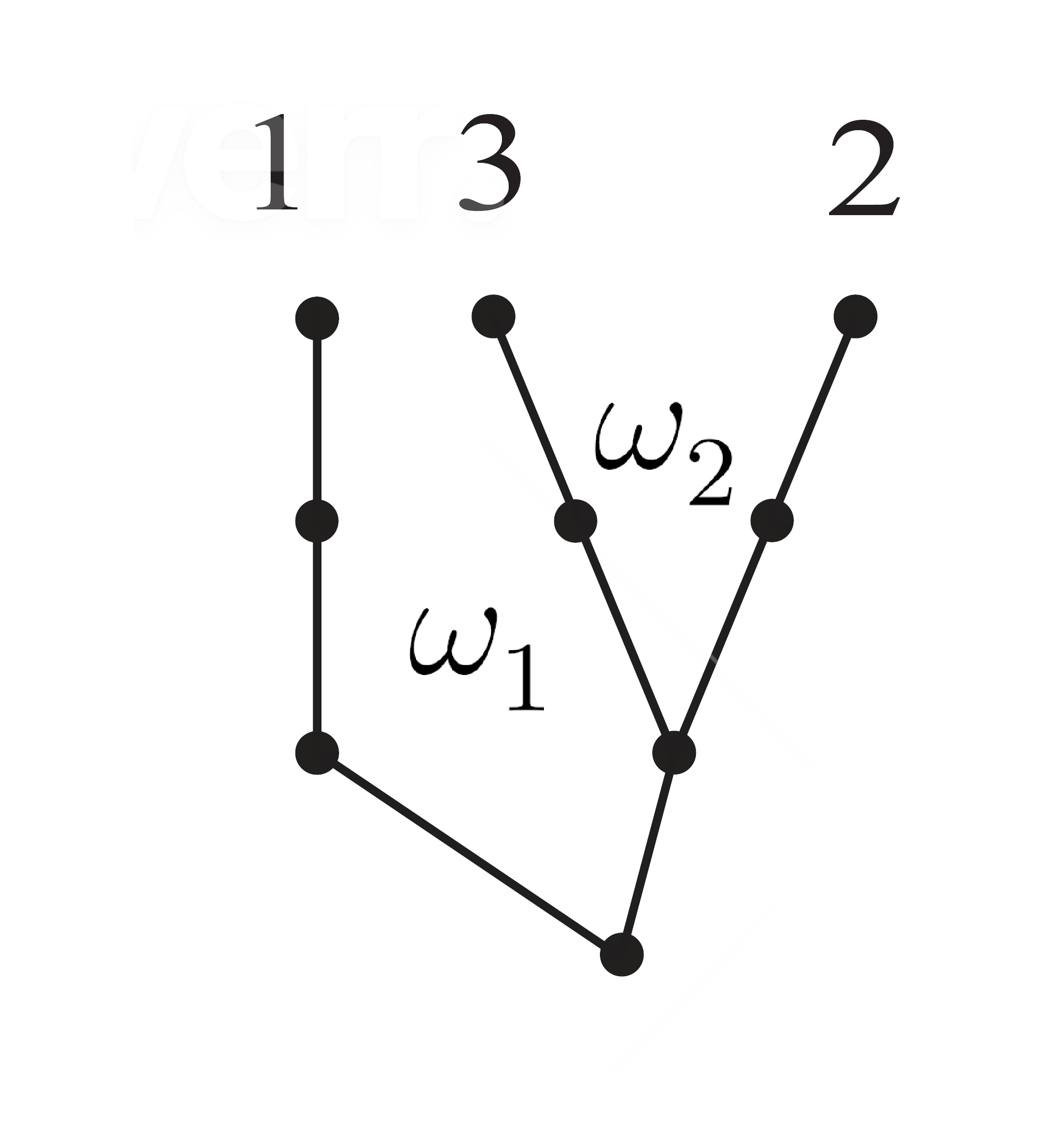}
\end{figure}
to correspond to a point $\vect{x} \in \Conf(1<_03<_1 2)$ such that $d(x_3, x_1) = \omega_1, d(x_2, x_3) = \omega_2$.
\subsection{Twisted geometric realization} \label{twisted-geo}
We introduce here a variant of the geometric realization. The first difference is that whichever (functorial) space can be used in place of simplices $\Delta^n$, but most relevantly the employed space can be different between cells of the same dimension. Firstly, in order to encode the functoriality relations, let us recall a purely categorical construction.

\begin{definition} Let $\sse{X}$ be a simplicial set. The category of elements of $\sse{X}$, denoted $\elm{X}$, is defined in the following way:
\begin{itemize}
\item Its objects are pairs $(n,x)$, where $x \in \sse{X}_n$;
\item There is a morphism $f: (n,x) \to (m,y)$ for each map $f: [n] \to [m]$ in $\Delta$ such that $\sse{X}(f)(y) = x$.
\end{itemize}

Let us notice that there is a functor $\elm{X} \to \Delta$, given by projecting on the first component. Fibers of such projection are discrete sets. 
\end{definition}
\begin{remark} Visually, the category of elements is a sort of "blow up" of the category $\Delta$, where each $[n]$ is replaced by all the elements $x \in \sse{X}_n$, and maps are arranged so that they keep track of the (contravariant) action of $\sse{X}(f)$, for all $f$ in $\Delta$. 
\end{remark}

\begin{figure}[H]
	\centering
	\includegraphics[scale=0.75]{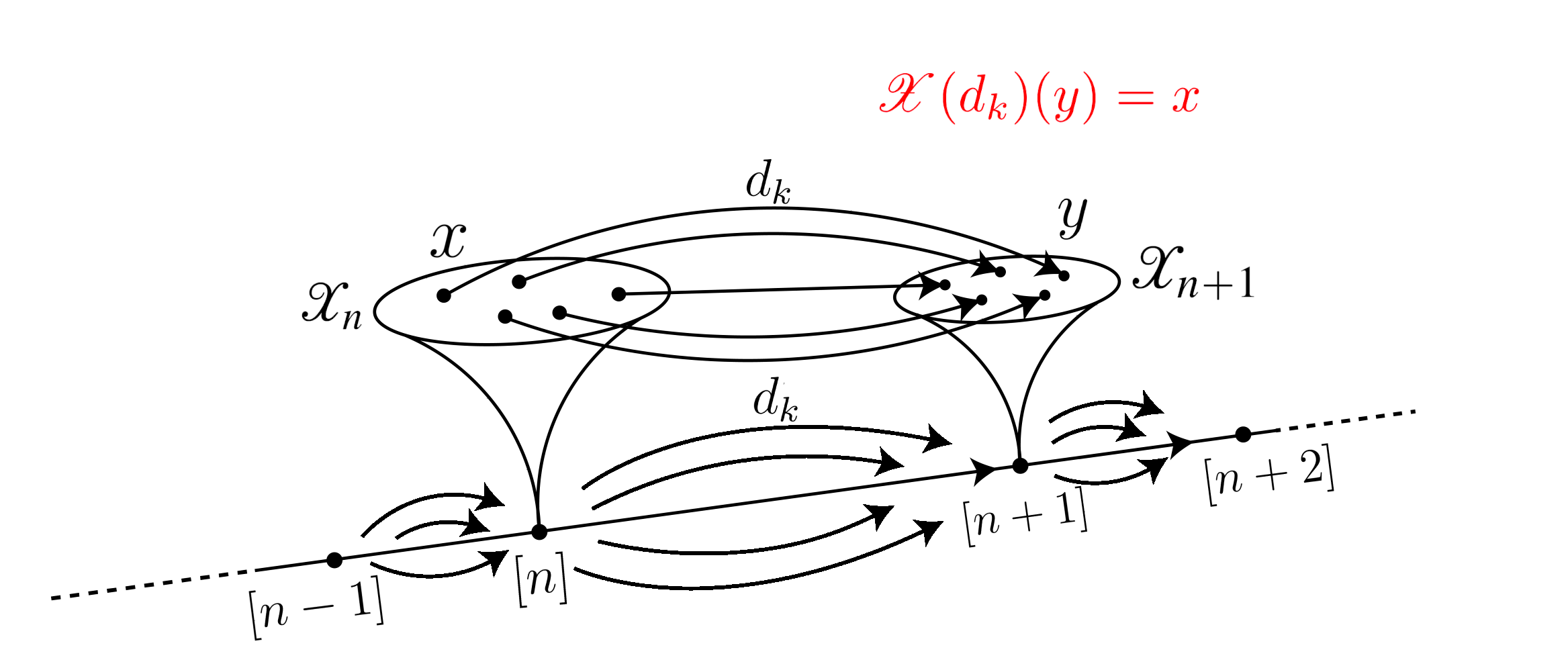}
	\caption{A blow-up-like illustration of the category of elements}
	\label{blow-up}
\end{figure}
From now on, we \textbf{drop the first component} of $(n,x) \in \elm{X}$ and simply write $x \in \sse{X}$ or $x \in \sse{X}_n$. We are ready to define a collection of spaces which can replace the ordinary cells $\Delta^n$.
\begin{definition} A \textbf{twist} over a simplicial set $\sse{X}$ is a functor $\omega: \elm{X}^{op} \to \Ttop$. 
\end{definition}
The definition of twisted geometric realization follows easily.

\begin{definition} Given a twist $\omega$ over a simplicial set $\sse{X}$, its \textit{twisted geometric realization} is defined as
$$ | \sse{X} |_{\omega}:= \colim_{x \in \elm{X}} \omega(x)\ .$$
\end{definition}
The twisted geometric realization admits an explicit description that is formally similar to the classical one.
\begin{lemma} \label{explicit-tr} For $\sse{X}$ a simplicial set and $\omega$ a twist on $\sse{X}$, we have:
$$ |\sse{X}|_{\omega} \simeq 	\frac{\bigsqcup_{x \in \sse{X} } \omega(x) } { (\partial_ix,\alpha) \sim (x, \omega(\partial_i) \alpha), (\sigma_j x,\alpha) \sim (x, \omega(\sigma_j) \alpha)} \ . $$
\end{lemma}
\begin{proof} Colimits in $\Ttop$ are calculated by endowing the corresponding colimit in $\textbf{Set}$, which is given by union and quotient, with the final topology with respect to the defining cocone. 
\end{proof}

\begin{remark} The differences of formula \ref{explicit-tr} with the classical one are the following:
\begin{itemize}
\item The glued pieces are not "homogeneous", since they depend on $x \in \sse{X}_n$ and not only on its degree;
\item The glued pieces are allowed to be any space, not just the standard cells.
\end{itemize}
\end{remark}
Let us show this is a generalization of the classical geometric realization.
\begin{lemma} \label{standard-twist} Let $\sse{X}$ be a simplicial set, and let $\textrm{std}$ be the twist defined by composing the projection $\Sigma_{\sse{X}} : \elm{X} \to \Delta$ with the standard cosimplicial object 
$$\Delta_{\bullet} : \Delta \to \Ttop, \ \ \ \  [n] \mapsto \Delta^n \ .$$
Then
$$ |\sse{X}|_{\textrm{std} } \simeq | \sse{X} | \ .$$

\end{lemma}
\begin{proof} This follows directly from Lemma \ref{explicit-tr}, since
\begin{align*}
    |\sse{X}|_{\textrm{std}} & \simeq 	\frac{\bigsqcup_{x \in \sse{X} } \textrm{std}(x) } { (\partial_ix,\alpha) \sim (x, \textrm{std}(\partial_i) \alpha), (\sigma_j x,\alpha) \sim (x, \textrm{std}(\sigma_j) \alpha)} \\
    & \simeq \frac{\bigsqcup_{n \in \mathbb{N} }  \sse{X}_n \times \Delta^n} { (\partial_ix,\alpha) \sim (x, \partial_i \alpha), (\sigma_j x,\alpha) \sim (x, \sigma_j \alpha)} = |\sse{X}|
\end{align*}

Alternatively, note that 
$$ | \sse{X} |_{\textrm{std} } := \colim_{x \in \elm{X}} \Delta_{\bullet}(\Sigma(x)) \simeq \sse{X} \otimes \Delta_{\bullet} \simeq \int^{ n \in \Delta} \sse{X}_n \cdot \Delta^n = |\sse{X} | \ ,$$
where $\sse{X} \otimes \Delta_{\bullet}$ is a notation for the weighted colimit of $\Delta_{\bullet}: \Delta \to \Ttop$ weighted by $\sse{X}: \Delta^{op} \to \Ttop$. The first isomorphism is \cite{kelly}, 3.34, while the second is \cite{Fosco}, 4.10. The third is the direct definition of geometric realization.
\end{proof}

Let us motivate the use of the word "twisted realization" for this geometric construction with further examples.
\begin{example} Consider a simplicial set $\sse{X}$ and a LCH (Locally Compact Hausdorff) space $Z$. Consider the twist $\omega_Z(\sigma) = \textrm{std}(\sigma) \times Z$. Since $Z$ is LCH \cite{colebunders}, the functor $ (-) \times Z : \Ttop \to \Ttop$  has a right adjoint and thus preserves colimits. We deduce that 
$$ |\sse{X}|_{\omega_Z} \cong \colim_{\sigma \in \elm{X}} \omega_Z(\sigma) = \colim_{\sigma \in \elm{X}} (\std(\sigma) \times Z) \cong \left ( \colim_{\sigma \in \elm{X}} \std(\sigma) \right ) \times Z \cong | \sse{X} | \times Z \ .$$
In the same way in which a product is a trivial fibration, $\omega_Z$ is a trivial twist over $\std$, in the sense that there is a natural transformation $\omega_Z \to \std$ as functors $\elm{X} \to \Ttop$ that is given by projecting on the first component on all elements. 

\begin{remark}
This analogy is more than qualitative, in that we can realize any sufficiently nice space over $|\sse{X}|$ as a twisted realization. For example, consider $f: Y \to |\sse{X}|$ where $Y$ is LCH. Set $\omega_f(\sigma) = f^{-1}( |\sse{X}(\sigma) | )$. Here $|\sse{X}(\sigma)| \subset |\sse{X}|$ denotes the cell in the geometric realization corresponding to $\sigma$. Formally, $\omega_f(\sigma)$ is obtained as the following pullback:
\[\begin{tikzcd}
	{\omega_f(\sigma)} & {\std(\sigma)} \\
	Y & {|\sse{X}|} 
	\arrow[from=1-1, to=1-2]
	\arrow["f", from=2-1, to=2-2]
	\arrow[from=1-2, to=2-2]
	\arrow[from=1-1, to=2-1]
	\arrow["\lrcorner"{anchor=center, pos=0.125}, draw=none, from=1-1, to=2-2]
\end{tikzcd}\]
Since colimits are stable for pullbacks with LCH spaces (Proposition 13.19 in \cite{bredon}), we have
$$ |\sse{X}|_{\omega_f} \cong \colim_{\sigma \in \elm{X}} \omega_f(\sigma) = \colim_{\sigma \in \elm{X}} (\std(\sigma) \times_{|\sse{X}|} Y) \cong \left ( \colim_{\sigma \in \elm{X}} \std(\sigma) \right ) \times_{|\sse{X}|} Y \cong Y\ .$$
By limiting the attention to sufficiently nice spaces and twists, we conjecture that this argument could be turned into an equivalence between the category of spaces over $|\sse{X}|$ and the category of twists over $\std$.

The upshot of this remark is the following: if we "twist" the cells of geometric realization enough, we can obtain almost any space. From this perspective, the twisted geometric realization is just a convenient way of defining a (relative) space, given (hopefully simpler) local data.

\end{remark}
\end{example}
\subsubsection{Contravariance conditions}
In this section we reformulate the colimit property in a hands-on way. This will simplify the construction of maps out of twisted geometric realizations in the rest of the article.
\begin{lemma}\textbf{(Contravariance conditions for twisted geometric realizations.)} Let $\sse{X}$ be a simplicial set, $\omega$ a twist on it and $Z$ a topological space. Let also $x \in \sse{X}_{n+1}, y \in \sse{X}_n$. Then a map $g: |\sse{X}|_{F} \to Z$ amounts to maps $g_z : \omega(z) \to Z$ for all $z \in \sse{X}$ that respects the following two conditions:
\begin{itemize}
\item$\partial$-\textbf{contravariance}: if $\partial_ix= y$, then 
\[\begin{tikzcd}
	{\omega(y)} & Z \\
	{\omega(x)}
	\arrow["{g_y}", from=1-1, to=1-2]
	\arrow["{g_x}"', from=2-1, to=1-2]
	\arrow["{\omega(\partial^{op}_i)}"', from=1-1, to=2-1]
\end{tikzcd}\]
\item $\sigma$-\textbf{contravariance}: if $\sigma_j y = x$, then 
\[\begin{tikzcd}
	{\omega(x)} & Z \\
	{\omega(y)}
	\arrow["{g_y}", from=1-1, to=1-2]
	\arrow["{g_x}"', from=2-1, to=1-2]
	\arrow["{\omega(\sigma^{op}_j)}"', from=1-1, to=2-1]
\end{tikzcd}\]
\end{itemize}
\end{lemma}
\begin{proof} This is just a reformulation of the colimit property, plus the fact that $\Delta$ is generated by faces and degeneracies.
\end{proof}

Morphisms between two twisted realizations can be realized locally, in the following way.

\begin{lemma}\textbf{(Simplicial morphisms between twisted geometric realizations.)} Let $\sse{X}, \sse{Y}$ be simplicial sets and $\omega,\theta$ twists on $\sse{X}, \sse{Y}$ respectively. A map from $|\sse{X}|_{\omega}$ to $|\sse{Y}|_{\theta}$ can be constructed with the data of
\begin{itemize}
\item A simplicial map $\varphi: \sse{X} \to \sse{Y}$;
\item For all $z \in \sse{X}$, a map of topological spaces $f_z: \omega(z) \to \theta(\varphi(z))$.
\end{itemize}
That respect the contravariance conditions: for all $x \in \sse{X}_{n+1}, y \in \sse{X}_n$
\begin{itemize}
\item$\partial$-\textbf{contravariance}: if $\partial_ix= y$, then 
\[\begin{tikzcd}
	\omega(y) & \theta(\varphi(y))\\
	\omega(x) &   \theta(\varphi(x)) 
	\arrow["{f_y}", from=1-1, to=1-2]
	\arrow["{f_x}"', from=2-1, to=2-2]
	\arrow["{\omega(\partial^{op}_i)}"', from=1-1, to=2-1]
	\arrow["{\theta(\partial^{op}_i)}"', from=1-2, to=2-2]
\end{tikzcd}\]
\item $\sigma$-\textbf{contravariance}: if $\sigma_j y = x$, then 
\[\begin{tikzcd}
	\omega(y) & \theta(\varphi(y))\\
	\omega(x) &   \theta(\varphi(x)) 
	\arrow["{f_y}", from=1-1, to=1-2]
	\arrow["{f_x}"', from=2-1, to=2-2]
	\arrow["{\omega(\sigma^{op}_j)}"', from=1-1, to=2-1]
	\arrow["{\theta(\sigma^{op}_j)}"', from=1-2, to=2-2]
\end{tikzcd}\]
\end{itemize}
\end{lemma}

\begin{proof} This is an easy consequence of the previous lemma applied to $Z= |\sse{Y}|_{\theta}$ and $g_z := i_{\varphi(z)} f_z $, where $i_y : \theta(y) \to |\sse{Y}|_{\theta}$ is the map defining the colimit of $|\sse{Y}|_{\theta}$. 
\end{proof}

\subsubsection{Non-degenerate realization}
For standard geometric realization, there is the following lemma due to \cite{milnor}:
\begin{theorem} For any simplicial set $K$, the space $|K|$ is a CW-complex having one $n$-cell corresponding to each non-degenerate $n$-simplex of $K$.
\end{theorem}
Under further assumptions that we will investigate, it is possible to give an explicit construction for $|X|$ that uses only non degenerate simplices. In this section, we will explain this result in the twisted setting. Let us firstly recall a definition from \cite{non-singular}:
\begin{definition} A simplicial set is said to be \textit{non-singular} if, for any non degenerate simplex $x \in \sse{X}_n$, the associated map $\hat{x} : \Delta^n \to \sse{X}$ is injective in all degrees.
\end{definition}

Let us see a very important example of non-singular simplicial sets, which will be used later. For the definition of acyclic category, see \cite{Kozlov2008}. Let us remark that an important example of acyclic categories are posets.
\begin{example} \label{poset-sing} \label{acyclic} Let $C$ be an acyclic category. Then its nerve $\Nerve(C)$ is a non-singular simplicial set.
\end{example} 
\begin{proof} Consider a non-degenerate simplex $\sigma: [n] \to C$. If we write it as $(f_1, \ldots, f_n)$, where $f_i : x_{i-1} \to x_i$ for $x_i = \sigma(i)$, the non-degeneracy constraint boils down to $f_i$'s being different from the identity. 

Firstly, let us show that $\sigma$ is injective on objects; that is, all $x_i$ are all different. Suppose $x_i = x_j$ for $i < j$. If $j=i+1$, then $f_{i+1} : x_i \to x_i$ must be the identity because of acyclicity. This contradicts non-degeneracy of $\sigma$, so we must have $j> i+1$. Let us denote by $F:= f_j \circ \ldots \circ f_{i+2} \in \Hom(x_{i+1}, x_j) $. Since $f_{i+1} \in \Hom(x_i, x_{i+1})$, the two composites 
$$F f_{i+1} \in \Hom(x_i, x_j), \ \ \ f_{i+1} F \in \Hom( x_j, x_i)$$
 must be the identity. This implies $F, f_{i+1}$ are inverses, and again by acyclicity this means $f_{i+1} $ is the identity, a contradiction. This shows that the $x_i$'s are all different.

Finally, let us show this implies $\hat{\sigma} : \Delta^n \to \Nerve(C)$ is degreewise injective. Consider two $k$-dimensional faces (possibly degenerate) $\varphi_1, \varphi_2 : [k] \to [n]$ such that $\sigma \varphi_1 = \sigma \varphi_2$. Since $\sigma$ is injective on objects, we get that $\varphi_1(i)= \varphi_2(i)$ for all $i=0, \ldots, k$. But this means $\varphi_1= \varphi_2$, since maps out of the poset $[k]$ are determined by the image of the objects.
\end{proof}
Recall that a simplex $x \in \sse{X}_n$ is said to be non-degenerate if it is not in the image of any degeneracy. A degenerate map $\sigma \in \Delta$ is the composition of degeneracy maps. Let us recall the following property about degenerate simplices:
\begin{lemma} \label{factorization} For $\sse{X}$ a non-singular simplicial set, every simplex $x \in \sse{X}_n$ can be expressed in one and only one way as
$$ x = \sigma_xr(x) \ ,$$
where $r(x)$ is non-degenerate and $\sigma_x$ is a degenerate map.
\end{lemma}
\begin{proof} Proposition 4.8 in \cite{friedman} proves that there exists a unique non degenerate $z$ such that $x= \sigma z$, for some degenerate $\sigma$. In order to prove unicity of $\sigma$, suppose $\sigma z = \sigma' z$. Passing to the corresponding cell maps we have $\hat{z} \circ \sigma^{op} = \hat{z} (\sigma')^{op}$. Since $z$ is non-degenerate, $\hat{z}$ is degreewise injective, thus $\sigma^{op} = (\sigma')^{op}$, which is equivalent to $\sigma = \sigma'$.
\end{proof}
The crucial non-degeneracy property that non-singular sSets have is the following. Let $\sse{X}^{nd}$ be the degreewise subset of $\sse{X}$ made of non-degenerate simplices. Then
\begin{lemma} If $\sse{X}$ is non-singular, $\sse{X}^{nd}$ is a semisimplicial set. \label{semisimplicial}
\end{lemma}
\begin{proof} We have to show that if $x \in \sse{X}_n$ is non degenerate, then $y=\partial_i x$ is non degenerate, for any given $i$ in $\{0, \ldots, n\}$. The map $\hat{y}: \Delta^{n-1} \to \sse{X} $ is obtained as the composition
$$ \Delta^{n-1} \stackrel{d_i}{\rightarrow} \Delta^n \xrightarrow{\hat{x} } \sse{X}  \ .$$
Since both maps are degreewise injective, the map $\hat{y}_i$ is degreewise injective. If by contradiction there exists $z$ such that $\sigma_jz=y$, then $\hat{y}$ would also be the commposition of
$$ \Delta^{n-1} \xrightarrow{s_j} \Delta^{n-2} \xrightarrow{\hat{z}} \sse{X} \ ,$$
which is not injective in degree $n-2$, a contradiction.
\end{proof}
Semisimplicial sets have an analogous definition of twisted geometric realization.
\begin{definition} Given a semisimplicial set $\sse{X}$ and a twist $\omega: \elm{X}^{op} \to \Ttop$, its geometric realization twisted by $\omega$ is defined as
$$ | \sse{X} |_{\omega} := \colim_{x \in \sse{X} } \omega(x) \ .$$
\end{definition}
Notice, however, that since semisimplicial sets do not have degeneracies, contravariance conditions reduce to just $\partial$-contravariance. Let us state the main theorem of this section:
\begin{theorem} \label{non-deg-real} Let $\sse{X}$ be a non singular simplicial set. Then
$$ |\sse{X}|_{\omega} \simeq | \sse{X}^{nd} |_{\omega^{nd} }   \ ,$$
where $\omega^{nd} := \omega | \sse{X}^{nd}$.
\end{theorem}
Before giving the proof, let us state the crucial lemma:
\begin{lemma} \label{final} Let $\sse{X}$ be a non singular simplicial set and let $i : \elm{X}^{nd} \to \elm{X} $ the inclusion of non-degenerate part. Then for all $x \in \sse{X}$ the comma category $\elm{X}^{nd}_{/x}$ has a final object.
\end{lemma}
\begin{proof} Recall the factorization of Lemma \ref{factorization}
$$ (**) \ \ \ x = \sigma_x r(x)  \ ,$$
where $r(x)$ is non-degenerate and $\sigma_x$ is a composition of degeneracies. This is an object $r(x) \to x$ in $\elm{X}^{nd}_{/x}$ which we claim to be final. Suppose $x = f y$ where $y$ is non-degenerate.  By the normal form of morhpisms in the simplex category \cite{vallette} $f$ can be uniquely written as $\sigma_1 \partial_1$ where $\sigma_1$ is a composition of degeneracies and $\partial_1$ is a composition of faces. This means 
$$ x= \sigma_1 (\partial_1 y)  \ .$$
By Lemma \ref{semisimplicial}, $\partial_1y$ is still non-degenerate, and by uniqueness of factorization $(**) $ we must have 
$$ \sigma_1 = \sigma_x, \ \ \ \partial_1 y = r(x)  \ .$$
The second equation gives a map $y \to r(x) $, and the first equation ensures that the following diagram commutes:
\[\begin{tikzcd}
	{y} & x \\
	{r(x)} \ .
	\arrow["{f}", from=1-1, to=1-2]
	\arrow["{\sigma_x}"', from=2-1, to=1-2]
	\arrow["\partial_1"', from=1-1, to=2-1]
\end{tikzcd}\]
This gives a map from $y \to x$ to $r(x) \to x$ in $\elm{X}^{nd}_{/x}$. Let us show this is unique. If $g: y \to r(x)$ is a map in $\elm{X}^{nd}_{/x}$, we must have $g(y) = r(x) $ and $g$ composition of faces. Since we also have $\partial_1(y) = r(x)$ from the previous paragraph, we conclude that $g(y) = \partial_1(y)$. The diagram of associated maps reads:
\[\begin{tikzcd}
	{\Delta^n} & {\Delta^{n+k}} & {\sse{X}} \ .
	\arrow["{\hat{\partial}_1}", shift left=1, from=1-1, to=1-2]
	\arrow["{\hat{y}}", from=1-2, to=1-3]
	\arrow["{\hat{g}}"', shift right=1, from=1-1, to=1-2]
\end{tikzcd}\]
Since $\hat{y}$ is injective degreewise, the post-composition with $\hat{y}$ is injective; since both of the compositions equal $\hat{r(x)}$, we must have $\hat{\partial}_1 =\hat{g}$, that is $g=\partial_1$.
\end{proof}
We are ready to prove the main theorem. Recall the definition of final functor from \cite{Riehl}, 8.3.
\begin{proof} We claim that the functor $(\elm{X}^{nd})^{op} \to \elm{X}^{op}$ is final. Indeed, for all $x \in \elm{X}^{op}$ the comma category 
$$(\elm{X}^{nd})^{op}_{x/} \simeq (\elm{X}^{nd}_{/x} )^{op}$$
 has an initial object by Lemma \ref{final}, thus it is connected and non-empty. It follows that 
$$ |\sse{X}|_{\omega} \simeq \colim_{x \in \elm{X}^{op} } \omega(x) \simeq \colim_{ x \in (\elm{X}^{nd})^{op} } \omega(i(x)) = \colim_{ x \in (\elm{X}^{nd})^{op} } \omega^{nd}(x) \simeq | \sse{X}^{nd} |_{\omega^{nd}}   \ .$$
\end{proof}
We give another version of Theorem \ref{non-deg-real}, which is a sort of "internalization" of the non-degenerate realization.
\begin{theorem} \label{internalized-ndr}
Let $\sse{X}$ be a non-singular simplicial set. Then
$$  \sse{X} \simeq \int^{n \in \Delta_s } \Delta^n \cdot \sse{X}^{nd}_n $$ 
in the category of simplicial sets. 
\end{theorem}
\begin{proof} Firstly, let us note that the coend can be written as
$$  \sse{X} \simeq \colim_{ x \in \elm{X}^{nd} } \Delta^{n(x)}  \ ,$$
where $n: \elm{X}^{nd} \to \Delta_s$ is the projection from the category of elements to the base category.  The density theorem for simplicial sets tells us that  (\cite{Riehl}, 7.2.8) 
$$  \sse{X} \simeq \colim_{ x \in \elm{X} } \Delta^{n(x)}  \ .$$
Using Lemma \ref{final}, we know that we can restrict the diagram to the non-degenerate part and obtain the same result, obtaining the thesis.
\end{proof}
As a final lemma, this weakened realization allows for a simpler definition of morphisms, analogous to the contravariance conditions in the previous subsection.

\begin{lemma}\textbf{(Contravariance conditions for non-degenerate realizations.)} Let $\sse{X}$ be a non singular simplicial set, $\omega$ a twist on it and $Z$ a topological space. Then a map $g: |\sse{X}|_{F} \to Z$ amounts to maps $g_z : \omega(z) \to Z$ for all $z \in \sse{X}^{nd}$ that respects the $\partial$-contravariance:  for all $x \in \sse{X}^{nd}_{n+1}, y \in \sse{X}^{nd}_n$ such that $\partial_i x = y$
\[\begin{tikzcd}
	{\omega(y)} & Z \\
	{\omega(x)} 
	\arrow["{g_y}", from=1-1, to=1-2]
	\arrow["{g_x}"', from=2-1, to=1-2]
	\arrow["{\omega(\partial^{op}_i)}"', from=1-1, to=2-1]
\end{tikzcd} \]
\end{lemma}

\begin{lemma}\label{simplicial-non-deg} \textbf{(Simplicial morphisms between non-degenerate realizations.)} Let $\sse{X}, \sse{Y}$ be simplicial sets - with $\sse{X}$ non singular -and $\omega,\theta$ twists on $\sse{X}, \sse{Y}$ respectively. A map from $|\sse{X}|_{\omega}$ to $|\sse{Y}|_{\theta}$ is determined uniquely by
\begin{itemize}
\item A simplicial map $\varphi: \sse{X} \to \sse{Y}$;
\item For all $z \in \sse{X}^{nd}$, a map of topological spaces $f_z: \omega(z) \to \theta(\varphi(z))$.
\end{itemize}
That respect $\partial$-contravariance: for all $x \in \sse{X}^{nd}_{n+1}, y \in \sse{X}^{nd}_n$ such that $\partial_i x = y$

\[\begin{tikzcd}
	\omega(y) & \theta(\varphi(y))\\
	\omega(x) &   \theta(\varphi(x)) 
	\arrow["{f_y}", from=1-1, to=1-2]
	\arrow["{f_x}"', from=2-1, to=2-2]
	\arrow["{\omega(\partial^{op}_i)}"', from=1-1, to=2-1]
	\arrow["{\theta(\partial^{op}_i)}"', from=1-2, to=2-2]
\end{tikzcd}\]
commutes.
\end{lemma}
The proof are completely analogous to the general case. 
\begin{remark}
Note that in Lemma \ref{simplicial-non-deg} $\sse{Y}$ is not required to be non-singular, nor the simplicial map $\sse{X} \to \sse{Y}$ to preserve non-degenerate simplices. The non-degenerate realization is used only on $\sse{X}$.
\end{remark}

\subsection{Positively Weighted Trees}

In this section, we define the space $\WT_m(n)$ as a generalization of $\BZ_m(n)$, together with a map to the space of configurations that allows for arbitrary distances between consecutive points. This is the starting point of a semicosimplicial zig-zag between $\BZ_m$ and $\Kons_m$. However, it will not be the intermediate object we are looking for, since the semicosimplicial action does not generalize to $\WT_m(n)$. In order to have a bridge between $\BZ_m$ and $\Kons_m$, we will have to wait until the construction of the weighted \textit{hairy} trees space, where we allow points to be "infinitesimally close" (or infinitely far).

Since the technicalities we encounter in $\WT_m$ are a simplified version of the ones we will encounter in $\WHT_m$, we chose to refer the reader to later proofs for technical lemmas. This will make the first presentation of the "weighting" idea more accessible.

\subsubsection{Construction by twisted geometric realization}
The basic idea is to use the (non-degenerate) twisted realization, to "enrich" the basic cells with weights on branches. Such cells are encoded in a functor $\check{\Omega} : \elt(\Nerve(\FNP_m(n))^{nd})^{op} \to \Ttop$, that is:
\begin{itemize}
    \item A topological space $\check{\Omega}(\bt \Gamma)$ for all chains $\bt \Gamma = \Gamma_0 < \ldots < \Gamma_d$ of Fox-Neuwrith trees;
    \item Maps 
$$\check{\Omega}(\partial_i): \check{\Omega}(\Gamma_0 < \ldots \hat{\Gamma}_i < \ldots \Gamma_d) \to \check{\Omega}(\Gamma_0 < \ldots < \Gamma_d)  \ ,$$
that satisfies the cosimplicial relations.
\end{itemize}
We start by describing the building blocks:
\begin{definition} Given $\Gamma_{\bullet} \in \Nerve(\FNP_m(n))^{nd}_d$, set 
$$ \check{\Omega}(\Gamma_{\bullet}) := \frac{ |\Delta^d| \times [(\mathbb{R}_{>0})^{n-1}]^{d+1} }{ \sim} \ ,$$
where $(\bt \lambda, \bts \omega) \sim (\bt \lambda, \bts \theta)$ if $\vect{\omega}^k = \vect{\theta}^k$ for all $k$ such that $\lambda_k \neq 0$. \newline 
Furthermore, for any increasing map $D : [r] \to [d]$, set
$$ \check{\Omega}(\Gamma_{\bullet}, D ) = \{ (\bt \lambda, \bts \omega) \in \check{\Omega}(\bt \Gamma): \ \ \lambda_k \neq 0 \ \ \Leftrightarrow \ \ k \in \im D\} \ .$$
\end{definition}

Just as for Fox-Neuwirth Trees, we can think about elements of $\check{\Omega}(\bt \Gamma)$ in a pictorial way. The first component $\bt \lambda \in |\Delta^d|$ provides the coefficients of a convex combination. The second component provides a vector of weights $\vect{\omega}^k \in \mathbb{R}_{> 0}^{n-1}$ for each tree $\Gamma_k$. If we put weights on forks, we can depict an element of $\check{\Omega}(\bt \Gamma)$ as:
\begin{figure}[H]
	\centering
	\includegraphics[scale=0.18]{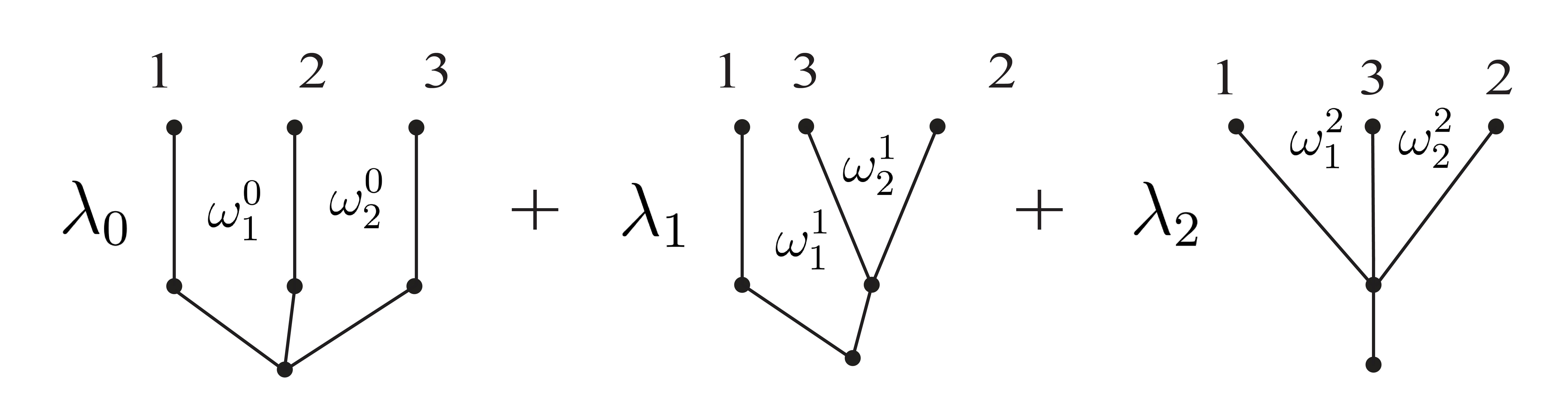}
	\caption{A convex combination of weighted trees}
	\label{weighted-combination}
\end{figure}
The equivalence relation, from this point of view, is extremely natural: if a tree has coefficient zero, the weights we put on it do not make any difference. 

\begin{remark} \label{polytope-exploration-wt} The geometry of cells $\check{\Omega}(\Gamma_{\bullet})$, being a blow up of the simplex (see Lemmas \ref{projection} and \ref{BZ-he-2}), deserves to be explored. 
\begin{comment}For the sake of geometrical intuition, let us refer to the \textit{associated polytope} $\mathcal{P}(C)$ of a cell $C$ that can be naturally embedded in the Euclidean space. In other words, we will identify $[0, \infty]$ with the closed interval $[0,1]$, and then take the closure 
\end{comment}
When $\Gamma$ is a single tree in $\FNP_m(n)$, we have that
$$ \check{\Omega}(\Gamma) = (\mathbb{R}_{>0})^{n-1} \simeq (0,1)^{n-1}\ , $$
which is, up to closure in the Euclidean space, a $(n-1)$-dimensional cube. We denote the latter by $\Cube_{n-1}$. When an entire chain of trees $\bt \Gamma$ is involved, we can describe the geometry of $\check{\Omega}( \bt \Gamma)$ in terms of the \textit{free join of polytopes} \cite{polytopes}. Let us explain the construction in a few words. Let $\bar{P}, \bar{Q}$ be emebddings of two polytopes $P,Q$, respectively, into skew affine planes of dimension $d$ and $e$ inside $\mathbb{R}^{d+e+1}$. Their free join $P \vee Q$ is defined as the convex hull of $\bar{P} \cup \bar{Q}$. This is just a polytopic realization of the abstract join of spaces. Since our definition essentially sets
$$ \check{\Omega}(\Gamma_0 < \ldots < \Gamma_d) = \check{\Omega}(\Gamma_0) * \ldots * \check{\Omega}(\Gamma_d) \ , $$
the associated polytope is $\Cube_{n-1}^{\vee(d+1)}$. This observation in the simple case $n=2, d=1$ is illustrated in figure \ref{polytope-1}. The projection to the simplex, from this perspective, amounts to collapsing a each cube to a point. Indeed, we have:
$$ \Cube_{n-1}^{\vee(d+1)} \to (\Delta^0)^{\vee(d+1)} \simeq \Delta^d \ . $$
\end{remark}

\begin{remark} Since we will not use such "associated polytopes" if not for geometrical intuition, we will not delve into the subtle details behind the sentence "up to a closure in the Euclidean space". In what follows, we will sometimes sloppily denote by $\mathcal{P}(C)$ the polytope associated to a cell $C$. The previous remark, in this notation, would be summed up as 
$$\mathcal{P}(\check{\Omega}(\bt \Gamma)) \simeq \Cube_{n-1}^{\vee(d+1)} \ ,$$
for $\bt \Gamma \in \Nerve( \FNP_m(n))_{d}^{nd}$. 
\end{remark}

\begin{figure}
\centering
\includegraphics[width=9cm]{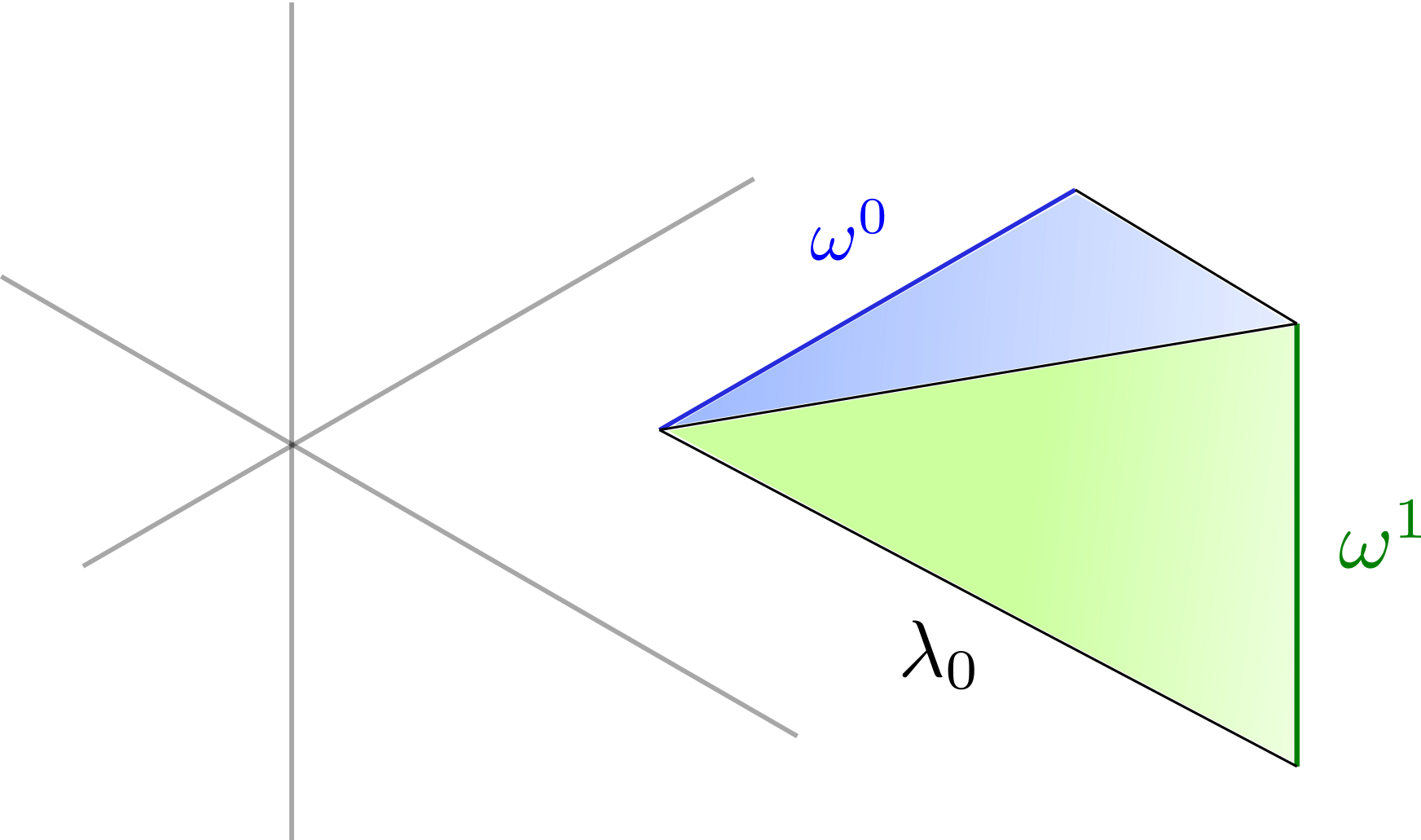}
\caption{The geometry behind $\check{\Omega}(1|2 <12)$, a cell of $\WT_2(2)$}
\label{polytope-1}
\end{figure}

Let us now define the $\check{\Omega}(\partial_i)$ maps:

\begin{definition} \label{twist-2} For $\Gamma_{\bullet} \in \Nerve(\FNP_m(n))_d^{nd} $ and $i \in [d+1]$, define $\check{\Omega}(\partial_i) : \check{\Omega}( \partial_i \Gamma_{\bullet}) \to \check{\Omega}( \bt \Gamma) $ as $\check{\Omega}(\partial_i)(\bt \lambda, \bts \omega) = (\partial_i \bt \lambda, \partial_i \bts \omega)$, where
$$ \partial_i(\bt \lambda)_k = \left\{ \begin{array}{ll}
       0,  & \textrm{if }  k=i\\
      	\lambda_{\sigma_i(k)}, & \textrm{ otherwise} 
    \end{array} \right.  $$
$$ \partial_i(\bts \omega )^k = \left\{ \begin{array}{ll}
       \underline{u},  & \textrm{if }  k=i\\
      	\vect{\omega}^{\sigma_i(k)}, & \textrm{ otherwise} 
    \end{array} \right.  $$
Here $\underline{u}=(1, \ldots, 1)$.
\end{definition}
The following lemmas ensures we can use such blocks to build a space, and that the structure behaves nicely with respect to the stratification:
\begin{lemma}  \label{restriction} The spaces $\check{\Omega}(\Gamma_{\bullet}, D )$ stratify $\check{\Omega}(\bt \Gamma)$ with exit poset structure
$$\overline{\check{\Omega}(\Gamma_{\bullet}, D )} = \bigsqcup_{T:[r'] \to [r] } \check{\Omega}(\bt \Gamma, DT)  \ .$$
The maps $\Omega(\partial_i)$ are stratified and turns $\check{\Omega}$ into a functor $\check{\Omega}: \elt(\Nerve(\FNP_m(n))^{nd})^{op} \to \Ttop$.
\end{lemma}
\begin{proof} Regarding the stratification, it is easy to see that $(\bt \lambda, \bts \omega) \in \overline{\check{\Omega}(\Gamma_{\bullet}, D )}$ if and only if
$$\{k: \lambda_k \neq 0 \}  \subset \im D  \ .$$
If $(\bt \lambda, \bts \omega) \in \check{\Omega}(\bt \Gamma, D')$ for some $D'$, this implies by the very definition $\im D' \subset \im D$. Since they are both increasing maps, we conclude the existence of increasing $T$ such that $D' = DT$.

Let us pass to examine the maps $\Omega(\partial_i)$. Since we only add weights equal to $1$, they are well-defined on $\WT$ cells that only allow for positive-and-finite weights. They are easily seen to be stratified with stratification map
$$ \Omega(\partial_i) \left ( \check{\Omega}(\bt \Gamma, D) \right ) \subset \check{\Omega}( \bt \Gamma, \partial_i D)  \ .$$
Regarding the well definition with respect to the equivalence relation and the cosimpliciality, the proof is the same as in Lemma \ref{stratified}.
\end{proof}
We are ready to give the definition of the weighted trees space.
\begin{definition} Define $\WT_m(n)$ as
$$  \WT_m(n) = | \Nerve(\FNP_m(n))^{nd}|_{\check{\Omega}} \ .$$
\end{definition}
\subsubsection{Equivalence with $\BZ_m$}
We want to see that the weighted trees space is another model for the $\BZ$ space. Let us define the projection:
\begin{lemma} \label{projection} For $\bt \Gamma \in \Nerve(\FNP_m(n))^{nd}_d$, the maps
$$\check{f}_{\bt \Gamma} : \check{\Omega}(\bt \Gamma) \to |\Delta^d| \ ,$$
defined as $\check{f}(\bt \lambda, \bts \omega) = \bt \lambda$ are indeed well defined with respect to the equivalence relation. They respect $\partial$-contravariance conditions, thereby providing a map
$$ \check{f} : \WT_m(n) \to \BZ_m(n)  \ .$$
\end{lemma}
The proof is identical to the one given in Lemma \ref{BZ-semi}. There is also a parallel of Lemma \ref{BZ-he}:
\begin{lemma} \label{BZ-he-2} For $\bt \Gamma \in \Nerve(\FNP_m(n))_d$, define
$$\check{g}_{\bt \Gamma} :  \textrm{std}(\bt \Gamma) \to \check{\Omega}(\bt \Gamma) $$
as  $\check{g}_{\bt \Gamma}(\bt \lambda) = (\bt \lambda, \bts u)$, where $\vect{u}^k_{\alpha} = 1$ for all $k, \alpha$. Then $\check{g}_{\bt \Gamma}$ respects $\partial$-contravariance, and the assembled maps 
$$\check{g}_m(n) : \BZ_m(n) \to \WT_m(n)$$
are homotopy inverses to $\check{f}_m(n)$. 
\end{lemma}
The proof is similar to Lemma \ref{BZ-he}, but much simpler for several reasons: 
\begin{enumerate}
\item Weights are in $(0,\infty)$, thus we can use ordinary convex combinations;
\item The maps are evidently stratified, since the stratum only depends on $\bt \lambda$ which does not change during the homotopy
\item The new component $\bts u$ is independent of $\bt \Gamma$.
\end{enumerate}
\subsubsection{Associated configurations}
Now that the new model is constructed, we want to generalize the definition sketched at the end of Section \ref{configurations} from $\BZ_m$ to $\WT_m$. The first step is to associate a configuration to a single weighted tree.
\begin{definition} Given a tree $\Lambda = (\sigma, a) \in \FNP_m(n)$ and $\vect{\theta} \in (\mathbb{R}_{> 0} )^{n-1}$ the function $x : \FNP_m(n) \times (\mathbb{R}_{> 0} )^{n-1} \to \Conf_n(\mathbb{R}^m)$ is inductively defined as
$$ x(\Lambda, \vect{\theta} )_{ \sigma(1) } = 0  \ ,$$
$$ x(\Lambda, \vect{\theta} )_{ \sigma(p+1) } = x(\Lambda, \vect{\theta} )_{ \sigma(p) } + \theta_p e_{1+a_p}  \ .$$
\end{definition}
The definition on a combination of weighted trees is simply the combination of the results:
\begin{definition} \label{ass-conf} For $\bt \Gamma \in \Nerve(\FNP_m(n))^{nd}_d $ and $D : [r] \to [d] $, we want to define a map
$$ \check{\tau}^0 : \check{\Omega}(\bt \Gamma, D) \to (\mathbb{R}^m)^n  \ .$$
Given $(\bt \lambda, \bts \omega) \in \check{\Omega}(\bt \Gamma, D)$, we set
$$ \check{\tau}^0(\bt \lambda, \bt \omega) = \sum_{k=0}^r \lambda_{D(k)}	x(\Gamma_{D(k)}, \vect{\omega}^{D(k)}) \ .$$
\end{definition}
\begin{figure}[H]
	\centering
	\includegraphics[scale=0.2]{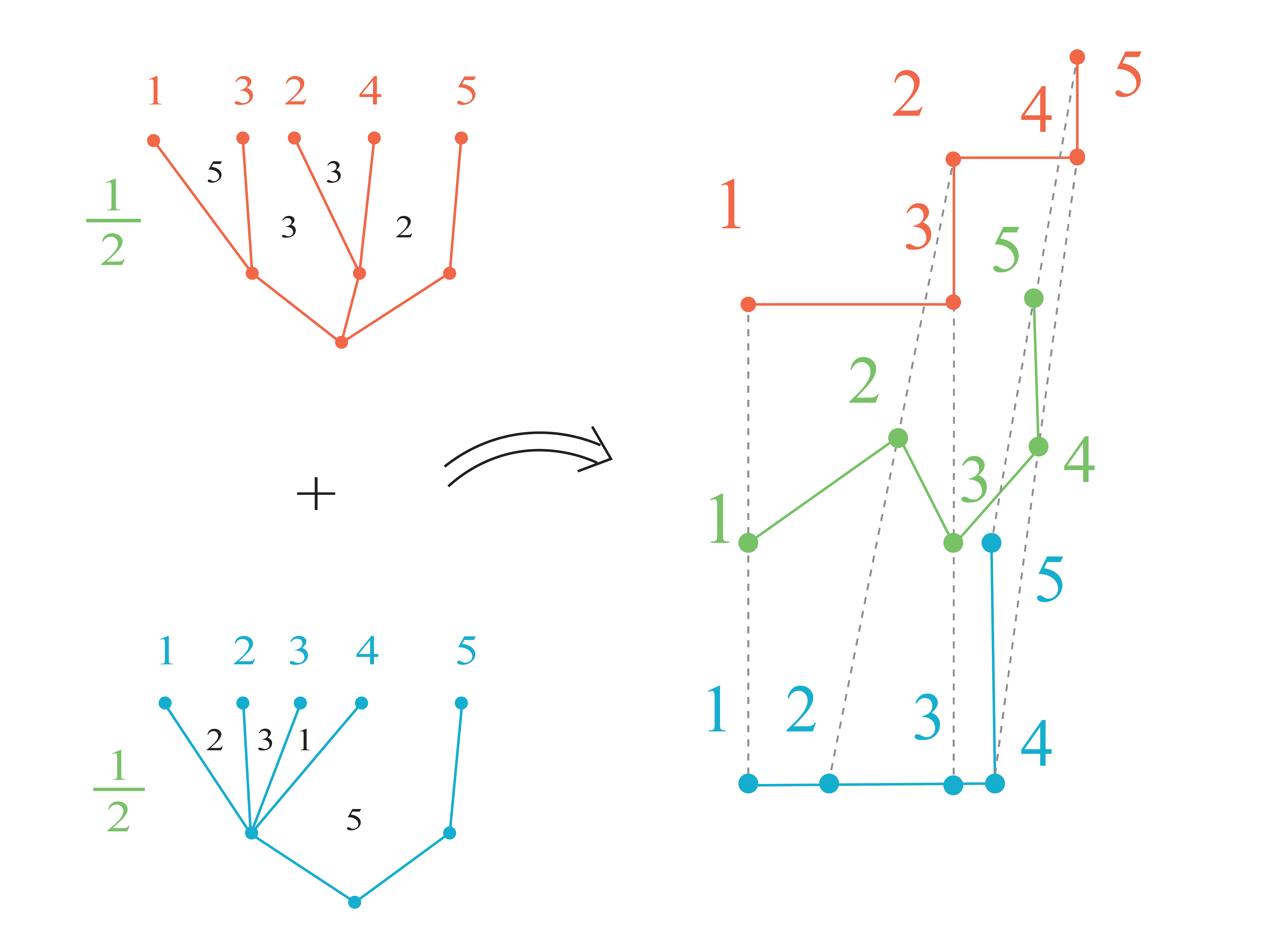}
	\caption{The configuration of points associated to a convex combination of weighted trees}
	\label{associated-configuration}
\end{figure}
It is not clear that $\check{\tau}^0$ yields a configuration since it is a combination of configurations. Before seeing the proof, note that $\check{\tau}^0$ is well-defined with respect to the equivalence relation, since we only use weights $\vect{\omega}^k_{\alpha}$ with $k \in \im D$. 

\begin{remark} From a geometric perspective, this should not be surprising. For any tree $\Gamma$ and vector of weights $\theta_{\bullet}$, the point $x(\Gamma, \theta_{\bullet})$ is constructed so that $x(\Gamma, \theta_{\bullet}) \in \Conf(\Gamma)$. If we imagine to take $x(\Gamma, \theta_{\bullet})$ as a barycenter of the cell $\Conf(\Gamma)$ in the one-point compactification $\Conf_n(\mathbb{R}^m)^+$, convex combinations of $x(\Gamma_0, \theta_{\bullet}^{(0)} ), \ldots, x(\Gamma_d, \theta_{\bullet}^{(d)})$ span the barycentric cell corresponding to the chain $\Gamma_0 < \ldots < \Gamma_d$.
\end{remark}

In order to show it is actually a configuration, we will explicitly compute the difference between two points and check it is a non-zero vector. Let us set up a notation.
\begin{definition}  Given $\Lambda \in \FNP_m(n)$ and $i \neq j$ in $ \{1, \ldots, n\}$, define 
$$ \sgn^{\Lambda}_{ij} = \left\{ \begin{array}{ll}
       +1 ,  & \textrm{if }  i <_{\Lambda} j \\
       -1 ,  & \textrm{if }  i >_{\Lambda} j 
    \end{array} \right. $$
\end{definition}
We are ready to give the formula for the difference of two points (see figure \ref{walking-man}).
\begin{lemma}[Walking-man formula] \label{difference}  For $\bt \Gamma \in \Nerve(\FNP_m(n))^{nd}_d $ and $D : [r] \to [d] $, let us consider $i \neq j $ in $ \{1, \ldots, n\}$ and $(\bt \lambda, \bts \omega) \in \check{\Omega}(\bt \Gamma, D)$. Let $\alpha_k = \min\{i,j\}, \beta_k = \max\{i,j \}$ with respect to the order of $\Gamma_{D(k)}$, and $p_k = \sigma_{D(k)}^{-1}(\alpha_k ), q_k = \sigma_{D(k)}^{-1}(\beta_k)$, where $\Gamma_{\ell} = (\sigma_{\ell}, a^{\ell})$. Then we have
$$ \check{\tau}^0(\bt \lambda, \bts \omega)_j - \check{\tau}(\bt \lambda, \bts \omega)_i = \sum_{k=0}^r \lambda_{D(k)} \sgn^{\Gamma_{D(k)}}_{ij} \sum_{h = p_k}^{q_k-1} \vect{\omega}^{D(k)}_{h} e_{1+a^{D(k)}_{h}}  \ .$$
\end{lemma}
\begin{proof} Firstly, let us reduce to the case of a single tree. Indeed, we have
\begin{align*} &
\check{\tau}^0(\bt \lambda, \bts \omega)_j - \check{\tau}(\bt \lambda, \bts \omega)_i  =  \sum_{k=0}^r \lambda_{D(k)} \left ( x(\Gamma_{D(k)}, \vect{\omega}^{D(k)})_j - x(\Gamma_{D(k)}, \vect{\omega}^{D(k)})_i \right ) \\
& \stackrel{?}{=} \sum_{k=0}^r \lambda_{D(k)} \sgn^{\Gamma_{D(k)}}_{ij} \sum_{h = p_k}^{q_k-1} \vect{\omega}^{D(k)}_{h} e_{1+a^{D(k)}_{h}}\ ,
\end{align*}
so it is enough to show, for each $k$, that
$$ x(\Gamma_{D(k)}, \vect{\omega}^{D(k)})_j - x(\Gamma_{D(k)}, \vect{\omega}^{D(k)})_i = \sgn^{\Gamma_{D(k)}}_{ij} \sum_{h = p_k}^{q_k-1} \vect{\omega}^{D(k)}_{h} e_{1+a^{D(k)}_{h}}  \ .$$
We can also get rid of the $\sgn$ part by describing what happens to the left-hand side and to the right-hand side when we swap $i,j$. The LHS is manifestly antisymmetric. The summation is symmetric, since $p_k = \sigma_k^{-1} \min_{\Gamma_k}\{i,j\}$, $q_k = \sigma_k^{-1} \max_{\Gamma_k} \{i,j\}$. The tensor $\sgn_{ij}^{\Gamma_{D(k)}}$ is antysimmetric. We conclude that it is enough to show that
$$ x(\Gamma_{D(k)}, \vect{\omega}^{D(k)})_j - x(\Gamma_{D(k)}, \vect{\omega}^{D(k)})_i = \sum_{h = p_k}^{q_k-1} \vect{\omega}^{D(k)}_{h} e_{1+a^{D(k)}_{h}} $$
for $ i <_{\Lambda_k} j$. We show this by induction on $q_k-p_k$. If it is one, then $\sigma_k(p_k) = i$ and $\sigma_k(p_k+1) = j$, so that by the very definition
$$ x(\Gamma_{D(k)}, \vect{\omega}^{D(k)})_{ \sigma_k(p_k+1)} = x(\Gamma_{D(k)}, \vect{\omega}^{D(k)})_{\sigma_k(p_k)} + \omega_{p_k}^{D(k)}e_{1+a_{p_k}^{D(k)}}  \ .$$
Regarding the inductive step, consider $j'= \sigma_k(q_k-1)$. Note that since $q_{k}-1 > p_k$ we have that $j' >_{\Lambda_k} i$. We can use the both for $j,j'$ and for $j',i$, respectively descending from the base case and from the inductive hypothesis. Putting all together we get
\begin{align*}
&x(\Gamma_{D(k)}, \vect{\omega}^{D(k)})_j -x(\Gamma_{D(k)}, \vect{\omega}^{D(k)})_i =  \\
= &\left (x(\Gamma_{D(k)}, \vect{\omega}^{D(k)})_j - x(\Gamma_{D(k)}, \vect{\omega}^{D(k)})_{j'} \right )+
\left ( x(\Gamma_{D(k)}, \vect{\omega}^{D(k)})_{j'} - x(\Gamma_{D(k)}, \vect{\omega}^{D(k)})_i \right ) = \\
 = &\vect{\omega}_{q_k-1}^{D(k)}e_{1+a_{q_k-1}^{D(k)}} + \sum_{h=p_k}^{q_k-2} \vect{\omega}_{h}^{D(k)} e_{1+a_h^{D(k)} } = \\
 = & \sum_{h=p_k}^{q_k-1} \vect{\omega}_{h}^{D(k)} e_{1+a_h^{D(k)} }\ .
\end{align*}
\end{proof}

\begin{figure}[H]
	\centering
	\includegraphics[scale=0.15]{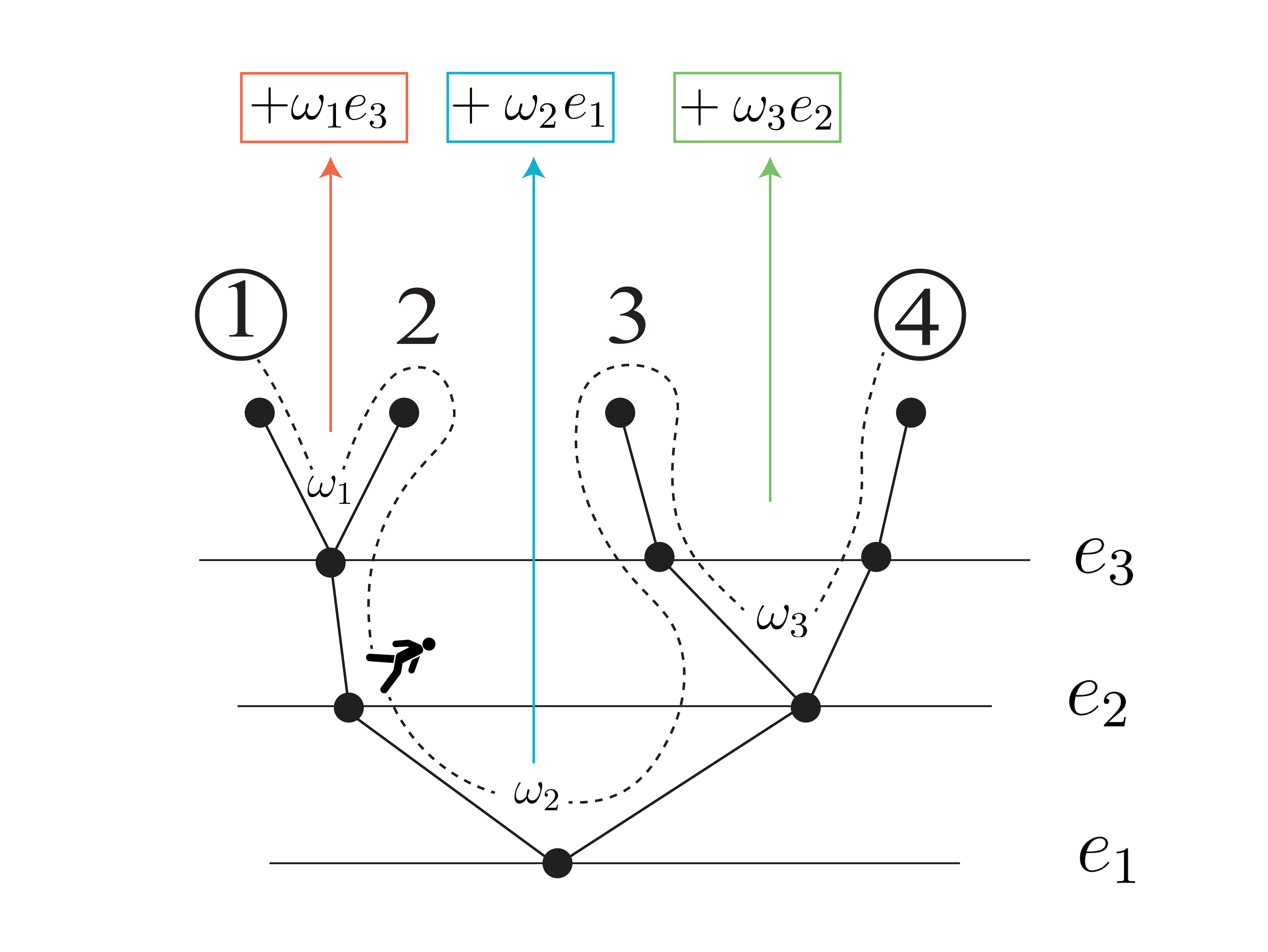}
	\caption{A tiny man walking from $1$ to $4$ and collecting the contributes on the forks}
	\label{walking-man}
\end{figure}

We have all the ingredients to promote $\check{\tau}^0$ to a map from $\check{\Omega}(\bt \Gamma)$ to the configuration space.
\begin{lemma} \label{stratum-tau} For all $\bt \Gamma \in \Nerve(\FNP_m(n))_d$ and $D: [r] \to [d]$, the map 
$$\check{\tau}^0 : \check{\Omega}(\bt \Gamma, D) \to (\mathbb{R}^m)^n$$
corestricts to $\Conf_n(\mathbb{R}^m)$. It is also compatible with the stratification, yielding a map
$$\check{\tau}^1 : \check{\Omega}(\bt \Gamma) \to \Conf_n(\mathbb{R}^m)  \ .$$
\end{lemma}
\begin{proof} In order to see that the result of $\check{\tau}^0$ is a configuration, we will show  that the difference of any two points has a non-zero component. Specifically, for any $i,j$, suppose that $i <_r j$ in $\Gamma_{D(0)}$. Also, take $q$ the maximum index such that $i<_r j$ in all $\Gamma_{D(0)}, \ldots, \Gamma_{D(q)}$. Recall that in $\Gamma_{D(q+1)}, \ldots, \Gamma_{D(r)}$ we will have $i<_s j$ or $j<_s i$ with $s > r$. Consider $(\bt \lambda, \bts \omega) \in \check{\Omega}(\bt \Gamma, D)$. Using the formula in Lemma \ref{difference} with the usual $p_k, q_k$ notation, we have
\begin{align*}
& \check{\tau}^0(\bt \lambda, \bts \omega)_j - \check{\tau}^0(\bt \lambda, \bts \omega)_i   = \sum_{r=0}^k \lambda_{D(k)} \sgn_{ij}^{\Lambda_{D(k)} } \sum_{h=p_k}^{q_k-1} \vect{\omega}^{D(k)}_h e_{1+a^{D(k)}_h} \\
= & \sum_{r=0}^q \lambda_{D(k)}\cdot (+1) \sum_{h=p_k}^{q_k-1} \vect{\omega}^{D(k)}_h e_{1+r} + \sum_{r=q+1}^k \lambda_{D(k)} \sgn_{ij}^{\Lambda_{D(k)} } \sum_{h=p_k}^{q_k-1} \vect{\omega}^{D(k)}_h e_{1+a^{D(k)}_h}\ .
\end{align*}
Note that $\lambda_p$ and $\omega^{p}_{\alpha}$ are positive for all $p, \alpha$, thus the first summation is a positive multiple of $e_{1+r}$. The other summation instead, because of the above observation, is in $\textrm{Span} \{e_{2+r}, \ldots, e_m \}$. We can conclude that 
$$ \langle \check{\tau}^0(\bt \lambda, \bts \omega)_j - \check{\tau}^0(\bt \lambda, \bts \omega)_i , e_{1+r} \rangle > 0 \ ,$$
which in particular implies that $\check{\tau}^0(\bt \lambda, \bts \omega)$ is a configuration. 

Let us examine the compatibility with the stratification. We have that
$$ \overline{\check{\Omega}(\bt \Gamma, D)}  = \bigsqcup_{T: [r'] \to [r] } \check{\Omega}(\bt \Gamma, DT)  \ .$$
Consider a point $(\bt \lambda, \bts \omega) \in \check{\Omega}(\bt \Gamma, DT)$ for some $T: [s] \to [r]$, and take $(\bt \lambda^{(n)}, (\bts \omega)^{(n)})$ converging to $(\bt \lambda, \bts \omega)$. We have to show that
$$ \lim_{n \to \infty} \tau^0(\bt \lambda^{(n)}, (\bts \omega)^{(n)}) = \tau^0(\bt \lambda, \bts \omega)  \ .$$
Note that 
\begin{align*}
 \lim_{n \to \infty} \check{\tau}^0(\bt \lambda^{(n)}, (\bts \omega)^{(n)} ) & =  \lim_{n \to \infty} \sum_{k=0}^r \lambda^{(n)}_{D(k)}	x(\Gamma_{D(k)}, (\vect{\omega}^{D(k)})^{(n)}) = \\
 & =  \lim_{n \to \infty}  \sum_{h \in \im D} \lambda^{(n)}_h x(\Gamma_h, (\vect{\omega}^h)^{(n)}) \\
 & = \sum_{h \in \im D} \lambda_h x(\Gamma_h, \vect{\omega}^h) \ .
\end{align*}
We now use the fact that $\lambda_h = 0$ for all $h \not \in \im DT$, giving
\begin{align*}
& = \sum_{h \in \im DT} \lambda_h x(\Gamma_h, \vect{\omega}^h) \\
& = \sum_{k=0}^s \lambda_{DT(k)} x(\Gamma_{DT(k)}, \vect{\omega}^{DT(k)} ) \\
& = \check{\tau}^0(\bt \lambda, \bts \omega) \ ,
\end{align*}
which is the thesis.
\end{proof}

\begin{remark} From a combinatorial perspective, the proof of $\check{\tau}^0$ being a configuration can be summarised in this way. Consider a chain of trees $\Gamma_0 < \ldots < \Gamma_d$ and a point $(\lambda_{\bullet}, \bts \omega)$. Suppose that without loss of generality $i <_p j$ in $\Gamma_0$ and that $\lambda_0 \neq 0$. Then the difference between the $i$-th and $j$-th point receives a positive contribution along direction $e_{1+p}$ from $x(\Gamma_0, \vect{\omega}^0)$, since the fork between $i$ and $j$ reaches depth $p$ at some point. At this point, all the bigger trees in the chain falls in two categories (see Section \ref{configurations}, explaining the combinatorial reformulation of the poset structure on Fox-Neuwirth trees):
\begin{itemize}
\item Whether $i<_p j$, so that an analogous contribution occurs;
\item Whether $i<_q j$ or $j <_q i$ with $q  > p$, contributing to a different direction.
\end{itemize}
All in all, the difference between the two points will be positive in the direction $e_{1+p}$.

\end{remark}
Finally, let us see that we can put this together into a globally defined map.
\begin{lemma} For $\bt \Gamma \in \Nerve(\FNP_m(n))^{nd}_d $, the map $\check{\tau}^1: \check{\Omega}(\Gamma_{\bullet}) \to \Conf_n(\mathbb{R}^m)$ satisfy the $\partial$-contravariance condition, so it assembles into a map
$$ \check{\tau} : \WT_m(n) \to \Conf_n(\mathbb{R}^m)  \ .$$
\end{lemma}
\begin{proof} The diagram for the $\partial$-contravariance reads
\[\begin{tikzcd}
	{\check{\Omega}(\partial_i \Gamma_{\bullet})} & {\textrm{Conf}_n(\mathbb{R}^m)} \\
	{\check{\Omega}( \Gamma_{\bullet})}
	\arrow[from=1-1, to=1-2]
	\arrow[from=1-1, to=2-1]
	\arrow[from=2-1, to=1-2]
\end{tikzcd}\]

Let us restrict to the stratum $\check{\Omega}(\partial_i \bt \Gamma, D) \subset \check{\Omega}(\partial_i \bt \Gamma)$ for some $D: [r] \to [d-1]$. Because of Lemma \ref{restriction}, we know that the left map lands in $\check{\Omega}(\bt \Gamma, \partial_i D)$. Thus we reduce to the commutativity of
\[\begin{tikzcd}
	{\check{\Omega}(\partial_i\Gamma_{\bullet}, D)} & {\textrm{Conf}_n(\mathbb{R}^m)} \\
	{\check{\Omega}( \Gamma_{\bullet}, \partial_i D )}
	\arrow[from=1-1, to=1-2]
	\arrow[from=1-1, to=2-1]
	\arrow[from=2-1, to=1-2]
\end{tikzcd}\]
Consider $(\bt \lambda, \bts \omega) \in \check{\Omega}(\partial_i \bt \Gamma, D)$. The upper arrow is
$$ (\bt \lambda, \bts \omega) \mapsto \sum_{k=0}^r \lambda_{D(k)} x( \partial_i \Gamma_{D(k)}, \vect{\omega}^{D(k)} ) = \sum_{k=0}^r \lambda_{D(k)} x(  \Gamma_{d_i D(k)}, \vect{\omega}^{D(k)} )  \ ,$$
while the composition of the other two arrows give
$$ (\bt \lambda, \bts \omega) \mapsto (\partial_i \bt \lambda, \partial_i \bts \omega) \mapsto \sum_{k=0}^r (\partial_i \lambda)_{d_i D(k)} x(\Gamma_{d_i D(k)}, (\partial_i \vect{\omega})^{d_i D(k)} ) =  \sum_{k=0}^r \lambda_{D(k)} x(\Gamma_{d_i D(k)}, \vect{\omega}^{D(k)} ) \ ,$$
which coincides with the previous sum. We used that $(\partial_i \lambda)_{d_i s} = \lambda_{s_i d_i s} = \lambda_s$ and similarly for $\bts \omega$. 
\end{proof}

\subsection{Weighted Hairy Trees} \label{wht-section}
We now construct a variant of the $\WT_m$ space, in which we allow zero or $\infty$ weights on "hairs": a group of consecutive labels with the highest possible height (e.g. $4 <_2 5 <_2 6$ in $\FNP_3$). This will serve as a bridge between $\BZ$ and $\Kons$. It is crafted in a four-step process:
\begin{enumerate}
\item Build $\tilde{\Omega}(\Lambda_{\bullet}, \phi, D)$ where $\bt \Lambda \in \Nerve(\FNP_m(n))^{nd}_r, \phi: [n] \to [\ell], D: [r] \to [d]$ in $\Delta_s$;
\item Put them together as strata of a space $\tilde{\Omega}(\Gamma_{\bullet}) \subset |\Delta^d | \times \left ( \mathbb{R}^{ \ell -1} \right )^{d+1}$ for $\bt \Gamma \in \Nerve(\FNP_{m}(\ell))^{nd}_{d}$, then quotient out by some equivalence relation $\mathcal{R}$ and obtain cells $\Omega(\bt \Gamma)$;
\item Turn $\Omega$ into a functor $\Nerve(\FNP_m(n))^{nd} \to \Ttop$, and use cells $\Omega(\Gamma_{\bullet})$ to build the twisted non-degenerate realization $\WHT_m(n) = | \Nerve(\FNP_m(n))^{nd} |_{\Omega}$;
\item Construct the semicosimplicial maps $d_j: \WHT_m(n) \to \WHT_m(n+1)$ as simplicial morphisms and verify the semicosimplicial identities.
\end{enumerate} 
A more intuitive explanation will be given in Remark \ref{wht-intuition}, but it is important to keep in mind the general roadmap from a formal point of view. A picture of $\WHT_2(2)$ can be found in \ref{wt-sd}.
\subsubsection{Construction of defining cells}
Since the construction involves convex combinations, let us give the following
\begin{definition} Given $D: [r] \to [d]$, let us define $ \mathring{D}|\Delta^d| \subset | \Delta^d | $ as the open face associated to $D$:
$$\mathring{D}|\Delta^d| := \{ (\lambda_0, \ldots, \lambda_d) \in |\Delta^d |: i \not \in \im(D) \Rightarrow \lambda_i = 0, i \in \im(D) \Rightarrow \lambda_i > 0 \} \ .$$
\end{definition}
We also have to distinguish a special kind of tree for the building blocks, for a reason that will be clear later:
\begin{definition} The tree $T_n = 1 <_{m-1} \ldots <_{m-1} n$ with trivial permutation and maximal depth indices is called the \textit{trivial tree} with $n$ leaves. If a chain of trees $\bt \Gamma \in \Nerve(\FNP_m(n))_d$ ends with a trivial tree, it is called trivializable.
\end{definition}
Note also that if a chain \textit{contains} a trivial tree it must be at the end since it is a maximal element of the poset. There is also a special map that will appear together with the special tree:
\begin{definition} Given $d \in \mathbb{N}$, denote by $u_d : [0] \to [d]$ the map that sends $0$ to $d$.
\end{definition}
To simplify the subsequent exposition, we also define:
\begin{definition} The triple $(\bt \Lambda, \phi, D)$ for $\Lambda_{\bullet} \in \Nerve(\FNP_m(n))^{nd}_r, \phi: [n] \to [\ell], D: [r] \to [d]$ is called \textit{trivial} if $r=0, \bt \Lambda = T_n, D = u_d$.
\end{definition}
We are ready to define the building blocks of $\WHT$.
\begin{definition} Given a non-trivial triple $\Lambda_{\bullet} \in \Nerve(\FNP_m(n))^{nd}_r, \phi: [n] \to [\ell], D: [r] \to [d]$, define
$$ \tilde{\Omega}(\bt \Lambda, \phi, D) := \{ ( \bt \lambda, \bts{\omega} ): \bt \lambda \in \mathring{D} | \Delta^d|, \bts{\omega} \in (\Rext^{\ell-1} )^{d+1} : \forall k \in [r] \ \ \forall \phi^{\Lambda_k}(0) \le \alpha \le \phi^{\Lambda_k}(n)  $$
$$
(\vect{\omega}^{D(k)})_{\alpha}  \left\{ \begin{array}{ll}
        = 0,  & \textrm{if }  \alpha \not \in \im \phi^{\Lambda_k } \\
        \in \mathbb{R}_{> 0},  &  \textrm{if }  \phi^{\Lambda_k }(0) < \alpha < \phi^{\Lambda_k}(n) \ \ \ \textrm{ and }  \alpha \in \im \phi^{\Lambda_k } \\
        = \infty, & \textrm{if }  \alpha = \phi^{\Lambda_k }(0) \textrm{ or } \alpha = \phi^{\Lambda_k}(n) \\
    \end{array} \right. $$
$$
\} $$
For a trivial triple $(T_n, \phi, u_d)$ the weights are not constrained, so the cell does not depend on $\phi$:
$$ \tilde{\Omega}(T_n, \phi, u_d) = \tilde{\Omega}( \phi T_n, u_d ) = \tilde{\Omega}( T_{\ell}, u_d) :=  \{ ( \bt \lambda, \bts{\omega} ): \bt \lambda \in \mathring{D} | \Delta^d|, \bts{\omega} \in (\Rext^{\ell-1} )^{d+1} \}  \ .$$
\end{definition}
\begin{remark} \label{wht-intuition}
Such block represents a convex combination of weighted trees, where:
\begin{itemize}
\item $\bt \lambda$ is the tuple of coefficients for the combination;
\item $D$ specifies which $\lambda_k$ are zero and which are positive;
\item $\Lambda_{\bullet}$ specify a tree for each non-zero coefficient of the combination;
\item $\vect{\omega}^{D(k)}$ is the vector of weights associated to the tree $\Lambda_k$; the weight $(\vect{\omega}^{D(k)})_{\alpha}$ sits in the branch between the $\alpha$-th and the $(\alpha+1)$-th position. We retain redundant weight vectors $\vect{\omega}^k$ for $k \not \in \im(D)$ to provide a common ambient space for the bricks when varying $D$. However, this degree of redundancy will be factored out in the process of construction.
\item $\phi= d_{i_1} \ldots d_{i_p}$ with $i_1 \le \ldots \le i_p $ describes an element of the cosimplicial action on Fox-Neuwirth trees and it prescribes where 'hairs' are meant to be (see corollary ). Note however that the doubling process must be homogeneous in $\Lambda_0, \ldots, \Lambda_r$, since $\phi$ is the same for all the trees.
\end{itemize}
\end{remark}

\begin{example} Consider $D, \phi$ defined as 
\hspace*{-1cm} \begin{figure}[H]
	\centering
	\includegraphics[scale=0.135]{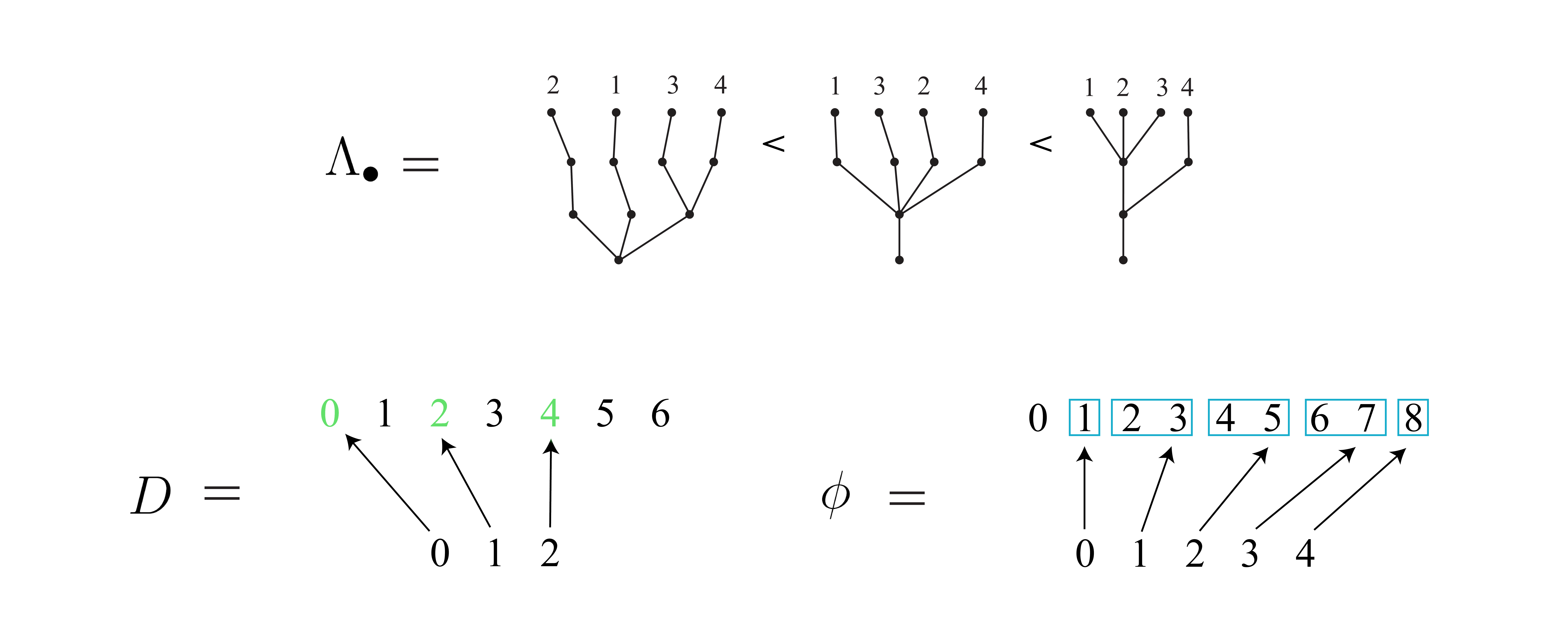}
\end{figure}
 An element $(\bt \lambda, \bt \omega) \in \check{\Omega}(\bt \Lambda, \phi, D)$ can be pictured analogously to figure \ref{weighted-combination}, but with two differences: we don't have to specify trees where coefficients of the convex combinations are zero, thus we only keep positive coefficients; weights are constrained to be $0,\infty$ on the hairs predicted by the Shape-Tree Lemma applied to $\phi$. All in all, we have 
\begin{figure}[H]
	\centering
	\includegraphics[scale=0.23]{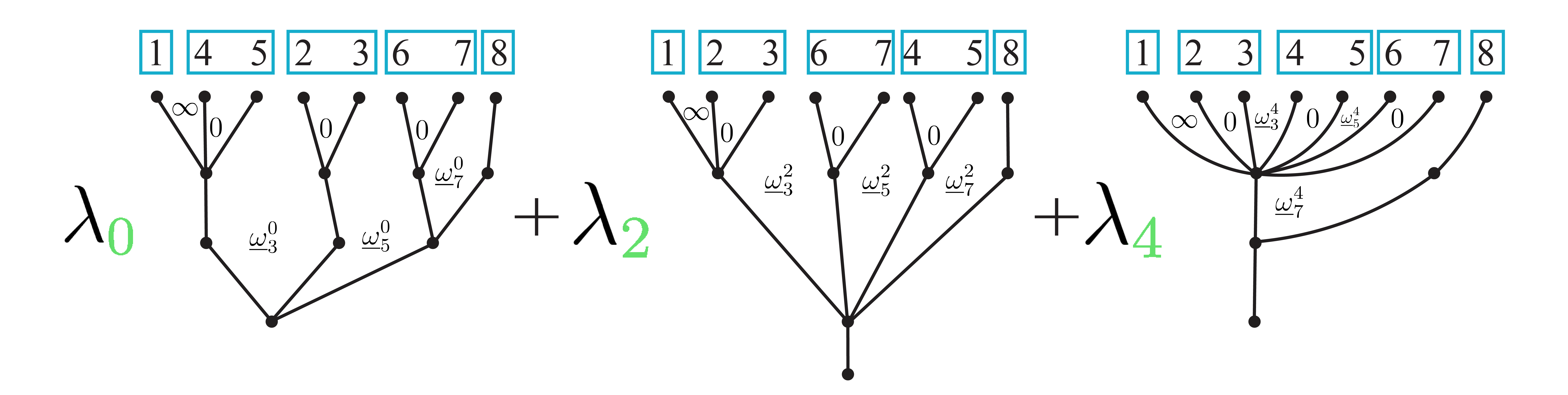}
	\caption{An element of the weighted hairy trees space}
	\label{weighted-hairy-element}
\end{figure}

\end{example}
These bricks constitute the strata of a bigger cell.
\begin{definition} Given $\bt \Gamma \in \Nerve(\FNP_m(\ell))^{nd}_d$, define
$$ \tilde{\Omega}(\bt \Gamma) = \bigcup_{\phi \Lambda_{\bullet} = D \Gamma_{\bullet} } \tilde{\Omega}(\bt \Lambda, \phi, D) \subset |\Delta^d| \times (\Rext^{\ell-1})^{d+1}  \ ,$$
with the subspace topology. Here $D$ acts on $\Gamma_{\bullet}$ via the simplicial structure on the nerve.
\end{definition}

We now quotient out by an equivalence relation that achieves three goals: 
\begin{enumerate}
\item It cancels the redundance of useless weight vectors \textit{after} the several strata have been put together in a common ambient space;
\item It contracts to a point the portion of space where different cosimplicial moves (still to be defined) of $\WHT_m$ would land, so that the space will result strictly cosimplicial.
\item In case of a trivializable $\bt \Gamma$, it identifies all the trivial strata and collapses them to a point.
\end{enumerate}

\begin{definition} Consider a triple $(\bt \Lambda, \phi, D)$ and two elements $(\bt \lambda, \bts \omega), (\bt \lambda, \bts \theta) \in \tilde{\Omega}(\bt \Lambda, \phi, D)$. If the triple is non-trivial, we say that $(\bt \lambda, \bts \omega) \sim (\bt \lambda, \bts \theta)$ if $(\vect{\omega}^k)_{\alpha} = (\vect{\theta}^k)_{\alpha}$ for all $k \in \im(D)$ and all $\alpha: \ \ \phi^{\Lambda_k}(0) \le \alpha \le \phi^{\Lambda_k}(n)$. If it is trivial, they are always equivalent. Now define
$$ \Omega(\bt \Gamma) = \tilde{\Omega}(\bt \Gamma) / \sim  \ ,$$
where the equivalence relation is generated by the union of the various equivalence relations on the $\tilde{\Omega}(\bt \Lambda, \phi, D)$.
\end{definition}

The cell $\Omega(\bt \Gamma)$ will be used later for a twisted non-degenerate realization.

\begin{remark} We want to explore the "polytopic geometry" of cells, in parallel to Remark \ref{polytope-exploration-wt} about positively weighetd trees. Since the goal of this investigation is to provide pictures and not to prove theorems, our arguments will be sloppy; details are left to the reader. 

The main difference stems from the new quotient relation outside the interval $[\phi(0), \phi(n)]$. Put it simply, if we have $k+1$ hairs with labels $1, \ldots, k+1$ at the beginning of a tree $\Gamma$, the weights $(\omega_1, \ldots, \omega_k)$ on such hairs will be subjected to a different topology than before. Let us call such a group of hairs "extremal", and denote by $C(k)$ the projection $\Omega(\Gamma)$ on these components. 

We want to sketch why $C(k) \simeq |\Delta^k|$, by induction on $k$. If $k=1$, the weight can freely vary from $0$ to $\infty$, so that $C(1) \simeq |\Delta^1|$. In the inductive step, consider $k+1$ extremal hairs. Note that we have a map 
$$C(k+1) \to C(k) \times [0, \infty] \ ,$$
$$(\omega_1, \ldots, \omega_{k+1}) \mapsto ( (\omega_1, \ldots, \omega_k), \omega_{k+1}) \ .$$
The first $k$ components land on $C(k)$ since the first $k$ hairs are extremal as well. This map is injective for $\omega_{k+1} \neq \infty$, but the equivalence relation collapse the vector to a point whenever $\omega_{k+1} = \infty$. This is exactly the definition of a cone, so that
$$ C(k+1) \simeq C(k) * |\Delta_0| \simeq |\Delta^k| * |\Delta^0| \simeq |\Delta^{k+1}| $$
by inductive hypothesis. In general, the polytope associated to $\Omega(\Gamma)$, where $\Gamma$ is a single tree, will be a product of two possibly trivial simplices (corresponding to "extremal" weights) and a cube (corresponding to "internal" weights). Let us define the following polytope:
$$ \mathcal{R}(u,v,p) \simeq  |\Delta^u| \times |\Delta^{p-v} | \times \Cube_{v-u-1} , \ \ \ \textrm{ for } 0\le u < v \le d$$ 
Referring to the definition of "extremal values" $a(\Gamma), b(\Gamma)$ given in \ref{extremal-values}, we will then have
$$ \mathcal{P}(\Omega(\Gamma)) \simeq \mathcal{R}(a(\Gamma), b(\Gamma), n) \ .$$
An example for $\Gamma = 123|4$ is depicted in figure \ref{polytope-2}. 

When examining a chain of trees $\bt \Gamma \in \Nerve(\FNP_m(n))^{nd}_d$, the situation is more intricated. We are only able to describe explicitly the geometry of a substratum $\Omega(\bt \Gamma, D)$, which is the union of all possible $\Omega(\bt \Lambda, \phi, D) \subset \Omega(\bt \Gamma)$ with a fixed $D$. The issue that emerges is the following. Since the map $\phi$ governing the stratum must be the same for all the involved trees, the extremal weights that contribute to simplices must sit on hairs in \textbf{every} tree. In other words, given $D:[r] \to [d]$, define
$$ a(\bt \Gamma, D) = \min_{k \in [r]} a(\Gamma_{D(k)}), \ \ \ b(\bt \Gamma) := \max_{k \in [r]} b(\Gamma_{D(k)})\ . $$
Then we have
$$ \mathcal{P}(\Omega(\bt \Gamma,D)) \simeq  \mathcal{R}(a(\bt \Gamma, D), b(\bt \Gamma, D), n)^{\vee(d+1)} \ .$$
The "free join" $\vee$ is defined, and its presence motivated, in  Remark \ref{polytope-exploration-wt}.
\begin{figure}
\centering
\includegraphics[width=6cm]{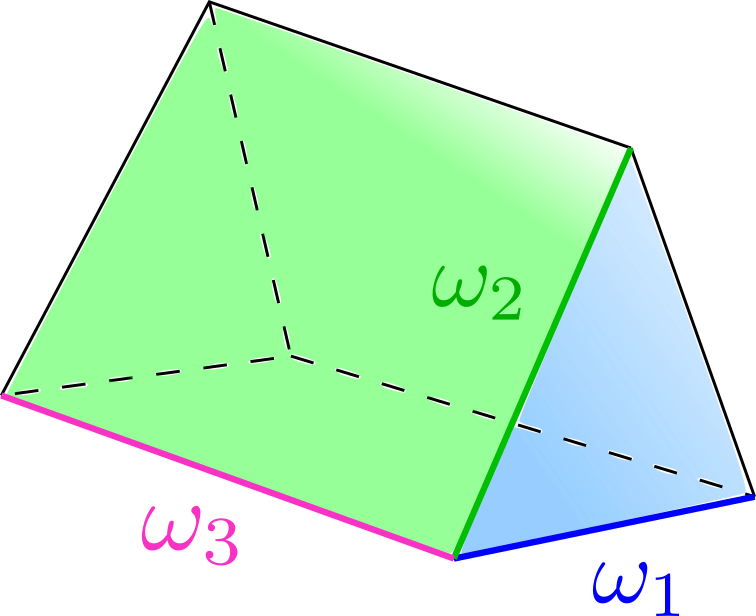}
\caption{The geometry behind $\Omega(123|4)$, a cell of $\WHT_2(4)$}
\label{polytope-2}
\end{figure}

\end{remark}
\subsubsection{Topology of cells}
The aim of this section is to better describe the topological structure of the union-and-quotient involved in the construction of $\Omega(\bt \Gamma)$. Let us give a preliminary definition
\begin{definition} Given $\bt \Gamma \in \Nerve(\FNP_m(\ell)^{nd}_d$ and $\bt \Lambda \in \Nerve(\FNP_m(n))_r^{nd}, \phi: [n] \to [\ell], D:[r] \to [d]$ such that $\phi \Lambda = D \Gamma$, define $\Omega(\bt \Lambda, \phi, D) = \tilde{\Omega}(\bt \Lambda, \phi, D)/ \sim$ as the image of $\tilde{\Omega}(\bt \Lambda, \phi, D)$ in $\Omega(\bt \Gamma)$.
\end{definition}
This is the central theorem:
\begin{theorem}[Stratification Theorem] Let $\bt \Gamma \in \Nerve(\FNP_m(\ell))^{nd}_d$.
\begin{enumerate}
\item If $\bt \Gamma$ is non-trivializable, then  $\Omega(\bt \Lambda, \phi, D) $ for $\bt \Lambda \in \Nerve(\FNP_m(n))_r^{nd}, \phi: [n] \to [\ell], D:[r] \to [d]$ such that $\phi \Lambda = D \Gamma$ are disjoint and stratify the space $\Omega(\bt \Gamma)$, with exit poset structure given by
$$ \overline{\Omega(\bt \Lambda, \phi, D) } = \coprod_{ \psi \bt \Pi = Q \bt \Lambda }  \Omega( \bt \Pi, \phi \psi, DQ)  \ .$$
\item If $\bt \Gamma$ is trivializable, all the trivial strata $\Omega(T_k, \phi, u_{\ell})$ for $k\le \ell$ and $\phi:[k] \to [\ell]$ are identified and collapsed to a point, denoted as $\Omega(T)$. Non trivial strata are pairwise disjoint and does not intersect $\Omega(T)$. The exit poset structure, for a triple $(\bt \Lambda, \phi, D)$ with $r \not \in \im D $ is equal to the non-trivializable case:
$$ \overline{\Omega(\bt \Lambda, \phi, D) } = \coprod_{ \psi \bt \Pi = Q \bt \Lambda }  \Omega( \bt \Pi, \phi \psi, DQ)  \ ,$$
while for $r \in \im D$ we have: 
$$ \overline{\Omega(\bt \Lambda, \phi, D) } = \left (\coprod_{ \substack{\psi \bt \Pi = Q \bt \Lambda \\ Q \neq u_r } }  \Omega( \bt \Pi, \phi \psi, DQ)  \right ) \sqcup \Omega(T)  \ .$$
\end{enumerate}
\end{theorem}

 Let us see a definition that means to capture the "hairy part" of a tree at the extreme left and right.

\begin{definition} \label{extremal-values} Given $\Gamma \in \FNP_m(n)$, define $E(\Gamma) = \{1, \ldots, a(\Gamma)\} \cup \{b(\Gamma), \ldots, n-1\}$, where 
$$ a(\Gamma) = \max \{ a \in \{1, \ldots, n-1\}: \ \ \forall k \le a+1\ \ \sigma^{\Gamma}(k) = k, \ \ \forall k \le a \ \ k <_{m-1} (k+1) \}  \ ,$$
$$ b(\Gamma) = \min \{ b \in \{1, \ldots, n-1\}: \ \ \forall b \le k \le n \ \ \sigma^{\Gamma}(k) = k, \ \ \forall b \le k \le n-1 \ \ k <_{m-1} (k+1) \}  \ .$$
\end{definition}

\begin{figure}[H]
	\centering
	\includegraphics[scale=0.16]{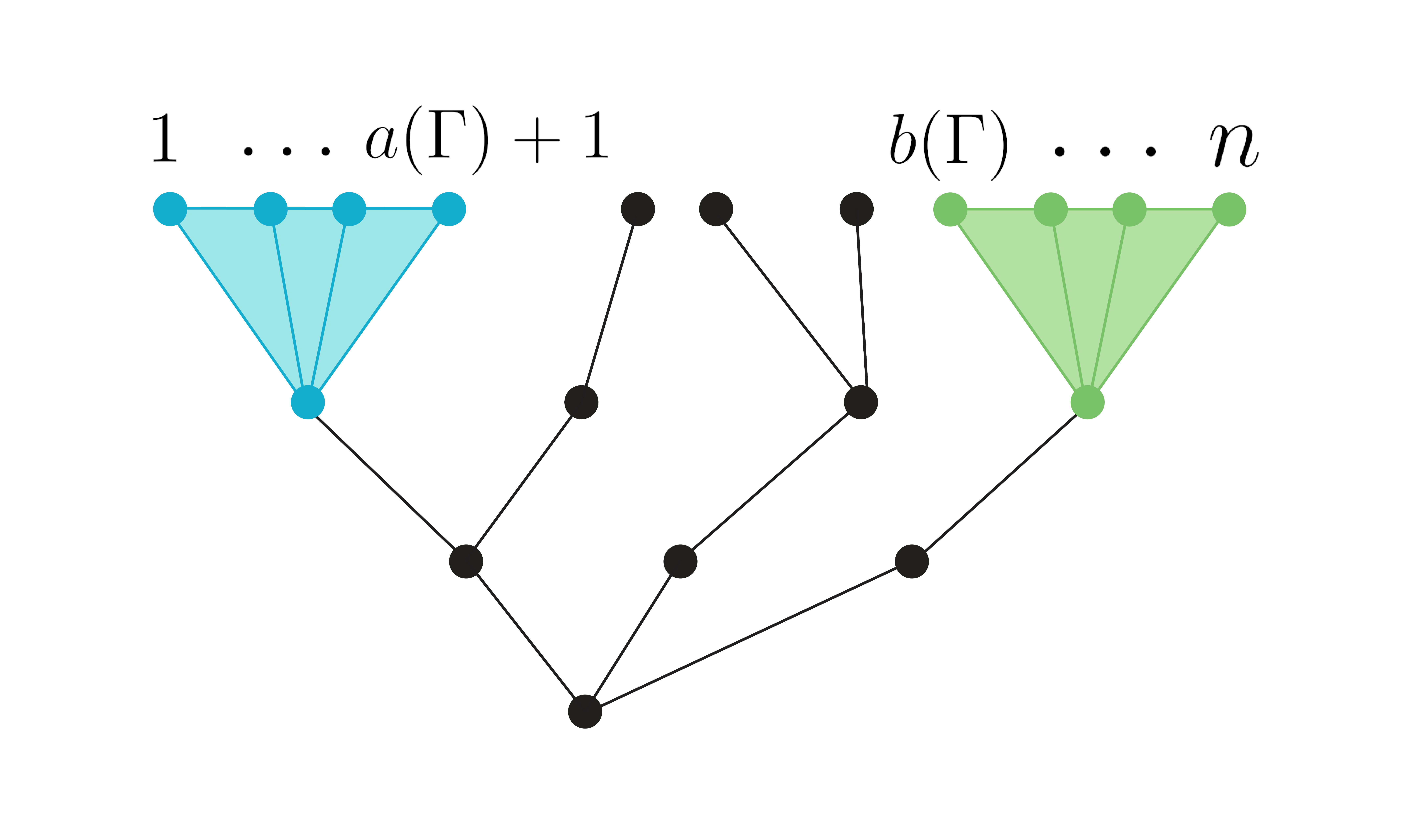}
	\caption{Extremal hair blocks}
	\label{Your label}
\end{figure}

Note that $E(\Gamma)$ is the set of indices where the weight $\infty$ can appear in $\Omega(\Gamma)$. Indeed, we can show the following
\begin{lemma} \label{infinities} Suppose $\bt \Gamma \in \Nerve(\FNP_m(n))_d, \bt \Lambda \in \Nerve(\FNP_m(p))_r, \phi:[p] \to [n], D:[r] \to [d]$ such that $\phi \Lambda = D\Gamma$. Then $\phi^{\Lambda_k}(0), \phi^{\Lambda_k}(p) \in E(\Gamma_{D(k)})$ for all $k \in [r]$. 

In particular, if $(\bt \lambda, \bt \omega) \in \Omega(\bt \Lambda, \phi, D)$ is such that $\omega^k_{\alpha}= \infty$ for some $k, \alpha$, then $\alpha \in E(\Gamma_{D(k)})$. 
\end{lemma}
\begin{proof} The first part follows from the very definition, since by the Shape-Tree Lemma \ref{twist-initial} we have 
$$\phi^{\Lambda_k}(0) \le a(\Gamma_{D(k)}) \ \ \ \ \ \phi^{\Lambda_k}(p) \ge b(\Gamma_{D(k)}) \ .$$
Regarding the second part, by definition, we can't have $\phi^{\Lambda_k}(0) < \alpha < \phi^{\Lambda_k}(p)$, where the weights are finite. Thus we have whether $\alpha \le \phi^{\Lambda_k}(0) \le a(\Gamma_{D(k)})$ or $\alpha \ge \phi^{\Lambda_k}(p) \ge b(\Gamma_{D(k)})$, showing $\alpha \in E(\Gamma_{D(k)})$.
\end{proof}
Let us also notice that weights equal to $0,\infty$ can give us hints about the stratum:
\begin{lemma}[Cell Recognition] \label{stratum-infinities} Consider $\bt \Lambda \in \Nerve(\FNP_m(p))_r, \bt \Gamma \in \Nerve( \FNP_m(n))_d, \phi: [p] \to [n], D:[r] \to [d]$ such that $\phi \bt \Lambda = D \bt \Gamma$ and an element $(\bt \lambda, \bts \omega) \in \Omega(\bt \Lambda, \phi, D)$. Suppose the stratum is non-trivial. If $a \le a(\Gamma_{D(k)})$ and $b \ge b(\Gamma_{D(k)})$ are such that
$$ \omega^{D(k)}_a = \infty, \ \ \omega^{D(k)}_b = \infty  \ ,$$
$$\omega^{D(k)}_s < \infty \ \ \forall a < s < b \ ,$$
then $\phi^{\Lambda_k}(0) = a, \phi^{\Lambda_k}(p) = b$. Furthermore:
$$ \im \phi^{\Lambda_k} = \{a,b\} \cup \{a < s < b: \omega_s^{D(k)} < \infty\} \ .$$
\end{lemma}
\begin{proof} Since $\omega^{D(k)}_{\phi^{\Lambda_k}(0)} = \infty$, we must have
$$ \phi^{\Lambda_k}(0) \le a \textrm{ or } \phi^{\Lambda_k}(0) \ge b  \ .$$
The second inequality would imply $\phi^{\Lambda_k}(0) \ge b \ge b(\Gamma_{D(k)})$. Since the stratum is non-trivial, we have $b(\Gamma_{D(k)}) > a(\Gamma_{D(k)})$, which yields a contradiction with $\phi^{\Lambda_k}(0) \le a(\Gamma_{D(k)})$. Thus we have $\phi^{\Lambda_k}(0) \le a$. Note that $\phi^{\Lambda_k}(0), \phi^{\Lambda_k}(p)$ satisfy the same properties of $a,b$. Exchanging the roles in the above argument, we deduce $\phi^{\Lambda_k}(0) = a$. Regarding $b= \phi^{\Lambda_k}(p)$, the argument is the same.

To conclude, the statement about $\im \phi^{\Lambda_k}$ is trivially true if we substitute $a$ with $\phi^{\Lambda_k}(0)$ and $b$ with $ \phi^{\Lambda_k}(p)$. 
\end{proof} 
In the same spirit, weights also give some information about the depth of the branches:
\begin{lemma} \label{special-weights}  Let $\Lambda_{\bullet} \in \Nerve(\FNP_m(n))^{nd}_r, \phi: [n] \to [\ell], D: [r] \to [d]$ and $(\bt \lambda, \bts \omega) \in \tilde{\Omega}(\bt \Lambda, \phi, D)$. If $\omega^{D(k)}_{\alpha} \in \{0, \infty\}$, then $a^{D(k)}_{\alpha} = m-1$.
\end{lemma}
\begin{proof} If $\omega_{\alpha}^{D(k)} = \infty$, then $\alpha \le \phi^{\Lambda_k}(0) = \phi(0)$ or $\alpha \ge \phi^{\Lambda_k}(n) = \phi(n)$, using Corollary \ref{extrema}. By the Shape-Tree Lemma \ref{twist-initial}, point 1 and 3, we have that in both cases $ a^{D(k)}_{\alpha} = m-1$.

If $\omega_{\alpha}^{D(k)} = 0$, then $\phi^{\Lambda_k}(0) < \alpha < \phi^{\Lambda_k}(n) $ and $\alpha \not \in \im \phi$. Because of Lemma \ref{twist-initial}, point 2, applied to $\beta = \alpha+1$, we get $a^{D(k)}_{\alpha} = m-1$.
\end{proof}
We are ready to give a proof of the Stratification Theorem.

\begin{proof} Consider $\bt \Gamma \in \Nerve(\FNP_m(\ell))_d^{nd}$. Firstly, let us assume $\bt \Gamma$ non trivializable. 

We want to show that for two triples $(\bt \Lambda, \phi, D)$ and $(\bt \Lambda', \phi, D')$ with $\bt \Lambda \in \Nerve( \FNP_m(n))^{nd}_r, \phi: [n] \to \ell, D:[r] \to [d]$ and $\bt \Lambda' \in \Nerve( \FNP_m(n'))^{nd}_{r'}, \phi': [n'] \to \ell, D':[r'] \to [d]$ ,such that $\phi \bt \Lambda = D \bt \Gamma$ and $\phi' \bt \Lambda' = D' \bt \Gamma$, we have
$$ \tilde{\Omega}(\bt \Lambda, \phi, D) \cap \tilde{\Omega}(\bt \Lambda, \phi, D) = \emptyset  \ .$$
Indeed, suppose that there exists $(\bt \lambda, \bt \omega) \in \tilde{\Omega}(\bt \Lambda, \phi, D) \cap \tilde{\Omega}(\bt \Lambda, \phi, D)$. We want to show that the two triples coincide. Since we have that
$$ \im D = \{ i: \lambda_i > 0 \} = \im D '  \ ,$$
and that $D,D'$ are strictly increasing maps, we must have $r=r'$ and $D=D'$. By the cell-recognition Lemma \ref{stratum-infinities} since the triples are not trivial, for all $k \in [r]$ we have $\phi^{\Lambda_k}(0) = (\phi')^{\Lambda_k'}(0) =: a_k$ and  $\phi^{\Lambda_k}(r) = (\phi')^{\Lambda_k'}(r') =: b_k$. Also,
$$ \im \phi^{\Lambda_k} = \{a_k, b_k \} \cup \{a_k < s < b_k : \omega^{D(k)}_s > 0 \} = \im (\phi')^{\Lambda'_k}  \ .$$
It is easy to see from the very definition that $\im \phi^{\Lambda_k } = (\sigma^{\Gamma_{D(k)} })^{-1} ( \im \phi )$, thus
$$ \im \phi = \sigma^{\Gamma_{D(k)} }  ( \im \phi^{\Lambda_k} ) =  \sigma^{\Gamma_{D(k)} }  ( \im (\phi')^{\Lambda'_k} ) = \im \phi'  \ .$$
Since they are strictly increasing maps, we must have $n=n'$ and $\phi= \phi'$. Finally, because of Lemma \ref{d-section}, it is easy to show that $d_i : \FNP_m(n) \to \FNP_m(n+1)$ is injective, because it has a retraction $s_i$. Since $\phi$ is a composition of $d_i$'s, it is injective too. Thus the equation $\phi \Lambda_k = \Gamma_{D(k)} = \phi \Lambda'_k$ implies $\Lambda_k = \Lambda'_k$ for all $k=0, \ldots,r$. This concludes the proof of disjointess.

This implies that the quotiented cells $\Omega(\bt \Lambda, \phi, D)$ are disjoint too, since the the equivalence relation is defined on the single cells. We are left with showing the exit poset structure. We start by computing the closure of the non-quotiented cell, which is easy. Substituting the strict inequalities of extended real numbers with weak ones we end up dropping the positive-and-finite constraint, getting:
$$ \textrm{cl}( \tilde{\Omega}(\bt \Lambda, \phi, D))  := \{ ( \bt \lambda, \bts{\omega} ): \bt \lambda \in D | \Delta^d|, \bts{\omega} \in (\Rext^{\ell-1} )^{d+1} : \forall k \in [r] \ \ \forall \phi^{\Lambda_k}(0) \le \alpha \le \phi^{\Lambda_k}(n)  \ ,$$
$$
(\vect{\omega}^{D(k)})_{\alpha}  \left\{ \begin{array}{ll}
        = 0,  & \textrm{if }  \alpha \not \in \im \phi^{\Lambda_k } \\
        = \infty, & \textrm{if }  \alpha = \phi^{\Lambda_k }(0) \textrm{ or } \alpha = \phi^{\Lambda_k}(n)  \\
    \end{array} \right. $$
$$
\} $$
It is a classical fact  that the closure of the quotiented cell can be computed as (temporarily denoting $f$ the quotient map):
$$ \textrm{cl}(\Omega(\bt \Lambda, \phi, D)) = f( \textrm{cl}(f^{-1}(\Omega(\bt \Lambda, \phi, D) ) ) ) = f( \textrm{cl}(\tilde{\Omega}(\bt \Lambda, \phi, D) ) ) )  \ ,$$
since $\tilde{\Omega}(\bt \Lambda, \phi, D)$ is stable for equivalence.  We then get $ \textrm{cl}(\Omega( \bt \Lambda, \phi,D)) = \textrm{cl}( \tilde{\Omega}(\bt \Lambda, \phi, D)) / \sim$. 
In order for some $\Omega(\Pi_{\bullet}, \phi', D')$ such that $\phi' \Pi_{\ell} = \Gamma_{D'(\ell)} $ to be contained in the above set, for a general $(\bt \lambda, \bts \omega) \in \Omega(\Pi_{\bullet}, \phi', D')$ the following condition must hold \footnote{They should hold for an element equivalent to $(\bt \lambda, \bts \omega)$, but since $\Omega(\Pi_{\bullet}, \phi', D')$ is stable for equivalence there is no need to make this intermediate passage in showing the containment.}
\begin{enumerate}
\item $\bt \lambda \in D| \Delta^d|$, that is some subface $\mathring{D}' |\Delta^d|$ of  $D| \Delta^d|$;
\item If $k \in [r]$ and $\alpha \in [d]$ satisfies $\phi^{\Lambda_k}(0) \le \alpha \le \phi^{\Lambda_k}(n)$ and $\alpha \not \in \im \phi^{\Lambda_k},$ then
$$ (\vect{\omega}^k)_{\alpha} = 0 \ \ \Rightarrow \ \ \alpha \not \in \im (\phi')^{\Pi_k} \textrm{ or } \alpha < (\phi')^{\Pi_k}(0) \textrm{ or } \alpha > (\phi')^{\Pi_k}(p)  \ ,$$ 
which reduces to just $\alpha \not \in \im (\phi')^{\Pi_k}$ since the other two conditions are redundant. 
\item Since the stratum is non-trivial, there are no weights equal to $\infty$ in the interval $\{s: (\phi')^{\Pi_k}(0) <s< (\phi')^{\Pi_k}(p)\} $. Thus  
$$\phi^{\Lambda_k}(0), \phi^{\Lambda_k}(n) \in \{s: s \le (\phi')^{\Pi_k}(0) \textrm{ or } s \ge (\phi')^{\Pi_k}(p) \}  \ .$$
\end{enumerate}
The first condition implies that $\im D' \subseteq \im D$; since they are both increasing maps, we must have $D' = DQ$ for some increasing $Q$. Let us see that condition $3$ implies $\phi^{\Lambda_k}(0) \le (\phi')^{\Pi_k}(0)$ and $\phi^{\Lambda_k}(p) \ge (\phi')^{\Pi_k}(p)$. Suppose by contradiction that $\phi^{\Lambda_k}(0) \ge (\phi')^{\Pi_k}(p)$. Then we would have 
$$ a(\Gamma_{D(k)} ) \ge \phi^{\Lambda_k}(0) \ge (\phi')^{\Pi_k}(p) \ge b(\Gamma_{D(k)} ) > a(\Gamma_{D(k)})  \ .$$
The last inequality comes from non-triviality of the stratum. By analogy, we also conclude $\phi^{\Lambda_k}(p) \ge (\phi')^{\Pi_k}(p)$.

Taking the complementary of condition $2$ we get (using the interval notation for natural numbers between two extrema):
$$ \im (\phi')^{\Pi_k} \subset \im \phi^{\Lambda_k} \cup [\phi^{\Lambda_k}(0), \phi^{\Lambda_k}(n)]^c  \ .$$
However, condition $3$ implies that 
$$ \im (\phi')^{\Pi_k} \subset [(\phi')^{\Pi_k}(0), (\phi')^{\Pi_k}(p)] \subset [\phi^{\Lambda_k}(0), \phi^{\Lambda_k}(n)]  \ .$$
Thus we get
$$ \im (\phi')^{\Pi_k} \subset \im \phi^{\Lambda_k}  \ .$$
Let $\Gamma_k = (\sigma_k, a^k)$. By the very definition of twisted morphism we have $\im \phi' = \sigma_k (\im (\phi')^{\Pi_k})$ and $\im \phi = \sigma_k( \im \phi^{\Lambda_k}) $. Applying the permutation we get $\im \phi' \subset \im \phi$. Since they are both strictly increasing maps, we deduce the existence of $\psi$ such that $\phi' = \phi \psi$, which concludes the proof.

We now pass to analyze the case in which $\bt \Gamma$ is trivializable. The proof of disjointness between non-trivial strata in the previous case applies also in this case. The disjointness between a non-trivial stratum and $\Omega(T_{\ell}, u_r)$ is easy, since an element $(\bt \lambda, \bts \omega) \in \Omega(\bt \Gamma)$ is in $ \Omega(T_{\ell}, u_r)$ if and only if $\bt \lambda = (0, \ldots, 0,1)$. We are left with showing the closure formula for a non-trivial stratum $\Omega(\bt \Lambda, \phi, D)$. The non-trivial strata $\Omega(\bt \Pi, \phi', D')$ contained in the closure are the same we found in the previous case, since we only used non-triviality of $\Omega(\bt \Pi, \phi', D')$. Regarding the trivial stratum, if $r \not \in \im D$, the closure does not intersect  $\Omega(T_{\ell}, u_r)$, since for all $(\bt \lambda, \bts \omega) \in \Omega(\bt \Lambda, \phi, D)$ we have $\lambda_r = 0$. If $r \in \im D$, we do have $\Omega(T_{\ell}, u_r) \subset \overline{\Omega(\bt \Lambda, \phi,D)}$, since any sequence $(\bt \lambda^{(n)}, \vect{\omega}^{\bullet, (n)} )$ with $\lambda^{(n)}_r \to 1$ will converge to the only point in $\Omega(T_{\ell}, u_r)$. We then obtain the formula stated in the theorem.
\end{proof}

\subsubsection{Realizing the space}
Finally, we have to assemble $\Omega(\bt \Gamma)$ when $\bt \Gamma$ ranges in $\Nerve(\FNP_m(n))^{nd}$. In order to do this, we turn $\Omega$ into a functor $\elt(\Nerve(\FNP_m(n))^{nd})^{op} \to \Ttop$ and use the formalism of twisted non-degenerate realization.
\begin{definition} \label{twist} For $\Gamma_{\bullet} \in \Nerve(\FNP_m(n))_d^{nd} $ and $i \in [d+1]$, define $\Omega(\partial_i) : \Omega( \partial_i \Gamma_{\bullet}) \to \Omega( \bt \Gamma) $ as $\Omega(\partial_i)(\bt \lambda, \bts \omega) = (\partial_i \bt \lambda, \partial_i \bts \omega)$, where
$$ \partial_i(\bt \lambda)_k = \left\{ \begin{array}{ll}
       0,  & \textrm{if }  k=i\\
      	\lambda_{\sigma_i(k)}, & \textrm{ otherwise} 
    \end{array} \right.  $$
$$ \partial_i(\bts \omega )^k = \left\{ \begin{array}{ll}
       \underline{u},  & \textrm{if }  k=i\\
      	\vect{\omega}^{\sigma_i(k)}, & \textrm{ otherwise} 
    \end{array} \right.  $$
Here $\underline{u}=(1, \ldots, 1)$.
\end{definition}

Note that since $\lambda_i = 0$, the value $\partial_i(\bts \omega )^i$ is not relevant, since it is quotiented out by $\sim$. 
Intuitively, the map is just the identification $|\Delta^d | \to \partial_i |\Delta^d|$, doing nothing more than adding a redundant weight vector in the other component.

\begin{lemma} \label{stratified} The maps $\Omega(\partial_i)$ are stratified, with stratum-map
$$ \Omega(\partial_i)( \Omega(\Lambda_{\bullet}, \phi, D) ) \subset \Omega( \Lambda_{\bullet}, \phi, \partial_i D)  \ .$$
They are well defined with respect to $\sim$ and turns $\Omega$ into a functor 
$$ \Omega: \elt (\Nerve(\FNP_m(n))^{nd})^{op} \to \Ttop \ .$$
\end{lemma}
\begin{proof} Let us show that $\Omega(\partial_i)$ is stratified. Consider $(\bt \lambda, \bts \omega) \in \Omega(\Lambda_{\bullet}, \phi, D) $ non trivial. We have ve to show that for all $k \in [r]$ and $\phi^{\Lambda_k}(0) \le \alpha \le \phi^{\Lambda_k}(n)$ some constraint depending on $\phi^{\Lambda_k}$ holds for $(\partial_i \bts \omega)^{\partial_iD(k)}$. Considering that $\partial_i D(k) \neq i$, we have
$$ (\partial_i \bts \omega)^{\partial_iD(k)} = (\bts \omega )^{ \sigma_i \partial_i D(k) } = (\bts \omega)^{D(k)} \ ,$$
thus the constraints are satisfied since $(\bt \lambda, \bts \omega) \in \Omega(\Lambda_{\bullet}, \phi, D)$. 

Now consider a trivial stratum $\Omega(T_n, u_d) \subset \Omega(\partial_i \bt \Gamma)$. We claim that since $\partial_i \bt \Gamma$ is trivializable, then $i \le d$. Let us explicit $\bt \Gamma = \Gamma_0 < \ldots < \Gamma_d < \Gamma_{d+1}$. If $i=d+1$, then $\partial_{d+1} \bt \Gamma = \Gamma_0 < \ldots < \Gamma_d$ would not be trivializable, because its element $\Gamma_d$ is not maximal. From this observation, we deduce $\partial_i u_d = u_{d+1}$, thus we have to show
$$ \partial_i( \Omega(T_{\ell}, u_d) ) \subset \Omega(T_{\ell}, u_{d+1}) \ .$$
Such strata are only characterized by having $\bt \lambda = (0,\ldots,0,1)$. The same inequality $i \le d$ implies the desired $\partial_i(0, \ldots, 0,1) = (0,\ldots,0,1)$.

Secondly, let us show they are well-defined with respect to the equivalence relation. For a trivial stratum, it is automatic since every two elements are equivalent. For a non-trivial stratum, if $(\bt \lambda, \bts \omega) \sim (\bt \lambda, \bts \theta)$, we have for $k \in [r]$ and $\phi^{\Lambda_k}(0) \le \alpha \le \phi^{\Lambda_k(n)} $:
$$ (\partial_i \omega^{\partial_i D(k)})_{\alpha} = (\omega^{D(k)})_{\alpha} = (\theta^{D(k)})_{\alpha} = (\partial_i \theta^{\partial_i D(k) } )_{\alpha} \ ,$$
showing $\Omega(\partial_i)(\bt \lambda, \bts \omega) \sim \Omega(\partial_i) (\bt \lambda, \bts \theta) $. 

We are left with verifying functoriality\footnote{Such verification is independent of the triviality of the stratum, since the formulas are defined uniformly.}, that is cosimpliciality of maps $\Omega(\partial_i)$. Consider $i <j$. We have to show that 
$$\Omega( \partial_j)\Omega(\partial_i)(\bt \lambda, \bts \omega)   = \Omega( \partial_i) \Omega(\partial_{j-1})  (\bt \lambda, \bts \omega)  \ .$$
On the left-hand side we have, in the first component:
$$ (\partial_j \partial_i  \bt \lambda )_k \left\{ \begin{array}{ll}
       0,  & \textrm{if }  k = j \textrm{ or } \sigma_j(k) = i \\
      	\lambda_{\sigma_i\sigma_j(k)}, & \textrm{ otherwise} 
    \end{array} \right.
      	$$
while in the second component:  
  $$ (\partial_j \partial_i \bts \omega )_k \left\{ \begin{array}{ll}
       \underline{u},  & \textrm{if }  k = j \textrm{ or } \sigma_j(k) = i \\
      	\omega^{\sigma_j\sigma_i(k)}, & \textrm{ otherwise} 
    \end{array} \right.
      	$$    	
On the right-hand-side we have, in the first component:
$$ (\partial_i  \partial_{j-1} \bt \lambda )_k \left\{ \begin{array}{ll}
       0,  & \textrm{if }  k = i \textrm{ or } \sigma_i(k) = j-1 \\
      	\lambda_{\sigma_{j-1}\sigma_i(k)}, & \textrm{ otherwise} 
    \end{array} \right.
      	$$
while in the second component:
$$ (\partial_i  \partial_{j-1}  \bts \omega )_k \left\{ \begin{array}{ll}
       \underline{u},  & \textrm{if }  k = i \textrm{ or } \sigma_i(k) = j-1 \\
      	\omega^{\sigma_{j-1}\sigma_i(k)}, & \textrm{ otherwise} 
    \end{array} \right.
      	$$

We conclude because:
\begin{itemize}
\item Degenerate cases on the left hand side happens when $k=i$ or $\sigma_j(k) = i$. Since $i<j$, the latter becomes $k=i$.
\item Degenerate cases on the right hand side happens when $k=i$ or $\sigma_i(k)= j-1$. Since $i <j $, the latter is possible if $k \ge i+1$ and $k-1 = j-1$ or $k \le i$ and $k=j-1$. But since $i< j = k+1 \le i+1$, in the last case $j=i+1$ so that $k=i$. In both cases, the only possibilities are $k=i,j$.
\item Non-degenerate cases have indices $\sigma_i \sigma_j(k) = \sigma_j \sigma_{i-1}(k)$ because of the cosimplicial identities. 
\end{itemize}
\end{proof}

We are ready to give the main definition of the section:
\begin{definition} The \textit{Weighted Hairy Trees} space is defined as the twisted non-degenerate realization
$$ \WHT_m(n) := | \Nerve(\FNP_m(n))^{nd}|_{\Omega}  \ .$$
\end{definition}

In order to provide some intuition for the difference between $\WHT_m(n)$ and the "underlying" space $\BZ_m(n) \simeq \Nerve(\FNP_m(n))$, we report here a graphical illustration of $\WHT_2(2)$. As we will see in the next chapter $\BZ_m(n)$ sits naturally inside $\WHT_m(n)$; in the picture, such subspace is highlighted. 

\begin{figure} \label{wt-sd}
\centering
\includegraphics[width=12cm]{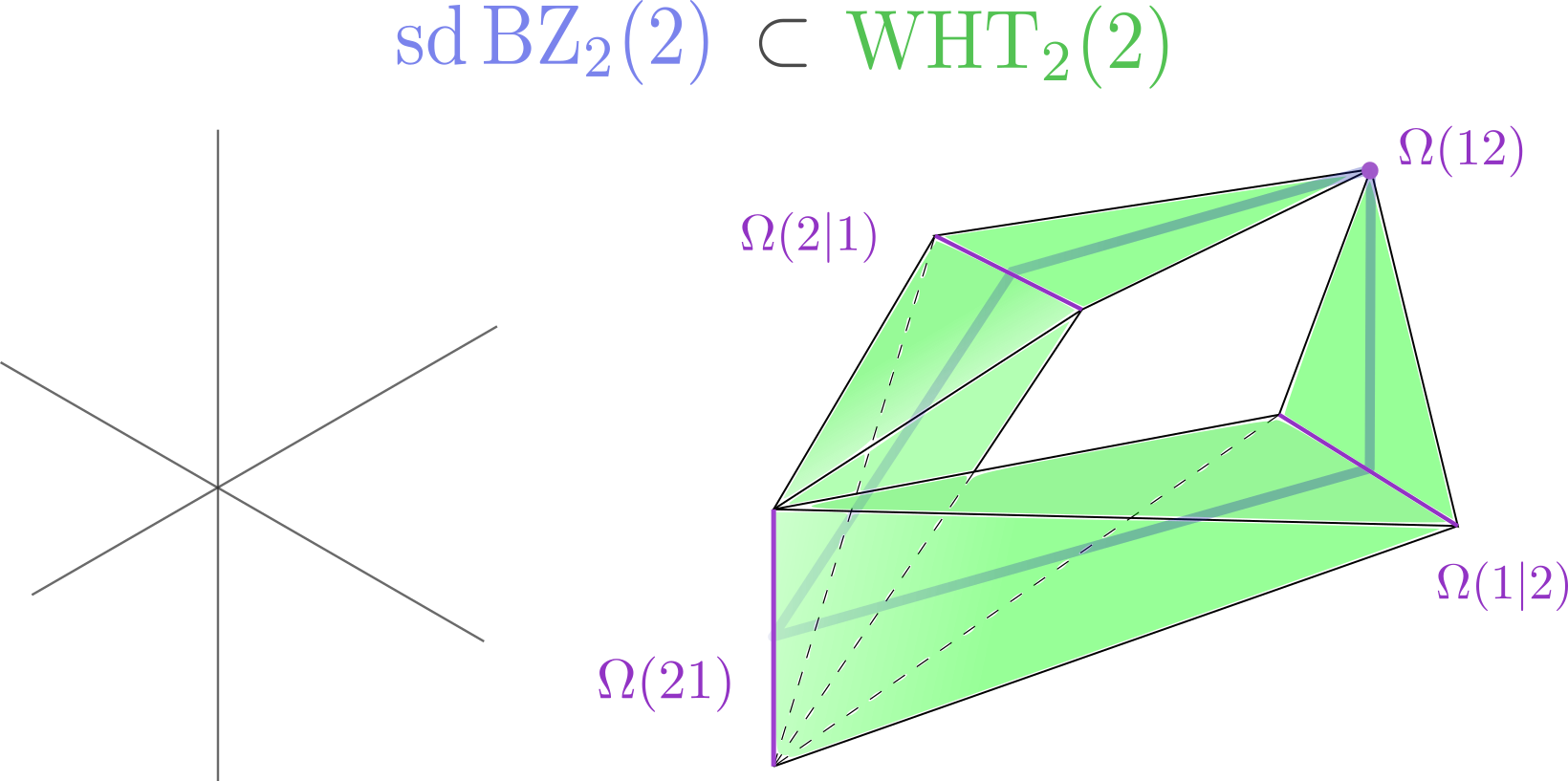}
\caption{A copy of $\BZ_2(2)$ inside $\WHT_2(2)$}
\end{figure}

\subsubsection{Semicosimpliciality}
We end this chapter by equipping $\WHT_m$ with a semicosimplicial structure.
\begin{definition} Given $\Gamma_{\bullet} \in \Nerve(\FNP_m(\ell))^{nd}$, we define $d_j: \Omega(\bt \Gamma) \to \Omega( d_j \bt \Gamma) $ as $d_j(\bt \lambda, \bts \omega) = (\bt \lambda, d_j^{\bt \Gamma}(\bts \omega) )$, with $ d_j^{\bt \Gamma}(\bts \omega)^k =  d_j^{\Gamma_k}(\vect{\omega}^k)$. The latter depends on $j$. If $0 < j < \ell+1$, then 
$$ d_j^{\Gamma_k}(\vect{\omega}^k)_{\alpha} = \left\{ \begin{array}{ll}
       0,  & \textrm{if }  \alpha = p_k\\
      	\vect{\omega}^k_{\sigma_{p_k}(\alpha)}  & \textrm{ otherwise} 
    \end{array} \right.  $$
Here $p_k = \sigma_k^{-1}(i)$. If $j=0$, then
      	$$ d_0^{\Gamma_k}(\vect{\omega}^k)_{\alpha}= d_0(\vect{\omega}^k)_{\alpha} = \left\{ \begin{array}{ll}
       \infty,  & \textrm{if }  \alpha = 1\\
      	\vect{\omega}^k_{\alpha -1 }  & \textrm{ otherwise} 
    \end{array} \right.  $$
If $j=\ell+1$, then
      	$$ d_{\ell+1}^{\Gamma_k}(\vect{\omega}^k)_{\alpha}= d_{\ell+1}(\vect{\omega}^k)_{\alpha} = \left\{ \begin{array}{ll}
       \infty,  & \textrm{if }  \alpha = \ell\\
      	\vect{\omega}^k_{\alpha}  & \textrm{ otherwise} 
    \end{array} \right.  $$
\end{definition}

We want to show that these maps are well-defined and prove that the following theorem holds:
\begin{theorem}The maps $d_k : \WHT_m(\ell) \to \WHT_m(\ell+1) $ turns $\WHT_m$ into a semicosimplicial object $\Delta_s \to \Ttop$.
\end{theorem}

The proof is not hard, despite it takes some effort to treat all the different cases. We will only report details about the main cases, leaving the others for the reader. Let us start with verifying that the maps $d_j : \WHT_m(\ell) \to \WHT_m(\ell+1)$ are well-defined.
\begin{lemma} The maps $d_j: \Omega(\bt \Gamma) \to \Omega( d_j \bt \Gamma) $ are stratified with stratum map
$$ d_j( \Omega(\bt \Lambda, \phi, D) ) \subset \Omega( \bt \Lambda, d_j \phi, D) $$
for $\phi : [n] \to [\ell], D: [r] \to [d], \bt \Lambda \in \Nerve(\FNP_m(n))^{nd}$. Furthermore, they are well-defined with respect to the equivalence relation and they satisfy $\partial$-contravariance, so that that they assemble into maps $d_j : \WHT_m(\ell) \to \WHT_m(\ell+1)$.
\end{lemma}
\begin{proof} Let us start by proving the stratification. If $\Omega(\bt \Lambda, \phi,D)$ is a trivial stratum, we have $D=u_d$ and $\phi \Lambda = T_{\ell}$. We then have to prove
$$ d_j(\Omega(T_{\ell}, u_d) ) = d_j(\Omega(\bt  \Lambda, \phi, D))  \subset  \Omega( \bt \Lambda, d_j \phi, D) = \Omega(d_jT_{\ell}, u_d) = \Omega(T_{\ell+1}, u_d) $$
that is, trivial strata are preserved. This is true since the map does not modify the $\bt \lambda$ component.

Now consider $(\bt \lambda, \bts \omega) \in \Omega(\bt \Lambda, \phi, D)$ non trivial. Let us explain the general case, that is $0<j< \ell+1$; the other two cases $j =0, j= \ell+1$ are analogous.

Recall that $d_j(\bt \lambda, \bts \omega) = (\bt \lambda, d_j^{\bt \Gamma} \bts \omega)$, where
$$ (d_j^{\Gamma_k} \omega^k)_{\alpha} = \left \{  
\begin{array}{ll}
0, & \textrm{if } \alpha = p_k \\
\omega^k_{s_{p_k(\alpha)}}, & \textrm{otherwise}
\end{array} \right.
$$
Here $p_k = \sigma_k^{-1}(j)$. For compactness of notation, we denote by $p'_k = p_{D(k)}$. The constraint on the first component is automatically verified, since $\bt \lambda$ is unchanged. Regarding the constraint on the weights, let us begin by noticing that 
$$ (d_j \phi)^{\Lambda_k} = d_j^{\phi \Lambda_k} \phi^{\Lambda_k} = d_j^{\Gamma_{D(k)} } \phi^{\Lambda_k} = d_{p'_k} \phi^{\Lambda_k}$$
We now have to verify that the new weight vector is $0,\infty$ or positive-and-finite depending on the indices. For all $k=0, \ldots, r$ we have
$$ (d_j^{\Gamma_{D(k)} } \omega^{D(k)} )_{ d_{p'_k} \phi^{\Lambda_k}(0) } \stackrel{?}{=} \infty $$
Note that the subscript is different from $p_k'$ because $p_k' \not \in \im d_{p'_k} $, thus
$$(d_j^{\Gamma_{D(k)} } \omega^{D(k)} )_{ d_{p'_k} \phi^{\Lambda_k}(0) } = \omega^{D(k)}_{ s_{p'_k} d_{p'_k} \phi^{\Lambda_k}(0)} = \omega^{D(k)}_{ \phi^{\Lambda_k}(0) } = \infty$$
Analogously
$$ (d_j^{\Gamma_{D(k)} } \omega^{D(k)} )_{ d_{p'_k} \phi^{\Lambda_k}(n) } = \omega^{D(k)}_{ \phi^{\Lambda_k}(n)} = \infty$$
Let us check the positive-and-finite weights. Consider an $\alpha$ such that $d_{p'_k} \phi^{\Lambda_k}(0) < \alpha < d_{p'_k} \phi^{\Lambda_k}(n)$ and $ \alpha \in \im d_{p'_k} \phi^{\Lambda_k}$. Applying $s_{p'_k}$ to such constraints we get $\phi^{\Lambda_k}(0)< s_{p'_k}(\alpha) < \phi^{\Lambda_k}(n)$ and $s_{p'_k} (\alpha) \in \im \phi^{\Lambda_k} $. Using that $\alpha \neq p'_k$ we get 
$$ (d_j^{\Gamma_{D(k)}} \omega^{D(k)} )_{\alpha} = \omega^{D(k)}_{ s_{p'_k}(\alpha)} \in \mathbb{R}_{> 0} $$
because of the constraints on $\omega$. We conclude with the $0$ weights, which concerns $\alpha$ such that $d_{p'_k} \phi^{\Lambda_k}(0) < \alpha < d_{p'_k} \phi^{\Lambda_k}(n)$ and $ \alpha \not \in \im d_{p'_k} \phi^{\Lambda_k}$. Applying $s_{p'_k}$ we get $\phi^{\Lambda_k}(0) < s_{p'_k}\alpha < \phi^{\Lambda_k}(n)$ as above. We want to show that $ \alpha \not \in \im d_{p'_k} \phi^{\Lambda_k}$ implies $\alpha = p'_k$ or $s_{p'_k}(\alpha) \not \in \im \phi^{\Lambda_k}$. Indeed, if $\alpha \neq p'_k$ and $s_{p_k}(\alpha) \in \im \phi^{\Lambda_k}$ we have 
$$ \alpha = d_{p'_k}s_{p'_k}(\alpha) \in \im \phi^{\Lambda_k} $$
since
$$ d_{p'_k} s_{p_k'} (x) = \left \{ \begin{array}{ll} 
x, & \textrm{if } x < p'_k \\
x+1, & \textrm{if } x = p'_k \\
x, & \textrm{if } x > p'_k 
\end{array} \right.
$$
In case $\alpha = p_k'$ we have $(d_j^{\Gamma_{D(k)}} \omega^{D(k)} )_{p'_k} = 0$ by the very definition, and if $s_{p'_k}(\alpha) \not \in \im \phi^{\Lambda_k}$ we have
$$(d_j^{\Gamma_{D(k)}} \omega^{D(k)} )_{\alpha} = \omega^{D(k)}_{s_{p'_k}(\alpha)} = 0$$
and we are done!

Now that the we have shown the map is stratified, we can verify that it respects the equivalence relation. This is automatic for trivial strata because any two elements are equivalent. Now suppose that $(\bt \lambda, \bts \omega) \sim (\lambda, \bts \theta)$ in $\tilde{\Omega}(\bt \Lambda, \phi, D)$ non trivial. Then for all $k=0, \ldots,r$ and $(d_j\phi)^{\Lambda_k}(0) \le \alpha \le (d_j \phi)^{\Lambda_k}(n)$ we have to check that
$$ (d_j^{\Gamma_{D(k)}} \omega^{D(k)} )_{\alpha} =  (d_j^{\Gamma_{D(k)}} \theta^{D(k)})_{\alpha}  $$
Let us treat the main case $0<j<\ell+1$.

Recall that $(d_j \phi)^{\Lambda_k} = d_{p'_k} \phi^{\Lambda_k}$. If $\alpha = p'_k$ we have
$$ (d_j^{\Gamma_{D(k)}} \omega^{D(k)} )_{p'_k} = 0 = (d_j^{\Gamma_{D(k)}} \theta^{D(k)})_{p'_k} $$
If $\alpha \neq p'_k$ then
$$ (d_j^{\Gamma_{D(k)}} \omega^{D(k)} )_{\alpha} = \omega^{D(k)}_{s_{p'_k}(\alpha) } = \theta^{D(k)}_{s_{p'_k}(\alpha)} =  (d_j^{\Gamma_{D(k)}} \theta^{D(k)})_{\alpha} $$
since $\phi^{\Lambda_k}(0) \le s_{p'_k}(\alpha) \le \phi^{\Lambda_k}(n) $. 

We conclude the proof of the lemma by showing that $d_j$ satisfy $\partial$-contravariance. For every $i,j$ we have to show that
\[\begin{tikzcd}
	{\Omega(\Lambda_{\bullet}, \phi, D)} & {\Omega(\Lambda_{\bullet}, d_j\phi, D)} \\
	{\Omega(\Lambda_{\bullet}, \phi, \partial_i D)} & {\Omega(\Lambda_{\bullet}, d_j\phi, \partial_i D)}
	\arrow["{\Omega(\partial_i)}"', from=1-1, to=2-1]
	\arrow["{\Omega(\partial_i)}", from=1-2, to=2-2]
	\arrow["{d_j}", from=1-1, to=1-2]
	\arrow["{d_j}"', from=2-1, to=2-2]
\end{tikzcd}\]
As usual, we report the proof for the general case of $j$. The $\rightarrow, \downarrow$ composition is
$$(\bt \lambda, \bts \omega) \mapsto (\bt \lambda, d_j^{\partial_i \bt \Gamma} \bts \omega) \mapsto (\partial_i \bt \lambda, \partial_i   d_j^{\partial_i \bt \Gamma} \bts \omega) $$
The second component, on the indices\footnote{Note that here $p_k' = p_{\partial_i D(k)}$ since the codomain has $\partial_iD$ in place of $D$. } $\partial_i D(k)$ and $d_{p'_k} \phi^{\Lambda_k}(0) \le \alpha \le d_{p'_k} \phi^{\Lambda_k}(n)$ which have to be checked for the equivalence is 
$$ (\partial_i   d_j^{\partial_i \bt \Gamma} \omega) ^{\partial_i D(k)}_{\alpha} =  (d_j^{\partial_i \bt \Gamma} \omega) ^{ D(k)}_{\alpha} =    (d_j^{ \Gamma_{\partial_i D(k)} }\omega^{D(k)} )_{\alpha} = \left \{ \begin{array}{ll} 
0, & \textrm{if } \alpha = p'_k \\
\omega^{D(k)}_{s_{p'_k}(\alpha)}, & \textrm{otherwise} \end{array} \right. 
 $$
The $\downarrow, \rightarrow$ composition is
$$(\bt \lambda, \bts \omega) \mapsto (\partial_i \bt \lambda, \partial_i \bts \omega) \mapsto (\partial_i \bt \lambda, d_j^{\bt \Gamma} \partial_i \bts \omega) $$ 
The second component, on relevant indices, is 
$$ (d_j^{\bt \Gamma} \partial_i   \omega) ^{\partial_i D(k)}_{\alpha} = 0$$
if $\alpha = p'_k$, and
$$ (d_j^{\bt \Gamma} \partial_i   \omega) ^{\partial_i D(k)}_{\alpha} = (\partial_i \omega) ^{\partial_i D(k)}_{s_{p'_k}(\alpha)} = \omega^{ D(k)}_{ s_{p'_k}(\alpha)}$$
otherwise. Thus, they coincide. 
\end{proof}

Now that we have shown the existence of actual maps $d_k$ between the various weighted hairy trees spaces, we want to show that they satisfy the cosimplicial relationships. The cosimpliciality basically holds because the shifts in positions behind are the same of many others "double-the-point-like" cosimplicial objects. However, the fact that we put $\infty$ weights and not zero weights at the extremes requires special attention, and this is the main objective of the equivalence relation we used in the definition of $\WHT$. For the sake of conciseness, we will not delve into the details of all cases
\begin{theorem}The maps $d_k : \WHT_m(\ell) \to \WHT_m(\ell+1) $ turns $\WHT_m$ into a semicosimplicial object $\Delta_s \to \Ttop$.
\end{theorem}

\begin{proof}  We have to show that $d_j d_i = d_i d_{j-1}$ for $i< j$. We restrict to a cell $\Omega(\bt \Gamma) \subset \WHT_m(n)$, where $\bt \Gamma \in \Nerve(\FNP_m(\ell))_d$.  The main cases are
$$ \boxed{0 < i < j < \ell+2}, \boxed{i > 0, j= \ell+2}, \boxed{i=0,  j < \ell+2}, \boxed{i=0, j=\ell+2} $$
We will report details for the general case only ($0 < i < j < \ell+2$). 
Let us give the following definitions of $p_k, q_k, r_k, t_k$ for $k=0, \ldots, d$ and $\sigma_k := \sigma^{\Gamma_k}$:
$$ p_k = (\sigma_k)^{-1}(i), \ \ r_k = (d_{j-1} \sigma_k)^{-1}( i), \ \ q_k = (d_i \sigma_k)^{-1} (j), \ \ t_k = \sigma_k^{-1}(j-1) $$
In general, let us see how the two maps $d_j d_i, d_i d_{j-1}$ act on a $(\bt \lambda, \bts \omega) \in \Omega(\bt \Gamma)$. Since the maps act as the identity on the first component, we will only write the second one. For simplicity, let us call
$$ \eta^k_{\alpha} := d_j d_i( \bts \omega)^k_{\alpha}, \ \ \ \epsilon^k_{\alpha} = d_i d_{j-1} ( \bts \omega)^k_{\alpha} $$
Here the computations
$$ \eta^k_{\alpha} = ( d_j^{ d_i \Gamma_k}  d_i^{ \Gamma_k} \vect{\omega}^k )_{\alpha} = \left \{ \begin{array}{ll} 
\textrm{if } \alpha = q_k: & 0 \\
\textrm{otherwise}: &  (d_i^{\Gamma_k} \omega^k )_{s_{q_k}(\alpha)} = \left \{ \begin{array}{ll} 
\textrm{if } s_{q_k}(\alpha) = p_k: & 0 \\
\textrm{otherwise}: \omega^k_{ s_{p_k} s_{q_k} \alpha} 
\end{array} \right.
 \end{array} \right. 
$$
 $$ \epsilon^k_{\alpha} = ( d_i^{ d_{j-1} \Gamma_k}  d_{j-1}^{ \Gamma_k} \vect{\omega}^k )_{\alpha} = \left \{ \begin{array}{ll} 
\textrm{if } \alpha = r_k: & 0 \\
\textrm{otherwise}: &  (d_i^{\Gamma_k} \omega^k )_{s_{r_k}(\alpha)} = \left \{ \begin{array}{ll} 
\textrm{if } s_{r_k}(\alpha) = t_k: & 0 \\
\textrm{otherwise}: \omega^k_{ s_{t_k} s_{r_k} \alpha} 
\end{array} \right.
 \end{array} \right. 
$$
In order to compare the two terms, we need relationships between $p_k, q_k, r_k, t_k$. Recall from the proof of lemma \ref{twisted-example} the formula for inverses of transformed permutations. We have
$$ r_k = (d_{j-1} \sigma_k)^{-1}(i) = \left \{ \begin{array}{ll} 
\textrm{if } i=j-1 , & i\\
\textrm{otherwise}, & d_{t_k} \sigma_k^{-1} s_{j-1} (i) = d_{t_k} \sigma_k^{-1}(i) = d_{t_k}(p_k) 
\end{array} \right.$$
$$ q_k = (d_i \sigma_k)^{-1}(j) = d_{p_k} \sigma_k^{-1} s_i (j) = d_{p_k} \sigma_k^{-1}(j-1) = d_{p_k}(t_k) $$
In order to use these relations effectively, we divide into three subcases: 
$$\boxed{p_k < t_k}, \boxed{t_k < p_k}, \boxed{p_k=t_k} $$ 
Note that the last is equivalent to $i=j-1$. Let's (re)start!
\begin{description}
\item[$\boxed{p_k < t_k}$ ]. We obtain $ r_k = p_k, q_k = t_k+1$. This means
$$ \eta^k_{\alpha} = \left \{ \begin{array}{ll} 
0, & \textrm{if} \alpha = t_k +1 \textrm{ or } \alpha = p_k \\
\omega^k_{ s_{p_k} s_{t_k +1}(\alpha) }, & \textrm{otherwise} 
\end{array} \right. $$
$$ \epsilon^k_{\alpha} = \left \{ \begin{array}{ll} 
0, & \textrm{if} \alpha = p_k \textrm{ or } \alpha = t_k +1 \\
\omega^k_{ s_{t_k} s_{p_k}(\alpha) }, & \textrm{otherwise} 
\end{array} \right. $$
We conclude because $s_{\alpha} s_{\beta} = s_{\beta-1} s_{\alpha} $ for $\beta > \alpha$. 

\item[$\boxed{p_k > t_k}$ ]. We obtain $ r_k = p_k+1, q_k = t_k$. This means
$$ \eta^k_{\alpha} = \left \{ \begin{array}{ll} 
0, & \textrm{if} \alpha = t_k \textrm{ or } \alpha = p_k+1 \\
\omega^k_{ s_{p_k} s_{t_k}(\alpha) }, & \textrm{otherwise} 
\end{array} \right. $$
$$ \epsilon^k_{\alpha} = \left \{ \begin{array}{ll} 
0, & \textrm{if} \alpha = p_k+1  \textrm{ or } \alpha = t_k \\
\omega^k_{ s_{t_k} s_{p_k+1}(\alpha) }, & \textrm{otherwise} 
\end{array} \right. $$
The conclusion follows as in the previous subcase.

\item[$\boxed{p_k = t_k}$ ]. We obtain $ r_k = p_k, q_k = p_k+1$. This means
$$ \eta^k_{\alpha} = \left \{ \begin{array}{ll} 
0, & \textrm{if} \alpha = p_k+1 \textrm{ or } \alpha = p_k \\
\omega^k_{ s_{p_k} s_{p_k+1}(\alpha) }, & \textrm{otherwise} 
\end{array} \right. $$
$$ \epsilon^k_{\alpha} = \left \{ \begin{array}{ll} 
0, & \textrm{if} \alpha = p_k \textrm{ or } \alpha = p_k+1 \\
\omega^k_{ s_{p_k} s_{p_k}(\alpha) }, & \textrm{otherwise} 
\end{array} \right. $$
The conclusion follows as in the previous subcase.
\end{description}

\end{proof}

In the next chapter, we show how the $\WHT$ space provides a bridge between the two cosimplicial objects we have examined: $\Kons_m$ and $\BZ_m$.

\section{From Kontsevich to Fox-Neuwirth} \label{kons-to-fox-l}
We want to show that there is a zig-zag of semicosimplicial maps
$$ \BZ_m \xleftarrow{\sim} \WHT_m \xrightarrow{\sim} \Kons_m \ ,$$
which are pointwise homotopy equivalences. Let us start with the leftmost map.
\subsection{Equivalence of $\WHT$ and $\BZ$}
Recall that $\BZ_m(n) \simeq |\Nerve(\FNP_m(n))|_{\textrm{std}}$. Lemma \ref{non-deg-real} implies that we can realize such simplicial set by glueing only non-degenerate cells, that is:
$$ |\Nerve(\FNP_m(n))|_{\textrm{std}} \simeq  |\Nerve(\FNP_m(n))^{nd}|_{\textrm{std}} $$
Indeed, the poset $\FNP_m(n)$ is an acyclic category, so that $\Nerve(\FNP_m(n))$ is a non-singular simplicial set (see Example \ref{poset-sing}). This means that $\BZ_m(n)$ and $\WHT_m(n)$ are realizations of the same (semi)simplicial set with different twists. Using Lemma \ref{simplicial-non-deg}, we can construct a simplicial morphism $\WHT_m(n) \to \BZ_m(n)$ by specifying a map $f_{\bt \Gamma} : \Omega(\bt \Gamma) \to \textrm{std}(\Gamma_{\bullet}) \simeq |\Delta^d| $ for all $\bt \Gamma \in \Nerve(\FNP_m(n))^{nd}_d$ that satisfy the $\partial$-contravariance condition. Let us give a definition-lemma of such maps:
\begin{lemma} \label{BZ-semi} The maps $f_{\bt \Gamma} : \Omega(\bt \Gamma) \to \textrm{std}(\bt \Gamma) $ defined as 
$$ f_{\bt \Gamma}(\bt \lambda, \bts \omega) = \bt \lambda$$ 
are well defined with respect to the equivalence relation and respect $\partial$-contravariance. Furthermore, the assembled maps  $f_m(n) : \WHT_m(n) \to \BZ_m(n)$ are semicosimplicial. 
\end{lemma}
\begin{proof} The well-definition with respect to the equivalence is obvious, since the result does not depend on the second component. Let us see the $\partial$-contravariance, that is the commutativity of the diagram
\[\begin{tikzcd}
	{\Omega(\partial_i \Gamma_{\bullet})} & {|\Delta^d|} \\
	{\Omega(\Gamma_{\bullet})} & {|\Delta^{d+1} |}
	\arrow["{f_{\partial_i \Gamma_{\bullet}}}", from=1-1, to=1-2]
	\arrow["{\Omega(\partial_i)}"', from=1-1, to=2-1]
	\arrow["{f_{\Gamma_{\bullet}}}"', from=2-1, to=2-2]
	\arrow["{\partial_i}", from=1-2, to=2-2]
\end{tikzcd}\]
It is straightforward to see that both compositions yield $\partial_i \bt \lambda$. We are left with verifying semicosimpliciality. Since $d_j \bt \Gamma$ has the same length of $\bt \Gamma$, we have $\textrm{std}(dj \bt \Gamma) \simeq | \Delta^d | \simeq \textrm{std}(\bt \Gamma)$. The map $d_j: \BZ_m(n) \to \BZ_m(n+1)$ when restricted to $\textrm{std}(\bt \Gamma) \to \textrm{std}(d_j \bt \Gamma) $ is the identity, being the realization of a simplicial map. The commutation of semicosimplicial constraint 
\[\begin{tikzcd}
	{\Omega(\Gamma_{\bullet})} & {|\Delta^d|} \\
	{\Omega(d_j\Gamma_{\bullet})} & {|\Delta^d|}
	\arrow["{f_{\Gamma_{\bullet}}}", from=1-1, to=1-2]
	\arrow["{d_j}"', from=1-1, to=2-1]
	\arrow["{f_{d_j \Gamma_{\bullet}}}"', from=2-1, to=2-2]
	\arrow["{\textrm{id}}", from=1-2, to=2-2]
\end{tikzcd}\]
is then immediate, because both the compositions are equal to $\bt \lambda$.

\end{proof}

We are now left with showing that the maps $f_m(n)$ are homotopy equivalences. We achieve this by providing explicit homotopy inverses $g_m(n)$ (which are not needed to be semicosimplicial) to $f_m(n)$. Recall the definition of extremal branches $E(\Gamma)$ of a tree $\Gamma \in \FNP_m(n)$ from \ref{extremal-values}. Then we have:
\begin{lemma} \label{BZ-he} For $\bt \Gamma \in \Nerve(\FNP_m(n))^{nd}_d$, define $g_{\bt \Gamma} :  \textrm{std}(\bt \Gamma) \to \Omega(\bt \Gamma) $ as 
$$ g_{\bt \Gamma}(\bt \lambda) = (\bt \lambda, \bts u (\bt \Gamma) ) \ ,$$
where
$$ \vect{u}(\bt \Gamma)^k_{\alpha} = \left \{ \begin{array}{ll} 
\infty, & \textrm{if } \alpha \in E(\Gamma_k) \\
1, & \textrm{otherwise}
\end{array} \right. $$
Then $g_{\bt \Gamma}$ respect $\partial$-contravariance, and the assembled maps $g_m(n) : \BZ_m(n) \to \WHT_m(n)$ are homotopy inverses to $f_m(n)$. 
\end{lemma}
An illustration of how $g_m(n)$ includes $\BZ_m(n)$ into $\WHT_m(n)$ can be found in figure \ref{wt-sd}.
\begin{proof} The diagram for $\partial$-contravariance reads 
\[\begin{tikzcd}
	{|\Delta^d|_{(\partial_i \Gamma_{\bullet})}} & {\Omega(\partial_i \Gamma_{\bullet})} \\
	{|\Delta^{d+1}|_{(\Gamma_{\bullet})}} & {\Omega(\Gamma_{\bullet})}
	\arrow[from=1-1, to=1-2]
	\arrow[from=2-1, to=2-2]
	\arrow[from=1-1, to=2-1]
	\arrow[from=1-2, to=2-2]
\end{tikzcd}\]
The $\rightarrow, \downarrow$ composition is
$$ \bt \lambda \mapsto (\bt \lambda, \bts u (\partial_i \bt \Gamma) ) \mapsto (\partial_i \bt \lambda, \partial_i \bts u (\partial_i \bt \Gamma)) \ .$$
The other composition is
$$ \bt \lambda \mapsto \partial_i \bt \lambda \mapsto (\partial_i \bt \lambda, \bts u (\bt \Gamma) )  \ .$$
But for $k \neq i$ (where $\lambda_k \neq 0$) we have
$$\partial_i \bts u ( \partial_i \bt \Gamma)^k = \bts u(\partial_i \bt \Gamma)^{s_i(k)} = \bts u( (\partial_i \bt \Gamma)_{s_i(k)}) = \bts u ( \Gamma_{d_i s_i (k)} ) )= \bts u ( \Gamma_k) \ .$$
Now let us verify that it is an homotopic inverse to $f$. It is straightforward to see that $fg = \textrm{Id}$. The other composition is only homotopic to the identity via the following homotopy:
$$ h: \Omega(\bt \Gamma) \times I \to \Omega(\bt \Gamma)  \ ,$$
$$ h_t(\bt \lambda, \bts \omega) = (\bt \lambda, H_t(\bts \omega) ) \ ,$$
where
$$ H_t(\bts \omega) = \tan( t \cdot \arctan( \bts \omega) + (1-t) \cdot \arctan( \bts u(\bt \Gamma)) )  \ .$$
Here $\tan, \arctan$ are extended continously to functions between $\Rext$ and $[0, \pi/2]$, and their action on a vector is componentwise (that is $\arctan(\bts \omega)^k_{\alpha} = \arctan(\omega^k_{\alpha})$ ). 

Let us verify that the map respects the equivalence relation. For a trivial stratum $\Omega(T_n, u_d)$ there is nothing to show, since any two elements are equivalent. Suppose $(\bt \lambda, \bts \omega) \sim (\bt \lambda, \bts \theta) \in \tilde{\Omega}(\bt \Lambda, \phi, D)$ non trivial, for some $\bt \Lambda \in \Nerve(\FNP_m(p))_r, \phi:[p] \to [n], D:[r] \to [d]$. This means $\omega^{D(k)}_{\alpha} = \theta^{D(k)}_{\alpha}$ for all $k \in [r]$ and $\phi^{\Lambda_k}(0) \le \alpha \le \phi^{\Lambda_k}(p)$. We want to show that $H_t(\bts \omega)$ is equivalent to $H_t(\bts \theta)$ for all $t$. 

Firstly, notice that for $t \neq 0$ we must have $H_t( \Omega( \bt \Lambda, \phi, D) ) \subset \Omega( \bt \Lambda', \phi', D)$ such that $(\phi')^{\Lambda'_k}(0) = \phi^{\Lambda_k}(0)$ and $(\phi')^{\Lambda'_k}(p') = \phi^{\Lambda_k}(p)$. 
Indeed, for $(\bt \mu, \bts \rho) \in \Omega( \bt \Lambda, \phi, D)$ we have for $\phi^{\Lambda_k}(0) < \alpha < \phi^{\Lambda_k}(p)$
$$ H_t(\bts \rho)^k_{\alpha} = \tan( t \cdot \arctan(\rho^k_{\alpha}) + (1-t) \cdot \arctan( \bts u(\bt \Gamma))^k_{\alpha} )) < \infty \textrm{ if and only if } $$
$$  t \cdot \arctan( \rho^k_{\alpha} ) + (1-t) \cdot \arctan( \bts u(\bt \Gamma)) < \pi/2  \ ,$$
which is true because $t \neq 0$ and $\rho^k_{\alpha} < \infty$. Also, for $\alpha = \phi^{\Lambda_k}(0), \phi^{\Lambda_k}(p)$
\begin{align*}
    H_t(\bts \rho)^k_{\alpha} & = \tan( t \cdot \arctan(\rho^k_{\alpha}) + (1-t) \cdot \arctan( \bts u(\bt \Gamma))^k_{\alpha} )) = \\
& = \tan( t \cdot \arctan(\infty) + (1-t) \cdot \arctan( \infty )) = \tan( t \pi/2 + (1-t) \pi/2) = \infty \ ,
\end{align*}
since $\alpha \in E(\Gamma_{D(k)})$ by Lemma \ref{infinities}, part (2). By Lemma \ref{stratum-infinities}, this implies $(\phi')^{\Lambda'_k}(0) = \phi^{\Lambda_k}(0)$ and $(\phi')^{\Lambda'_k}(p') = \phi^{\Lambda_k}(p)$. 

This makes the verification of the equivalence relation easy, since for all $(\phi')^{\Lambda'_k}(0) = \phi^{\Lambda_k}(0) \le \alpha \le \phi^{\Lambda_k}(p) = (\phi')^{\Lambda'_k}(p')$ we have
$$ H_t(\vect{\omega})^{D(k)}_{\alpha} = H_t(\vect{\omega}^{D(k)}_{\alpha}) =H_t(\vect{\theta}^{D(k)}_{\alpha}) = H_t(\vect{\theta})^{D(k)}_{\alpha}  \ .$$
For $t=0$, on the other hand, new weights equal $\infty$ are created, up to positions $a(\Gamma_{D(k)})$ and $b(\Gamma_{D(k)})$, so that the stratum is possibly different. However, the new stratum $\Omega(\bt \Pi, \psi, D)$ is such that 
$$\psi^{\Pi_k}(0) = a(\Gamma_{D(k)} ) \ge \phi^{\Lambda_k}(0), \psi^{\Pi_k}(q) = b(\Gamma_{D(k)}) \le \phi^{\Lambda_k}(p) \ .$$ 
This makes the verification above equally valid. Now note that the homotopy also satisfies
$$ h_0(\bt \lambda, \bts \omega) =(\bt \lambda, \bts u) = gf(\bt \lambda, \bts \omega)  \ ,$$
$$ h_1(\bt \lambda, \bts \omega) = (\bt \lambda, \bts \omega)  \ .$$
The pieces $\Omega^{(I)}(\bt \Gamma) := \Omega(\bt \Gamma) \times I$ still define a twist (trivial on the interval component) that assemble into
$$ | \Nerve(\FNP_m(n)|_{\Omega^{(I)}} = \colim_{\bt \Gamma \in \Nerve(\FNP_m(n)) } \Omega(\bt \Gamma) \times I \cong \left ( \colim_{\bt \Gamma \in \Nerve(\FNP_m(n)) } \Omega(\bt \Gamma) \right ) \times I \cong \WHT_m(n) \times I \ ,$$
since $I$ is locally compact. If we verify that $h_t$ is $\partial$-contravariant, it will then assemble into a function between the respective twisted realizations $\WHT_m(n) \times I \to \WHT_m(n)$. The above equalities $h_0 = gf, h_1 = \textrm{Id}$ will extend to the assembled maps, providing the needed homotopy. This is the diagram:
\[\begin{tikzcd}
	{\Omega(\partial_i \Gamma_{\bullet})\times I} & {\Omega(\partial_i \Gamma_{\bullet})} \\
	{\Omega(\Gamma_{\bullet}) \times I} & {\Omega(\Gamma_{\bullet})}
	\arrow[from=1-1, to=2-1]
	\arrow[from=1-1, to=1-2]
	\arrow[from=2-1, to=2-2]
	\arrow[from=1-2, to=2-2]
\end{tikzcd}\]
The $\rightarrow, \downarrow$ composition is
$$ (\bt \lambda, \bts \omega) \mapsto (\bt \lambda, H_t(\bts \omega) \mapsto (\partial_i \bt \lambda, \partial_i H_t(\bts \omega) ) \ ,$$
while
$$ (\bt \lambda, \bts \omega) \mapsto (\partial_i \bt \lambda, \partial_i \bts \omega) \mapsto (\partial_i \bt \lambda, H_t(\partial_i \bts \omega) ) \ .$$
We conclude because, for $k \neq i$:
$$ H_t (\partial_i \bts \omega)^k =  H_t( \omega^{s_i(k)}) = \partial_i H_t (\bts \omega)^k \ .$$
\end{proof}

The analogous proof for $\WT_m$ we described in Lemma \ref{BZ-he-2} has an interesting corollary, that we show for subsequent arguments:
\begin{corollary} \label{gamma-he} The map $\gamma: \WT_m(n) \to \WHT_m(n)$ is an homotopy equivalence.
\end{corollary}
\begin{proof} The following diagram is commutative: 
\[\begin{tikzcd}
	{\WT_m(n)} & {\WHT_m(n)} \\
	{\BZ_m(n)}
	\arrow["\gamma", from=1-1, to=1-2]
	\arrow["{\check{f}}"', from=1-1, to=2-1]
	\arrow["f", from=1-2, to=2-1]
\end{tikzcd}\]
Since for all $\bt \Gamma \in \Nerve(\FNP_m(n))^{nd}_d$ the diagram reduces to
\[\begin{tikzcd}
	{\check{\Omega}(\bt \Gamma)} & {\Omega(\bt \Gamma)} \\
	{|\Delta^d|}
	\arrow["\gamma", from=1-1, to=1-2]
	\arrow["{\check{f}}"', from=1-1, to=2-1]
	\arrow["f", from=1-2, to=2-1]
\end{tikzcd}\]
and both compositions are $(\bt \lambda, \bts \omega) \mapsto \bt \lambda$. By the 2-out-of-3 property of homotopy equivalences, since $f,\check{f}$ are homotopy equivalences, we deduce the thesis.
\end{proof}
\begin{remark} This is not straightforward with sections, since $\gamma \check{g} \neq g$.
\end{remark}

\subsection{The map from $\WHT$ to $\Kons$}
This is the harder part of the zig-zag, mainly because the map has to be defined on the single strata $\Omega(\Lambda_{\bullet}, \phi, D)$ and has many cases and subcases. However, almost all verifications are straightforward: we paid the price of a complicated definition for a tame behaviour. In practice, we will extend the map $\WT_m \to \BZ_m$ to zero and infinite weights. Since distances between points can be zero (or $\infty$) in this generalization, we have to move to the Kontsevich spaces. We start from a preliminary definition.

\begin{definition} Denote by $\hat{n} : \mathbb{R}^m \setminus \{0 \} \to S^{m-1} $ the normalization map
$$ \hat{n}(x) = \frac{x}{\|x \| }  \ .$$
\end{definition}
The next definition is meant to capture if some points belong to the same cloud, or if they have been put at infinity.
\begin{definition} Given $\Lambda \in \FNP_m(n)$ and $i < j$ in $ \{1, \ldots, n\}$, we say that $(i,j)$ is $\phi$-degenerate if one of three following cases occur:
\begin{itemize}
\item $i \le \phi(0)$ - a \textit{left-extreme} pair;
\item $j  > \phi(n)$ - a \textit{right-extreme} pair;
\item $\{i, \ldots, j-1\} \subset (\im \phi)^c$ - a \textit{collapsed} pair.
\end{itemize}
\end{definition}

Let us see the building blocks of the maps $\WHT_m \to \Kons_m$. As you can see, the formula for the difference of two points in the wt space will pop up. This witnesses the coherence of the two definitions, that is stated precisely in Lemma \ref{coherence-wht}.

\begin{definition} Consider $\bt \Lambda \in \Nerve(\FNP_m(n))^{nd}_r , \bt \Gamma \in \Nerve(\FNP_m(\ell))^{nd}_d$, $1 \le i < j \le \ell$ and $ \phi: [n] \to [\ell], D: [r] \to [d]$ in $\Delta_s$ such that $\phi \bt \Lambda = D \bt \Gamma$. Let $\alpha_k = \min\{i,j\}, \beta_k = \max\{i,j \}$ with respect to the order of $\Gamma_{D(k)}$, and $p_k = \sigma_{D(k)}^{-1}(\alpha_k ), q_k = \sigma_{D(k)}^{-1}(\beta_k)$. Let us explicit depths and permutations of $\Gamma_k = (\sigma_k, a^k)$. We define a map
$$ \tau^0_{ij} : \Omega(\bt \Lambda, \phi, D) \to S^{m-1} $$
to be $e_1$ if $(i,j)$ is $\phi$-degenerate or the stratum is trivial, and
$$ \tau^0_{ij}(\bt \lambda, \bts \omega) =       \hat{n} \left ( \sum_{k=0}^{r} \lambda_{D(k)} \sgn^{\Gamma_{D(k)}}_{ij} \sum_{h = p_k}^{q_k-1} \vect{\omega}^{D(k)}_{h} e_{1+a^{D(k)}_{h}} \right )  $$
otherwise.
\end{definition}

In the following lemma, we show this is a robust definition.

\begin{lemma} \label{stratum-comp} The map $\tau^0_{ij}: \Omega(\Lambda_{\bullet}, \phi, D) \to S^{m-1} $ is well defined,  respects the equivalence relation and assemble into a map $\tau^1_{ij} : \Omega(\Gamma_{\bullet}) \to S^{m-1} $.
\end{lemma}
\begin{proof}  Firstly, we have to show that the argument of $\hat{n}$ is non-zero. This mimicks the strategy of Lemma \ref{stratum-tau}. Suppose that $i <_r j$ in $\Gamma_{D(0)}$ and take $q$ to be the maximum index such that $i<_r j$ in all $\Gamma_{D(0)}, \ldots, \Gamma_{D(q)}$. We saw in Lemma \ref{stratum-tau} that 
$$\langle \sum_{k=0}^{r} \lambda_{D(k)} \sgn^{\Gamma_{D(k)}}_{ij} \sum_{h = p_k}^{q_k-1} \vect{\omega}^{D(k)}_{h} e_{1+a^{D(k)}_{h}}, e_{1+r} \rangle =   \sum_{k=0}^q \lambda_{D(k)} \sum_{h = p_k}^{q_k-1} \vect{\omega}^{D(k)}_{h}  \ .$$
Since $i,j$ is not a $\phi$-degenerate pair, for all $k$ there exists by Corollary \ref{consecutive} an $h_k^*$ with $p_k \le h_k^* \le q_k-1$ such that $h^*_k \in \im \phi^{\Lambda_k}$, which implies in particular $\vect{\omega}^{D(0)}_{h^*_0} > 0$. Together with $\lambda^{D(0)} > 0$ and non-negativity of all $\lambda$'s and $\vect{\omega}$'s, we conclude the $(1+r)$-th component is $\ge \lambda_{D(0)} \omega_{h^*_0} > 0$. 

The equivalence relation is easily satisfied. If $(i,j)$ are $\phi$-degenerate or the stratum is trivial there is nothing to show, since the function is constant. If $(i,j)$ is non-degenerate and the stratum is non trivial, we only use weights $\vect{\omega}^{\ell}_{\alpha}$ with $\ell \in \im D$ and $\phi(0) < p_k \le \alpha < q_k \le \phi(n) $, which do not depend on the equivalence representative. 

The stratum compatibility requires a detailed inspection by cases. Consider a substratum $\Omega(\bt \Pi, \phi\psi, DT) \subset \overline{\Omega( \bt \Lambda, \phi, D)} $ and take $(\bt \lambda, \bts \omega) \in \Omega( \bt \Lambda, \phi\psi, DT), \ \ (\bt \lambda^{(n)}, (\bts \omega)^{(n)}) \in \Omega(\bt \Lambda, \phi, D)$ such that
$$ \lim_{n \to \infty} (\bt \lambda^{(n)}, (\bts \omega)^{(n)}) = ( \bt \lambda, \bts \omega)  \ .$$
We have to show that
$$ \lim_{n \to \infty} \tau^0_{ij}(\bt \lambda^{(n)}, (\bts \omega)^{(n)})  = \tau^0_{ij}(\bt \lambda, \bts \omega)  \ .$$
If $\Omega(\bt Pi, \phi \psi, DT) = \Omega(T_n, u_d)$ is trivial, we have that $\bt \lambda^{(n)} \to (0,\ldots, 0,1)$, and by definition of the trivial cell $a^d_h = m-1$ for all $h$ and the permutation is the identity. Thus
$$\lim_{n \to \infty} \tau^0_{ij}(\bt \lambda^{(n)}, (\bts \omega)^{(n)}) = \hat{n} \left (  \sum_{h = i}^{j-1} \vect{\omega}^{D(r)}_{h} e_1 \right ) = e_1 \ .$$
We suppose from now on that $\Omega( \bt \Lambda, \phi\psi, DT)$ is non trivial and we start distinguishing cases depending on $(i,j)$.

If $(i,j)$ are $\phi$-degenerate, they also are $\phi\psi$-degenerate, thus both sides are equal to $e_1$. In case $(i,j)$ are not $\phi$-degenerate nor $\phi \psi$-degenerate, the proof is identical to Lemma \ref{stratum-tau}. 

If $(i,j)$ are not $\phi$-degenerate, but $\phi\psi$ degenerate, we have to show that
$$ \lim_{n \to \infty} \tau^0_{ij}(\bt \lambda^{(n)}, (\bts \omega)^{(n)}) = e_1  \ .$$
There are three sub-cases:
\begin{enumerate}
\item $(i,j)$ is a $\phi \psi$- collapsed pair;
\item $(i,j)$ is a $\phi \psi$-left-extreme pair but not collapsed;
\item $(i,j)$, is a $\phi \psi$-right-extreme pair but not collapsed.
\end{enumerate}

In the first case, by Lemma \ref{twist-initial}, we have that  $a^{D(k)}_h = m-1$ for all $p_k \le h \le q_k -1$. Note also that the same lemma ensures  
$$\sigma_k^{-1}(j) = \sigma_k(i) + (j-i) > \sigma_k^{-1}(i)  \ ,$$
that is $i$ appears before $j$ in the permutation. This implies $\sgn_{ij}^{\Gamma_{D(k)}} = 1$. Let us use this information to simplify $\tau^0_{ij}$:
\begin{align*}
\tau^0_{ij}(\bt \lambda^{(n)}, (\bts \omega)^{(n)}) &  =  \hat{n} \left ( \sum_{k=0}^{r} \lambda^{(n)}_{D(k)} \sgn^{\Gamma_{D(k)}}_{ij} \sum_{h = p_k}^{q_k-1} (\vect{\omega}^{(n)} )^{D(k)}_{h} e_{1+a^{D(k)}_{h}} \right ) \\
& = \hat{n} \left ( \sum_{k=0}^{r} \lambda^{(n)}_{D(k)} (+1) \sum_{h = p_k}^{q_k-1} (\vect{\omega}^{(n)} )^{D(k)}_{h} e_1 \right ) \ .
\end{align*}
We conclude because $\hat{n}(v) = e_1$ for all positive multiples of $e_1$.

In the second case we have to split in two subcases: $j \le \phi\psi(0)$ and $j > \phi\psi(0)$. If $j\le \phi\psi(0)$, we have that $\sigma_k^{-1}(j) = j$ because of Lemma \ref{twist-initial}; in particular $\sgn_{ij}^{\Gamma_{D(k)} } =1$. By the same lemma, $a^{D(k)}_h = m-1$ for all $i \le h \le j-1$. As in the first case, the expression simplify to a multiple of $e_1$. 

If $j > \phi\psi(0)$, we must have $q_k > \phi(0)$ and $\sgn^{\Gamma_{D(k)}}_{ij} =1$. Let us remark that in particular $p_k \le \phi\psi(0) \le q_k -1$. Define $A^k_{ij}$ as the set of indices $r: p_k \le r \le q_k-1$ for which $a^{D(k)}_r = m-1$, and $B^k_{ij} = \{p_k, \ldots, q_k-1\} \setminus A^k_{ij}$. Let us write
\begin{align*} \tau^0_{ij}(\bt \lambda^{(n)}, (\bts \omega)^{(n)}) & =   \hat{n} \left ( \sum_{k=0}^{r} \lambda_{D(k)}^{(n)} \sum_{h = p_k}^{q_k-1} (\vect{\omega}^{D(k)}_{h})^{(n)} e_{1+a^{D(k)}_{h}} \right ) \\
& =  \hat{n} \left ( \sum_{k=0}^{r} \lambda_{D(k)}^{(n)} \sum_{h \in A^k_{ij} } (\vect{\omega}^{D(k)}_{h})^{(n)} e_1 + \sum_{k=0}^{r} \lambda_{D(k)} \sum_{h \in B^k_{ij} } (\vect{\omega}^{D(k)}_{h})^{(n)} e_{1+a^{D(k)}_{h}} \right ) \ .
\end{align*}

Note that $\phi(0) \in A^k_{ij}$ because of the usual Lemma \ref{twist-initial}. Let us call 
$$ L^{(n)} = \sum_{k=0}^{r} \lambda_{D(k)}^{(n)} \sum_{h \in A^k_{ij} } (\vect{\omega}^{D(k)}_{h})^{(n)} \in \Rext, \ \ \ U^{(n)} = \sum_{k=0}^{r} \lambda_{D(k)} \sum_{h \in B^k_{ij} } (\vect{\omega}^{D(k)}_{h})^{(n)} e_{1+a^{D(k)}_{h}} \in \Rext^m \ ,$$
so that the expression rewrites  as
$$ \hat{n} \left ( L^{(n)} e_1 + U^{(n)} \right )   \ .$$
Observe that $L^{(n)} \to \infty$, since 
$$L^{(n)} \ge \left ( \sum_{k=0}^r \lambda^{(n)}_{D(k) } \right ) (\omega^{D(k)}_{\phi(0)} )^{(n)} = (\omega^{D(k)}_{\phi(0)} )^{(n)}  \ ,$$
which tends to $\omega^{D(k)}_{\phi\psi(0)} = \infty $. Thus eventually in $n$ we have $L^{(n)} > 0$. Dividing by $L^{(n)}$, since $\hat{n}$ is homogenous, we reduce to
 $$ \hat{n} \left ( e_1 + \frac{U^{(n)}}{L^{(n)}} \right )   \ .$$
Now notice that the norm of $U^{(n)}$ is bounded, since
$$ \norm{ \sum_{k=0}^{r} \lambda_{D(k)} \sum_{h \in B^k_{ij} } (\vect{\omega}^{D(k)}_{h})^{(n)} e_{1+a^{D(k)}_{h}} } \le \sum_{k=0}^{r} \lambda_{D(k)} \sum_{h \in B^k_{ij} } (\vect{\omega}^{D(k)}_{h})^{(n)} \le \sum_{h \in B^k_{ij} } (\vect{\omega}^{D(k)}_{h})^{(n)}  \ ,$$
 which converges to $\sum_{h \in B^k_{ij} } \vect{\omega}^{D(k)}_{h}$. By Lemma \ref{special-weights}, since $a^{D(k)}_h < m-1$ for $h \in B^k_{ij}$, we have that $ \vect{\omega}^{D(k)}_{h} < \infty$. This shows the boundedness of $U^{(n)}$ in $n$. Together with $L^{(n)} \to \infty$, it implies that $U^{(n)}/L^{(n)} \to 0$. Since $\hat{n}$ is continous, we get that 
 $$ \tau^0_{ij}(\bt \lambda^{(n)}, (\bts \omega)^{(n)}) =  \hat{n} \left ( e_1 + \frac{U^{(n)}}{L^{(n)}} \right ) \to \hat{n}(e_1) = e_1  \ ,$$
as desired. The third case is completely analogous to the second one.
\end{proof}

\begin{lemma} Given $1 \le i < j \le \ell$, the maps defined in Lemma \ref{stratum-comp}
$$\tau^1_{ij} : \Omega(\Gamma_{\bullet}) \to S^{m-1} $$
for $\bt \Gamma \in \Nerve(\FNP_m(\ell))_d$ satisfy the $\partial$-contravariance condition, so that they assemble into maps
$$\tau_{ij}: \WHT_m(\ell) \to S^{m-1}  \ .$$
\end{lemma}

\begin{proof}  The diagram which has to be checked is the following:
\[\begin{tikzcd}
	{\Omega(\partial_i\Gamma_{\bullet})} & {S^{m-1}} \\
	{\Omega(\Gamma_{\bullet})}
	\arrow[from=1-1, to=1-2]
	\arrow[from=1-1, to=2-1]
	\arrow[from=2-1, to=1-2]
\end{tikzcd}\]
We restrict our attention to a stratum  $\Omega(\Lambda_{\bullet}, \phi, D) \subset \Omega(\partial_i\Gamma_{\bullet})$. Because of Lemma \ref{stratified}, the left map lands in $\Omega(\Lambda_{\bullet}, \phi, \partial_iD)$. We thus reduce to the commutativity of
\[\begin{tikzcd}
	{\Omega(\Lambda_{\bullet}, \phi, D)} & {S^{m-1} } \\
	{\Omega(\Lambda_{\bullet}, \phi, \partial_iD)}
	\arrow[from=1-1, to=1-2]
	\arrow[from=1-1, to=2-1]
	\arrow[from=2-1, to=1-2]
\end{tikzcd}\]
If the strata are trivial of $(i,j)$ is $\phi$-degenerate, the result is always $e_1$. Otherwise, consider $(\bt \lambda, \bts \omega) \in \Omega(\Lambda_{\bullet}, \phi, D)$. The upper arrow is simply
$$ \sum_{k=0}^r \lambda_{D(k)} \sgn_{ij}^{(\partial_i \Gamma)_{ D(k)} }  \sum_{h=p_k}^{q_k-1} \vect{\omega}^{D(k)}_h e_{1+(\partial_i a)^{D(k)}_h} = \sum_{k=0}^r \lambda_{D(k)} \sgn_{ij}^{\Gamma_{ \partial_iD(k)} }  \sum_{h=p_k}^{q_k-1} \vect{\omega}^{D(k)}_h e_{1+a^{\partial_i D(k)}_h}   \ .$$
The other composition is
\begin{align*}
 \tau^0_{ij} \partial_i (\bt \lambda, \bts \omega) = \tau^0_{ij} \partial_i (\bt \lambda, \bts \omega) & = \sum_{k=0}^r (\partial_i \lambda)_{\partial_i D(k)} \sgn_{ij}^{ \Gamma_{ \partial_i D(k)} }  \sum_{h=p_k}^{q_k-1} (\partial_i \vect{\omega})^{\partial_i D(k)}_h e_{1+a^{\partial_i D(k)}_h} \\
& = \sum_{k=0}^r \lambda_{\sigma_i \partial_i D(k)} \sgn_{ij}^{ \Gamma_{ \partial_i D(k)} }  \sum_{h=p_k}^{q_k-1}  \vect{\omega}^{\sigma_i \partial_i D(k)}_h e_{1+a^{\partial_i D(k)}_h} \\
 & = \sum_{k=0}^r \lambda_{D(k)} \sgn_{ij}^{ \Gamma_{ \partial_i D(k)} }  \sum_{h=p_k}^{q_k-1}  \vect{\omega}^{ D(k)}_h e_{1+a^{\partial_i D(k)}_h} \ .
\end{align*}
The two results coincide and the lemma is proved.
\end{proof}

We denote by 
$$\tau=(\tau_{ij})_{i \neq j} : \WHT_m(\ell) \to (S^{m-1})^{\binom{\ell}{2}} $$
the map which is $\tau_{ij}$ for $i<j$, and $-\tau_{ji}$ for $i > j$. In order to show this actually lands in the Kontsevich space, we need to see it is an extension of the map on $\WT_m(n)$. Firsty, let us define the inclusion of $\WT_m(n)$ in $\WHT_m(n)$.
\begin{lemma} \label{closure} For $\bt \Gamma \in \Nerve(\FNP_m(\ell))^{nd}_d$, the natural map $\gamma: \check{\Omega}(\bt \Gamma) \to \Omega(\bt \Gamma) $ is well-defined and satisfy the $\partial$-contravariance conditions, so it extends to a map
$$ \gamma: \WT_m(\ell) \to \WHT_m(\ell)  \ .$$
Furthermore, the image of $\gamma$ is dense. 
\end{lemma}

\begin{proof} The map that sends $(\bt \lambda, \bts \omega)$ in $(\bt \lambda, \bts \omega)$ is stratified. Indeed, if $\bt \Gamma$ is trivializable and $D=u_d$, then $\gamma(\check{\Omega}(\bt \Gamma, u_d)) \subset \Omega(T_{\ell}, u_d)$ since $	\bt \lambda = (0, \ldots, 0,1)$ is preserved. Otherwise, we have $\gamma( \check{\Omega}(\bt \Gamma, D) ) = \Omega( D \bt \Gamma, \Id, D)$, since they have the same positive-and-finite-weights definition. It also respects the equivalence relation, since in the trivial case any two elements in the codomain are equivalent, while in the non-trivial case the equivalence relation is the same.

The contravariance conditions are automatic, since $\check{\Omega}(\partial_i)$ has the same formula of $\Omega(\partial_i)$. (see Lemma \ref{restriction}). Regarding density, notice that for all $\bt \Gamma$:
\begin{align*}
\cl \left ( \gamma \left (  \check{\Omega}(\bt \Gamma) \right ) \right ) & = \cl \left ( \gamma \left ( \bigsqcup_{D: [r] \to [d] } \check{\Omega}(\bt \Gamma, D) \right ) \right ) \\
 & = \bigcup_{D : [r] \to [d]}  \cl \left ( \gamma \left ( \check{\Omega}(\bt \Gamma, D) \right ) \right ) \\
& = \bigcup_{D : [r] \to [d]}  \cl \Omega(D \bt \Gamma, \Id, D) \\
&  \supset \bigcup_{D : [r] \to [d]}  \bigsqcup_{ \phi \bt \Lambda = D \bt \Gamma } \Omega( \bt \Lambda, \phi, D)  \\
& = \Omega(\bt \Gamma)  \ .
\end{align*}
Since $\WHT_m(\ell)$ is obtained as a colimit (aka union and quotients) of pieces $\Omega(\bt \Gamma)$ when varying $\bt \Gamma$, the thesis is proved.
\end{proof}

We are ready to show the coherence of the maps $\WHT_m(n) \to \Kons_m(n)$ and $\WT_m(n) \to \Conf_n(\mathbb{R}^m)$.
\begin{lemma} \label{coherence-wht} The following diagram
$$\begin{tikzcd}
	{\WT_m(n)} & {} & {\WHT_m(n)} \\
	{\textrm{Conf}_n(\mathbb{R}^m)} & {\textrm{Kons}_m(n)} & {(S^{m-1} )^{\binom{n}{2}}}
	\arrow["{\check{\tau}}"', from=1-1, to=2-1]
	\arrow["\gamma", from=1-1, to=1-3]
	\arrow["\phi"', from=2-1, to=2-2]
	\arrow["\subseteq"{description}, draw=none, from=2-2, to=2-3]
	\arrow["\tau", from=1-3, to=2-3]
\end{tikzcd}$$
is commutative.
\end{lemma}

\begin{proof} Firstly, recall that where $\phi = (\phi_{ij})_{i \neq j}: \textrm{Conf}_n(\mathbb{R}^m) \to \textrm{Kons}_m(n)$ is given by $\vect{x} \mapsto \hat{n} ( x_j - x_i ) $. Note also that we can reduce to $i< j$, since both $\phi_{ij}$ and $\tau_{ij}$ maps change sign when $i,j$ are swapped. Let us now restrict our attention to a stratum $\check{\Omega}(\bt \Gamma, D)$. By the previous lemma, $\gamma$ is stratified and we can reconduce to two cases. If $\bt \Gamma$ is trivializable and $D=u_d$, then the diagram becomes
$$\begin{tikzcd}
	{\check{\Omega}(\bt \Gamma, u_d)} & {} & {\Omega(T_n, u_d)} \\
	{\textrm{Conf}_n(\mathbb{R}^m)} & {\textrm{Kons}_m(n)} & {(S^{m-1} )^{\binom{n}{2}}}
	\arrow["{\check{\tau}}"', from=1-1, to=2-1]
	\arrow["\gamma", from=1-1, to=1-3]
	\arrow["\phi"', from=2-1, to=2-2]
	\arrow["\subseteq"{description}, draw=none, from=2-2, to=2-3]
	\arrow["\tau", from=1-3, to=2-3]
\end{tikzcd}$$
	The $\rightarrow, \downarrow$ composition is $e_1$ by definition. Let us compute the other way around. Consider $(\bt \lambda, \bts \omega) \in \check{\Omega}(\bt \Gamma, u_d)$. Since $\bt \lambda = (0,\ldots, 0,1)$ and $\Gamma_d = T_n$, we have
	$$\check{\tau}(\bt \lambda, \bts \omega) = x( T_n, \vect{\omega}^d)  \ .$$
The tree $T_n$ is characterized by having trivial permutation and $a_h = m-1$ for all $h=1,\ldots, n-1$. This implies in particular $\sgn^{\Gamma_d}_{ij}=1$ for $i<j$ and $e_{1+a^d_h} = e_1$ for all $h$. Using the difference formula from Lemma \ref{difference}, we have
	$$ \phi_{ij}\check{\tau}(\bt \lambda, \bts \omega) = \hat{n} \left ( \sum_{h=i}^{j-1} \vect{\omega}^d_h e_1 \right ) = e_1 \ ,$$
as desired. 

In the non-trivial case, we reduce to the following commutative diagram:
$$\begin{tikzcd}
	{\check{\Omega}(\bt \Gamma, D)} & {} & {\Omega(D \bt \Gamma, \Id, D)} \\
	{\textrm{Conf}_n(\mathbb{R}^m)} & {\textrm{Kons}_m(n)} & {(S^{m-1} )^{\binom{n}{2}}}
	\arrow["{\check{\tau}}"', from=1-1, to=2-1]
	\arrow["\gamma", from=1-1, to=1-3]
	\arrow["\phi"', from=2-1, to=2-2]
	\arrow["\subseteq"{description}, draw=none, from=2-2, to=2-3]
	\arrow["\tau", from=1-3, to=2-3]
\end{tikzcd}$$
In other words, we have to check that for points in a stratum with $\phi=\Id$ the formula for $\tau_{ij}$  coincides with $\phi_{ij} \circ \check{\tau}$. Observe that if $\phi=\Id$, no pair $(i,j)$ is degenerate, since $i > \phi(0)=0$, $j\le \phi(n) = n$ and $\phi$ is surjective. Thus for all $(\bt \lambda, \bts \omega) \in \check{\Omega}(\bt \Gamma, D)$ and $i<j$ we have that
$$ \tau_{ij}\gamma(\bt \lambda, \bts \omega) = \hat{n} \left ( \sum_{k=0}^{r} \lambda_{D(k)} \sgn^{\Gamma_{D(k)}}_{ij} \sum_{h = p_k}^{q_k-1} \vect{\omega}^{D(k)}_{h} e_{1+a^{D(k)}_{h}} \right ) \ .$$
On the other hand, by Lemma \ref{difference}, we have that
$$ \phi_{ij} \check{\tau}( \bt \lambda, \bts \omega) = \hat{n} \left ( \check{\tau}( \bt \lambda, \bts \omega)_j - \check{\tau}( \bt \lambda, \bts \omega)_ i \right ) = \hat{n} \left ( \sum_{k=0}^r \lambda_{D(k)} \sgn^{\Gamma_{D(k)}}_{ij} \sum_{h = p_k}^{q_k-1} \vect{\omega}^{D(k)}_{h} e_{1+a^{D(k)}_{h}} \right )  \ .$$
Voilà!
\end{proof}

This allows us to define the map $\tau : \WHT_m(n) \to \Kons_m(n)$.
\begin{corollary} The map $\tau : \WHT_m(n) \to ( S^{m-1} )^{\binom{n}{2} }$ corestricts to $\Kons_m(n)$.
\end{corollary}
\begin{proof} This is an easy consequence of Lemmas \ref{coherence-wht} and \ref{closure}, since 
\begin{align*}
\tau(\WHT_m(n)) & = \tau \left ( \textrm{cl}( \gamma(\WT_m(n))) \right ) \subset \textrm{cl}(\tau(\gamma(\WT_m(n)))) \\
& = \textrm{cl}(\phi(\check{\tau}(\WT_m(n)) )) \subset  \textrm{cl}(\phi(\Conf_n(\mathbb{R}^m))) \\
& = \Kons_m(n) \ .
\end{align*}
\end{proof}

Before diving into the core of the section, let us prove some preliminary lemmas. 
\begin{definition} Given $1 \le i < j \le \ell+1$ and $u \in [\ell+1]$, we say that $(i,j)$ is an $u$-exceptional pair in the following cases:
\begin{description}
\item[$\boxed{0 < u < \ell}$] $(i,j)= (u,u+1)$;
\item[$\boxed{u=0}$] $i=1$;
\item[$\boxed{u=\ell+1}$] $j=\ell+1$.
\end{description} 
\end{definition}
\begin{remark} \label{exceptional} Let us notice that the exceptions raise because the couple $(s_u(i), s_u(j))$ would not respect\footnote{With the convention that $s_{\ell+1} = \Id$. } $1 \le s_u(i) < s_u(j) \le \ell$: in the first case because it would have equal components, while in the others because one of the two components would be too big or too small. 

Also, note that if $1 \le i < j \le \ell+1 $ is $u$-exceptional we can never have $(i,j) = (u, u+1)$: in the first case is explicitly forbidden; in the other two we would have $i=0$ or $j= \ell+2$.
\end{remark}

The following is a self-contained calculation that is useful in proofs. It can be considered as a proof of $s_u$ being left adjoint of $d_u$, seen as functors between the categories $[n], [n+1]$.
\begin{lemma} \label{adjointness} For all $p \in [n], x \in [n+1]$ and $u \in [n+1]$, we have
$$ (1) \ \ \ d_u(p) \ge x \ \ \Leftrightarrow \ \ p \ge s_u(x)  \ ,$$
$$ (2) \ \ \ d_u(p) < x \ \ \Leftrightarrow \ \ p < s_u(x)  \ .$$
\end{lemma}
\begin{proof} We start by proving $(1)$. We divide in four cases to explicit the piecewise definitions of $d_u, s_u$. The two statements become
\begin{description}
\item[$\boxed{p\le u-1, \ \ \  x \ge u+1}$] $p \ge x \Leftrightarrow p \ge x-1$. Both are false since $p \le u-1 < u \le x-1$.
\item[$\boxed{p \le u-1, \ \ \  x \le u}$]$p \ge x \Leftrightarrow p \ge x$.
\item[$\boxed{p \ge u, \ \ \  x \ge u+1}$] $p+1 \ge x \Leftrightarrow p \ge x-1$
\item[$\boxed{p \ge u, \ \ \  x \le u}$] $p+1 \ge x \Leftrightarrow p \ge x$. Both are true since $p \ge u \ge x$.
\end{description} 
Now $(2)$ follows by negating both statements of the equivalence, since in a linear order $\neg (a \le b) = a > b$.
\end{proof}
We now pass to prove something that is more specific to our context, which is a direct corollary of the above simple lemma.
\begin{lemma} \label{degeneracy} Consider $1 \le i< j \le \ell+1$, an increasing $\phi: [n] \to [\ell]$ and a face map $d_u : [\ell] \to [\ell+1]$. Then $(i,j)$ is $d_u \phi$-degenerate if and only if is $u$-exceptional or $(s_u(i), s_u(j) )$ is $\phi$-degenerate.
\end{lemma}

\begin{proof}  
We firstly assume $(i,j)$ is not $u$-exceptional. Let us distinguish three cases, depeding on the type of degeneracy. We repeatedly use Lemma \ref{adjointness}.
\begin{description}
\item[$(i,j)$ collapsed:] we have to show that $\{i, \ldots, j-1\} \subset (\im d_u \phi)^c$ if and only if $\{s_u(i), \ldots, s_u(j)-1\} \subset (\im \phi)^c$. We prove the contrapositive statement. There exists $y$ such that $i \le d_u \phi(y) < j$ witnessing $\{i, \ldots, j-1\} \cap \im d_u \phi \neq \emptyset$ if and only if $s_u(i) \le \phi(y) < s_u(j)$, witnessing $\{s_u(i), \ldots, s_u(j)-1\} \cap \im \phi \neq \emptyset$.
\item[$(i,j)$ left-extreme:] we have $i \le d_u\phi(0)$ if and only if $s_u(i) \le \phi(0)$.
\item[$(i,j)$ right-extreme:] we have $j > d_u\phi(n)$ if and only if $s_u(j) < \phi(n)$.
\end{description}
Secondly, let us examine the three $u$-exceptional cases and prove they yield $d_u\phi$-degenerate pairs.
\begin{description}
\item{$\boxed{0 < u < \ell+1, \ \ \ (i,j) = (u,u+1)}$} A collapsed pair, since $\{u\} \subset (\im d_u \phi)^c$
\item{$\boxed{u=0, \ \ \ i=1}$} A left-extreme pair, since $i=1 \le d_0\phi(0)$
\item{$\boxed{u=\ell+1, \ \ \ j=\ell+1}$} A right extreme pair, since $j = \ell+1 > \phi(n) = d_{\ell+1}\phi(n)$.
\end{description}
\end{proof}

The second lemma concerns the sign tensor.
\begin{lemma} \label{sign} Consider $\Gamma \in \FNP_m(\ell)$. For all $1 \le i < j \le \ell+1$ and $d_u:[\ell]\to \ell+1$, we have 
$$ \sgn_{ij}^{d_u \Gamma} =  \left \{ \begin{array}{ll} 
1, & \textrm{if } (i,j) \textrm{ is } u\textrm{-exceptional}\\
\sgn_{s_u(i)s_u(j)}^{\Gamma}, & \textrm{otherwise}
\end{array} \right. $$
\end{lemma}
\begin{proof} In the $u$-exceptional cases, this follows from the very definition. Otherwise, let us suppose that $\sgn_{ij}^{d_u \Gamma} = 1$. Denote by $\sigma$ the permutation associated to $\Gamma$. We divide in three cases:
\begin{description}
\item[$\boxed{i,j \neq u}$] Denote by $\alpha:= \sigma^{-1}(u)$. The sign assumption means
$$ (d_u \sigma)^{-1}(i) < (d_u \sigma)^{-1}(j) \ \ \Rightarrow \ \ d_{\alpha} \sigma^{-1} s_u(i) < d_{\alpha} \sigma^{-1} s_u(j)  \ .$$
Applying $s_{\alpha}$ we get $\sigma^{-1} s_u(i) \le \sigma^{-1} s_u(j) $, but they can't be equal, otherwise the above inequality would be an equality. We deduce that $\sgn_{s_u(i) s_u(j)}^{\Gamma} = 1$.
\item[$\boxed{i=u}$] In this case
$$  (d_u \sigma)^{-1}(u) = \alpha < (d_u \sigma)^{-1}(j) = d_{\alpha} \sigma^{-1} s_u (j) \ \ \Rightarrow \ \ \alpha \le \sigma^{-1}s_u(j)  \ .$$
Since $(i,j)$ is not $u$-exceptional, we have $j \neq u+1$ (see Remark \ref{exceptional}). If the inequality was not strict, we would have $s_u(j) = u$, which is impossible since $j \neq u+1, u$. We conclude that
$$ \sigma^{-1} s_u(i) = \sigma^{-1}(u) = \alpha < \sigma^{-1} s_u(j)  \ ,$$
giving $\sgn^{\Gamma}_{s_u(i) s_u(j) } = 1$.
\item[$\boxed{j=u}$] In this case
$$  (d_u \sigma)^{-1}(i) = d_{\alpha} \sigma^{-1} s_u (i)   < \alpha = (d_u \sigma)^{-1}(u) \ \ \Rightarrow \ \  \sigma^{-1}s_u(i) \le  \alpha  \ .$$
Since $i < j = u$, we can't have an equality, otherwise $s_u(i) = \alpha \Rightarrow i \in \{u, u+1\}$.  We conclude that
$$ \sigma^{-1} s_u(i) < \alpha = \sigma^{-1} s_u(j)  \ ,$$
giving $\sgn^{\Gamma}_{s_u(i) s_u(j) } = 1$.
\end{description}

The above argument shows $\sgn^{d_u \Gamma}_{ij}= \sgn^{\Gamma}_{s_u(i) s_u(j)}$ for non exceptional pairs in case the LHS is $1$. If it is $-1$, we use antisymmetricity to deduce
$$\sgn^{d_u \Gamma}_{ij}= -\sgn^{d_u \Gamma}_{ji}= -\sgn^{\Gamma}_{s_u(j) s_u(i)} = \sgn^{\Gamma}_{s_u(i) s_u(j)} \ ,$$
and the lemma is proved.
\end{proof}
Finally, let us recall the semicosimplicial structure on $\Kons_m$ in terms of our definitions.
\begin{recall*} Given $u \in [n+1]$, define the maps $d_u : \Kons_m(n) \to \Kons_m(n+1)$ as
$$k < \ell: \ \ \ d_u(x_{ij})_{k \ell} =  \left \{ \begin{array}{ll} 
e_1, & \textrm{if } (k, \ell) \textrm{ is } u\textrm{-exceptional} \\
x_{s_u(i)s_u(j)}, & \textrm{otherwise}
\end{array} \right. $$
$$ \ell < k: \ \ \ d_u(x_{ij})_{k \ell} = - d_u(x_{ij})_{\ell k} \hspace{4.5cm}$$
\end{recall*}
In the article \cite{sinha2004operads}, the cosimplicial structure is described as the Mc-Clure and Smith construction applied to $\Kons_m$. The operadic structure on the latter and the associated multiplication are described in Theorem 4.5 and Proposition 4.7; the streamlined version of such construction we provided in the background section (\ref{kons-cosimp}) can be used to check that our definition agrees with the classical one.

We are ready to prove the two main theorems.
\begin{theorem} \label{WHT-semi}
The map $ \tau : \WHT_m(n) \to \Kons_m(n) $ is semicosimplicial with respect to the semicosimplicial structures on $\WHT_m$ and $ \Kons_m$.
\end{theorem}

\begin{proof} For all $u \in [\ell+1]$, we must show that the following diagram is commutative:
\[\begin{tikzcd}
	{\WHT_m(\ell)} & {\Kons_m(\ell)} \\
	{\WHT_m(\ell+1)} & {\Kons_m(\ell+1)}
	\arrow["\tau", from=1-1, to=1-2]
	\arrow["{d_u}"', from=1-1, to=2-1]
	\arrow["\tau"', from=2-1, to=2-2]
	\arrow["{d_u}", from=1-2, to=2-2]
\end{tikzcd}\]
We restrict to a single stratum $\Omega(\bt \Lambda, \phi, D) \subset \WHT_m(n)$ for some $D: [r] \to [d], \phi: [n] \to [\ell], \bt \Lambda \in \Nerve(\FNP_m(n))_r$ and $\bt \Gamma \in \Nerve(\FNP_m(\ell))_d$ such that $\phi* \bt \Lambda = D \bt \Gamma$. Also, we examine the $(i,j)$-th component of the object where maps land, for $i< j$ (since in the Kontsevich space tensors are antisymmetric). If the stratum is trivial, we must show commutativity of the following diagram:
\[\begin{tikzcd}
	{\Omega(T_{\ell}, u_d) } & {\Kons_m(\ell)} \\
	{\Omega(T_{\ell+1}, u_d) } & {S^{m-1} }
	\arrow["\tau", from=1-1, to=1-2]
	\arrow["{d_u}"', from=1-1, to=2-1]
	\arrow["\tau_{ij}"', from=2-1, to=2-2]
	\arrow["({d_u})_{ij}", from=1-2, to=2-2]
\end{tikzcd}\]
The $\downarrow, \rightarrow$ composition yields $e_1$. In the other way around, we must compute $d_u( \vect{v})_{ij}$ where $\vect{v}_{\alpha \beta} = e_1$ for all $\alpha < \beta$; it is straightforward to see the latter is $e_1$ too.

In the non-trivial case, we reduce to the following diagram:
\[\begin{tikzcd}
	{\Omega(\bt \Lambda, \phi, D) } & {\Kons_m(\ell)} \\
	{\Omega(\bt \Lambda, d_u \phi, D) } & {S^{m-1} }
	\arrow["\tau", from=1-1, to=1-2]
	\arrow["{d_u}"', from=1-1, to=2-1]
	\arrow["\tau_{ij}"', from=2-1, to=2-2]
	\arrow["({d_u})_{ij}", from=1-2, to=2-2]
\end{tikzcd}\]
Consider $(\bt \lambda, \bts \omega) \in \Omega(\bt \Lambda, \phi, D)$. The equation we have to check is
$$ \tau_{ij}d_u( \bt \lambda, \bts \omega) = (d_u \tau(\bt \lambda, \bts \omega))_{ij}  \ .$$
Firstly, let us examine the case in which $(i,j)$ is $d_u\phi$-degenerate. The LHS is $e_1$ by the very definition of $\tau$. Regarding the RHS, Lemma \ref{degeneracy} shows that there are two subcases: $(i,j)$ is $u$-exceptional; $(s_u(i), s_u(j)$ is $\phi$-degenerate. In the first subcase, the RHS is $e_1$ because of the definition of $d_u$; in the second subcase, we conclude because
$$ (d_u \tau(\bt \lambda, \bts \omega))_{ij} = (\tau(\bt \lambda, \bts \omega))_{s_u(i) s_u(j)} = e_1  \ .$$
We suppose from now on that $(i,j)$ is not $d_u \phi$-degenerate, which implies by the same lemma that $(i,j)$ is not $u$-exceptional and $(s_u(i), s_u(j))$ is not $\phi$-degenerate. 

Let 
\begin{align*}
\alpha_k & = \sigma_k^{-1}(s_u(i)) \\
\beta_k & = \sigma_k^{-1}(s_u(j)) \\
r_k & = \sigma_k^{-1}(u) \\
p_k & = (d_u \sigma_k)^{-1}(i) = d_{r_k} \sigma_k^{-1}(s_u(i)) \\
q_k & = (d_u \sigma_k)^{-1}(j) = d_{r_k} \sigma_k^{-1}(s_u(j)) . \ 
\end{align*}
The right-hand side becomes
$$\tau(\bt \lambda, \bts \omega)_{s_u(i) s_u(j)} =  \hat{n} \left ( \sum_{k=0}^r \lambda_{D(k)} \sgn_{s_u(i) s_u(j)}^{\Gamma_k} \sum_{h=\alpha_k}^{\beta_k-1} \vect{\omega}^{D(k)}_h e_{1+a^{D(k)}_h} \right ) \ ,$$
so there is no need to distinguish cases to compute it. The left-hand side, on the other hand, requires a more detailed inspection. Since $d_u$ appears, we have to examine separately the cases $u=0, u=\ell+1$ and $0 < u < \ell+1$. The latter must be further divided into five subcases. The beginning of the computation is common, using Lemma \ref{sign} about the sign:
\begin{align*}
& \hat{n} \left ( \sum_{k=0}^r \lambda_{D(k)} \sgn_{ij}^{d_u \Gamma_k }\sum_{h = p_k}^{q_k-1} (d_u\vect{\omega} )^{D(k)}_h e_{1+(d_ua)^{D(k)}_h } \right )  \\
= & \hat{n} \left ( \sum_{k=0}^r \lambda_{D(k)} \sgn_{s_u(i)s_u(j)}^{\Gamma_k }\sum_{h = p_k}^{q_k-1} (d_u\vect{\omega} )^{D(k)}_h e_{1+(d_ua)^{D(k)}_h } \right ) \ .
\end{align*}
We only report the details for the general case, that is $0<u<\ell+1, \ \ p_k < r_k < q_k$ ; the others are completely analogous and left to the reader. In this case, the permutation indices become
\begin{align*}
p_k = d_{r_k} \sigma_k^{-1}(s_u(i)) \ \ \Rightarrow  \ \ p_k & = s_{r_k}(p_k) = \sigma_k^{-1}(s_u(i)) = \alpha_k \\
q_k = d_{r_k} \sigma_k^{-1}(s_u(j)) \ \ \Rightarrow  \ \ q_k -1 & = s_{r_k}(q_k) = \sigma_k^{-1}(s_u(j)) = \beta_k \ .
\end{align*}
Regarding weights and depths, this depends on where $h$ falls. For $p_k \le h < r_k$ we have 
\begin{align*}
(d_u\vect{\omega})^{D(k)}_h & = \vect{\omega}^{D(k)}_{s_{r_k}(h)} = \vect{\omega}^{D(k)}_h \\
(d_ua)^{D(k)}_h & = a^{D(k)}_{s_{r_k}(h)} = a^{D(k)}_h \ .
\end{align*}
For $h=r_k$ we have
\begin{align*}
(d_u\vect{\omega})^{D(k)}_{r_k} & = 0 \\
(d_ua)^{D(k)}_{r_k} & = m-1 \ .
\end{align*}
Finally, for $r_k < h < q_k$ we have
\begin{align*}
(d_u\vect{\omega})^{D(k)}_h & = \vect{\omega}^{D(k)}_{s_{r_k}(h)} = \vect{\omega}^{D(k)}_{h-1} \\
(d_ua)^{D(k)}_h & = a^{D(k)}_{s_{r_k}(h)} = a^{D(k)}_{h-1} \ .
\end{align*}
Plugging this into the left-hand side:
\begin{align*}
& \hat{n} \left ( \sum_{k=0}^r \lambda_{D(k)} \sgn_{s_u(i)s_u(j)}^{\Gamma_k }\sum_{h = p_k}^{q_k-1} (d_u\vect{\omega} )^{D(k)}_h e_{1+(d_ua)^{D(k)}_h }  \right )  \\
= &  \hat{n} \left ( \sum_{k=0}^r \lambda_{D(k)} \sgn_{s_u(i)s_u(j)}^{\Gamma_k } \left ( \sum_{h = \alpha_k}^{r_k-1} (\vect{\omega} )^{D(k)}_h e_{1+a^{D(k)}_h} + (\vect{\omega} )^{D(k)}_{r_k} e_{1+a^{D(k)}_{r_k}} + \sum_{h = r_k+1}^{q_k-1} (\vect{\omega} )^{D(k)}_{h-1} e_{1+a^{D(k)}_{h-1}} \right )  \right )   \\
= &  \hat{n} \left ( \sum_{k=0}^r \lambda_{D(k)} \sgn_{s_u(i)s_u(j)}^{\Gamma_k }\sum_{h = \alpha_k}^{\beta_k-1} (\vect{\omega} )^{D(k)}_h e_{1+a^{D(k)}_h} \right )   \ ,
\end{align*}
since the $r_k$-th term is zero. In the second summation, we substituted $h=h-1$. This concludes the proof.
\end{proof}
\begin{theorem} \label{WHT-he}The map $ \tau : \WHT_m(n) \to \Kons_m(n) $ is an homotopy equivalence.
\end{theorem}

\begin{proof} This is a matter of diagram-chasing, using the $2/3$ property of homotopy equivalences. Consider the following diagram:
\[\begin{tikzcd}
	& {\textrm{WT}_m(n)} & {\textrm{WHT}_m(n)} \\
	{\textrm{BZ}_m(n) } & {\textrm{Conf}_n(\mathbb{R}^m)} & {\textrm{Kons}_m(n)}
	\arrow["{\check{g}}", from=2-1, to=1-2]
	\arrow["v"', from=2-1, to=2-2]
	\arrow["{\check{\tau}}", from=1-2, to=2-2]
	\arrow["\tau", from=1-3, to=2-3]
	\arrow["\gamma", from=1-2, to=1-3]
	\arrow["\phi", from=2-2, to=2-3]
\end{tikzcd}\]
The desired result follows from this chain of observations:
\begin{itemize}
\item The map $\check{g}$ is an homotopy equivalence, as proved in Lemma \ref{BZ-he-2};
\item The map $v$ is an homotopy equivalence, as proved in \cite{blagojevic2013convex};
\item The map $\gamma$ is an homotopy equivalence, as proved in Corollary \ref{gamma-he};
\item The lefthand diagram commutes. Indeed, consider a chain $\bt \Gamma \in \Nerve(\FNP_m(n))_d$ and $\bt \lambda \in \textrm{std}(\bt \Gamma)$. We have that $\check{g}(\bt \lambda) = \bts u $ (all components equal to $1$), thus $\check{\tau}\check{g}(\bt \lambda)$ is defined as 
$$ \sum_{k=0}^d \lambda_k x \left ( \Gamma_k, \vect{u} \right )  \ ,$$
with the inductive definition
$$ x \left ( \Gamma_k, \vect{u} \right )_{\sigma_k(1) } = 0, \ \ \ x \left (\Gamma_k, \vect{u} \right )_{\sigma_k(p+1) } = x \left (\Gamma_k, \vect{u} \right )_{\sigma_k(p) } + u_p e_{1+a_p^k} = x \left (\Gamma_k, \vect{u} \right )_{\sigma_k(p) } + e_{1+a_p^k} \ .$$
This is exactly the definition of $v(\bt \Gamma)$.
\item By $2/3$, we get that $\check{\tau}$ is an homotopy equivalence.
\item Since $\phi$ is an homotopy equivalence, by $2/3$ applied to $\phi \check{\tau}$ and $\gamma$, we get that $\tau$ is an homotopy equivalence.
\end{itemize}
\end{proof}

\subsection{Relating spectral sequences}
Let us draw the conclusions of the article. Theorems \ref{WHT-semi}, \ref{WHT-he}, \ref{BZ-semi}, \ref{BZ-he} combine into the following
\begin{theorem} \label{zzthm} For all $m\ge 2$, there is a zig-zag of semicosimplicial homotopy equivalences:
\[\begin{tikzcd}
	& {\textrm{WHT}_m} \\
	{\BZ_m} && {\textrm{Kons}_m}
	\arrow["\simeq"', from=1-2, to=2-1]
	\arrow["\simeq", from=1-2, to=2-3]
\end{tikzcd}\]
\end{theorem}
We are ready to prove the main result of the article:
\begin{theorem} \label{ssthm} For all $m \ge 2$ and abelian groups $A$, the Barycentric Fox-Neuwirth and Sinha Spectral Sequence in (co)homology are isomorphic from the first page on.
In particular, $E_{pq}^r(\FNC_m,A)$ (resp. $E^{pq}_r(\FNC_m,A)$) converges to the homology (resp. cohomology) of $E_m$ with coefficients in $A$ for $m \ge 4$. 
\end{theorem}

\begin{proof} The first part of the theorem is a consequence of the semicosimplicial zig-zag above. Indeed, passing to (co)chains with coefficients in an abelian group $A$, we have a semicosimplicial zig zag that is degreewise a quasi-iso:
\[\begin{tikzcd}
	{\FNC_m(n)_* \otimes A } && {C_*(\WHT_m(n), A)} \\
	& {C_*(\BZ_m(n), A)} && {C_*(\Kons_m(n), A)}
	\arrow["\simeq"', from=1-3, to=2-2]
	\arrow["\simeq", from=1-3, to=2-4]
	\arrow["\simeq", from=1-1, to=2-2]
\end{tikzcd}\]
Taking the associated bicomplexes, we have a zig-zag that is a quasi-iso on columns:
\[\begin{tikzcd}
	{\BFNA_{\bullet *}} && {\mathcal{B}(C_*(\WHT_m(\bullet),A))} \\
	& {\mathcal{B}(C_*(\BZ_m(\bullet),A))} && {\mathcal{B}(C_*(\Kons_m(\bullet),A))}
	\arrow["\simeq"', from=1-3, to=2-2]
	\arrow["\simeq", from=1-3, to=2-4]
	\arrow["\simeq", from=1-1, to=2-2]
\end{tikzcd}\]
We are interested in the spectral sequences built from these bicomplexes. Since in the first step of the spectral sequence we take the homology of columns, such quasi-isomorphisms become isomorphisms from the first page on.
\
Regarding the second part, the convergence to the (co)homology of $T_{\infty} E_m \simeq E_m$ in the case $m \ge 4$ is proved in \cite{sinha2004operads}, Theorem 7.2.
\end{proof}

\printbibliography

\end{document}